\documentclass[a4paper,english,11pt]{article}

\usepackage[dvipsnames]{xcolor}
\usepackage[active]{srcltx}
\usepackage{amsmath,amsfonts,amssymb,amsthm,bbm,graphics,epsfig,psfrag,epstopdf}
\usepackage{paralist,subfigure,float}
\usepackage{color,soul}
\usepackage{array}
\usepackage{geometry}
\usepackage[T1]{fontenc}
\usepackage{mathrsfs}
\usepackage[utf8]{inputenc}
\usepackage{graphicx}
\usepackage{subfigure}
\usepackage{xcolor}
\usepackage{amsmath,amsfonts,amsthm,amssymb}
\usepackage{tikz}
\usepackage{fontawesome5}
\usepackage{pgfplots}
\usepackage{ulem}
\usepackage{array} % Improves `tabular` and `array` environments
\usepackage{pict2e} % Allows \linethickness{...} in diagonal lines
\usepackage{slashbox} % Defines \backslashbox{..}{..}
\usepackage{float}
\usepackage{enumitem}
\usepackage{fontawesome5}
\usepackage{hyperref}
\newcommand{\opt}{\mathbf{O}}
\renewcommand{\tilde}{\widetilde}

\newcommand{\A}{\mathcal{A}}

\newcommand{\E}{\mathbf{E}}

\newcommand{\M}{\mathbf{M}}
\newcommand{\N}{\mathbb{N}}
\newcommand{\p}{\partial}

\newcommand{\R}{\mathbb{R}}

\newcommand{\x}{\mathbf{x}}
\newcommand{\y}{\mathbf y}
\newcommand{\ug}{\mathbf u}
\newcommand{\Ug}{\mathbf U}
\newcommand{\vg}{\mathbf v}
\newcommand{\Vg}{\mathbf V}
\newcommand{\Wg}{\mathbf W}
\newcommand{\wg}{\mathbf w}

\newcommand{\qg}{\mathbf q}
\newcommand{\pg}{\mathbf p}

\newcommand{\norm}[1]{\left\lVert#1\right\rVert}

\newcommand{\tphi}{\tilde \varphi}

\newcommand{\ds}{\displaystyle}

\newcommand{\rb}{\overline{r}}

\newcommand{\Oc}{\mathcal  O}

\newcommand{\rmax}{r_{\max}}
\newcommand{\md}{\mathrm{d}}

\newcommand{\Fm}{\mathcal{F}}

\definecolor{aquamarine}{rgb}{0.13, 0.68, 0.8}
\def\lp {\left( }
\def\rp {\right) }

\newcommand{\baco}{\left\{ \begin{array}}\newcommand{\eaco}{\end{array} \right.}

\newtheorem{theorem}{Theorem}[section]
\newtheorem{proposition}[theorem]{Proposition}
\newtheorem{corollary}[theorem]{Corollary}
\newtheorem{remark}[theorem]{Remark}

\theoremstyle{definition}

\newcolumntype{C}[1]{>{\centering\arraybackslash}m{#1}}
\makeatletter
\define@key{Gin}{mysize}[]{\setkeys{Gin}{scale=1}}
\makeatother
\newcommand{\setmysize}[1]{\setkeys{Gin}{#1}}

\title{{\bf{Adaptation in a heterogeneous environment\\  II: To be three or not to be }}\thanks{This work has received funding from the French ANR RESISTE (ANR-18-CE45-0019) and DEEV (ANR-20-CE40-0011-01) projects, from Excellence Initiative of Aix-Marseille Universit\'e~-~A*MIDEX, a French ``Investissements d'Avenir'' program, and from  the {\it région Normandie} BIOMA-NORMAN (21E04343) project. The authors acknowledge support of the Institut Henri Poincaré (UAR 839 CNRS-Sorbonne Universit\'e) and LabEx CARMIN (ANR-10-LABX-59-01).}}
\author{M. Alfaro$^{\hbox{\small{ a},\small{b}}}$, F. Hamel$^{\hbox{\small{ c}}}$, F. Patout$^{\hbox{\small{ b}}}$ and L. Roques$^{\hbox{\small{ b}}}$\\
\\
\footnotesize{$^{\hbox{a }}$Univ Rouen Normandie, LMRS, CNRS, Rouen, France}\\
\footnotesize{$^{\hbox{b }}$INRAE, BioSP, 84914, Avignon, France}\\
\footnotesize{$^{\hbox{c }}$Aix Marseille Univ, CNRS, I2M, Marseille, France}}
\date{}

\begin{document}
\maketitle
%\tableofcontents

\begin{abstract}
We propose a model to describe the adaptation of a phenotypically structured popu\-lation in a $H$-patch environment connected by migration, with each patch associated with a different phenotypic optimum, and we perform a rigorous mathematical analysis of this model. We show that the large-time behaviour of the solution (persistence or extinction) depends on the sign of a principal eigenvalue, $\lambda_H$, and we study the dependency of $\lambda_H$ with respect to $H$. This analysis sheds new light on the effect of increasing the number of patches on the persistence of a population, which has implications in agroecology and for understanding zoonoses; in such cases we consider a pathogenic population and the patches correspond to different host species. The occurrence of a {\it springboard} effect, where the addition of a patch contributes to persistence, or on the contrary the emergence of a detrimental effect by increasing the number of patches on the persistence, depends in a rather complex way on the respective positions in the phenotypic space of the optimal phenotypes associated with each patch. From a mathematical point of view, an important part of the difficulty in dealing with $H\ge 3$, compared to $H=1$ or $H=2$, comes from the lack of symmetry. Our results, which are based on a fixed point theorem,  comparison principles, integral estimates, variational arguments, rearrangement techniques, and numerical simulations, provide a better understanding of these dependencies. In particular, we propose a precise characterisation of the situations where the addition of a third patch increases or decreases the chances of persistence, compared to a situation with only two patches.
\end{abstract}

%%%%%%%%%%%%%%%%%%%%%%%%%%%%%%%%%%%%%%%%%%%%%%%%%%%%%%%%

\section{Introduction}

In \cite{HamLavRoq20}, we analysed PDE systems describing the dynamics of adaptation of a phenotypically structured population, under the effects of mutation, selection and migration in a two-patch environment, each patch being associated with a different phenotypic optimum. Consistently with  current literature \cite{DebRon13,MesCzi97,PapDav13} in evolutionary biology, our analysis showed that migration between the two patches leads to a locally reduced fitness. This reduction in fitness is known as a \textit{migration load} ~\cite{GarKir97} and implies decreased chances of persistence of the global population in a two-patch system compared to a single patch environment.

From an epidemiological viewpoint, the two patches can be interpreted as two different types of \textit{hosts} (different species, or different genetic variants).
Thus, the above result means that it is more difficult for a pathogen to adapt and establish in a two-patch environment connected by migration than in a single patch environment. This observation is in agreement with one of the fundamental principles of agroecology~\cite{CaqGas20,FAO18}, which is that host species diversification should lead to a higher resilience of agroecosystems.

Nevertheless, as discussed in~\cite{LavMar20}, what is right for two patches may not be necessarily right for three patches or more. The presence of a third host may indeed cause a {\it springboard} effect, leading to higher chances of persistence of the pathogen, compared to an environment with two hosts. This is a common pattern in zoonoses. For instance the main reservoirs of influenza~A virus are the aquatic birds, but it is widely accepted that these viruses need to adapt in an intermediate host (such as pig or poultry) before they lead to an outbreak in human populations \cite{ParMur15,WebWeb01}. Coronaviruses, including SARS-CoV-1 and SARS-CoV-2, are also zoonotic, with bats as presumed main reservoir. Intermediate hosts are also suspected to have played an important role in the 2003 and 2019 outbreaks~\cite{LatHu20,LauWoo05}.

Up to our knowledge, there is no rigorous mathematical framework to study the effect of host diversification on the adaptation of a pathogen, when there are three hosts or more. Here, using the same assumptions as in \cite{HamLavRoq20}, we study a system for a phenotypically structured population, under the effects of mutation, selection and migration between several hosts. Our main goal is to check whether the introduction of a third host leads to the above-mentioned springboard effect, or whether, on the contrary,  it reduces the chances of persistence of the global population compared to a situation with two hosts. Thus, we mostly focus on the case of three hosts, described by the following system:
$$\left\{\begin{array}{rcl}
\partial_t u_1 (t,\x) & = &\ds \frac{\mu^2} 2 \ \Delta u_1(t,\x)+f_1(\x,u_1(t,\cdot))+\delta\,\left[\frac{u_2(t,\x)+u_3(t,\x)}{2} - u_1(t,\x)\right]\!, \vspace{2mm}\\
\partial_tu_2 (t,\x) & =&\ds \frac{\mu^2} 2 \ \Delta u_2(t,\x)+f_2(\x,u_2(t,\cdot)) +\delta\,\left[\frac{u_1(t,\x)+u_3(t,\x)}{2} - u_2(t,\x)\right]\!, \vspace{2mm}\\
\partial_t u_3 (t,\x) & =&\ds \frac{\mu^2} 2 \ \Delta u_3(t,\x)+f_3(\x,u_3(t,\cdot)) +\delta\,\left[\frac{u_1(t,\x)+u_2(t,\x)}{2} - u_3(t,\x)\right]\!,
\end{array}\right.$$
for $t>0,$ and $\x=(x_1,\cdots,x_n)\in \R^n$. Here, and as in \cite{HamLavRoq20}, $\x$ is a breeding value for phenotype (for short, we simply write ``phenotype'' in the sequel),  and corresponds to a set of $n\geq 1$ traits. The unknowns $u_i$ are the phenotype densities in the hosts $i\in\{1,2,3\}$, the Laplace operator describes the mutation effects on the phenotype, $\mu>0$ is a mutational parameter, $\delta>0$ the migration rate, and the functions $f_i(\x,u_i(t,\cdot))$ describe the growth of the phenotype $\x$ in the host $i$ (the precise assumptions on the functions $f_i$ are given below).  Note that the migration and mutation parameters are assumed to be identical over the three hosts. In particular, each host sends migrants to the other hosts, at a rate $\delta$, and the amount of migrants is evenly split between the other hosts, hence the factor $1/2$.

Some of our results also deal with the general case of $H\ge 2$ hosts. Thus, we also consider the following system:
\begin{equation}\label{eq:sys H}
\partial_t u_i(t,\x)=\frac{\mu^2}{2}\Delta u_i(t,\x)+f_i(\x,u_i(t,\cdot))+  \delta \left(   \sum \limits_{\substack{k=1 \\k \neq i}}^H\frac {u_k(t,\x)}{H-1}-u_i(t,\x) \right)\!,\ \ 1\le i\le H,
\end{equation}
for $t>0$ and $\x\in\R^n$. Again, we assume that each host sends migrants to the other hosts, at a rate $\delta$, and that the amount of migrants is evenly split between the other hosts, hence the factor $1/(H-1)$. With these assumptions, the migration rate from a given host $i$ towards the pool made up of all the $H-1$ other hosts does not depend on $H$. In the particular case $H=1$, there is no migration and~\eqref{eq:sys H} reduces to the scalar equation $\p_tu_1(t,\x)=(\mu^2/2)\Delta u_1(t,\x)+f_1(\x,u_1(t,\cdot))$.

Each host is characterised by a phenotype optimum $\opt_i \in \R^n$. Namely, the fitness (reproductive success) of a  phenotype $\x$ in the host $i$ is described by a function $r_i(\x)$ that decreases away from the optimum $\opt_i$. As in \cite{HamLavRoq20} we use Fisher's geometrical phenotype-to-fitness model (FGM)  \cite{MarLen15,Ten14}, which assumes that:
\begin{align}\label{def ri}
 r_i(\x) = \rmax-\alpha \, \frac{ \norm{\x-\opt_i}^2}{2},
\end{align}
where $\rmax\in\R$ is the fitness of the optimal phenotype ($\opt_i$) in the host $i$, $\alpha>0$ is a measure of the intensity of selection, and $\|\ \|$ denotes the Euclidean norm in $\R^n$ (we use the same notation for any dimension $n\ge1$, we also denote $\cdot$ the Euclidean inner product in $\R^n$).

In \cite{HamLavRoq20}, we considered two types of growth functions, either linear (corresponding to a Malthusian growth) or  logistic-like. Here, we only consider the logistic-like growth term, which describes a nonlocal competition between the phenotypes within each host. This is the most biologically relevant and mathematically involved case. It corresponds to $f_i(\x,\phi)=\phi(\x)\,\big(r_i(\x)-\int_{\R^n}\phi(\y)\md\y\big)$ for any $\x\in\R^n$ and any continuous~$L^1(\R^n)$ function $\phi:\R^n\to\R$, that is,
\begin{equation}\label{eq:f_type2}
f_i(\x,u_i(t,\cdot))=u_i(t,\x)\, \lp r_i(\x) -  \int_{\R^n} u_i(t,\y) \, \md\y\rp.
\end{equation}

When $H=2$, we established in \cite{HamLavRoq20} the existence and uniqueness of the solution of the Cauchy problem associated with \eqref{eq:sys H} with growth functions~\eqref{eq:f_type2}, under some symmetry assumptions on the initial conditions. Then, still in the case $H=2$, we obtained a characterisation of the large-time behaviour of the solution (persistence or extinction) based on the sign of a principal eigenvalue, here denoted $\lambda_2$.

In this work, we first extend these results to the general case $H\ge 2$. Let us emphasise that, when $H\geq 3$, symmetry arguments are no longer  applicable (except in some very particular configurations). In Section~\ref{sec:existence}, we thus state the existence and uniqueness of the solution of the Cauchy problem associated with~\eqref{eq:sys H}, and we establish a condition for the persistence of the population, based on the sign of a principal eigenvalue denoted~$\lambda_H$, and we describe the large-time behaviour of the total population size. Then, we propose in Section~\ref{sec:effect} a mathematical analysis of the effects of the parameters on the value of $\lambda_H$, including the effect of the mutation, selection and migration parameters. As explained above, our main objective is to investigate the effect of adding a third host, compared to a baseline situation with two hosts. In that respect, we fix  the position of the two optima~$\opt_1$ and~$\opt_2$, leading to a given value of~$\lambda_2$. Then, in Section~\ref{sec:H3vsH2}, we study the sign of $\lambda_2-\lambda_3$, depending on the position~$\opt_3$ of the optimum of the third host. This sign determines whether the presence of a third host increases ($\lambda_2-\lambda_3>0$) or decreases ($\lambda_2-\lambda_3<0$) the chances of persistence.  We propose a mathematical analysis, completed by some numerical simulations. The outcomes are various, and sometimes surprising, see the discussion in Section \ref{s:discussion}. We gather all the proofs in Section~\ref{sec:proofs}.

%%%%%%%%%%%%%%%%%%%%%%%%%%%%%%%%%%%%%%%%%%%%%%%%%%%%%%%%
%%%%%%%%%%%%%%%%%%%%%%%%%%%%%%%%%%%%%%%%%%%%%%%%%%%%%%%%

\section{Existence, uniqueness and persistence results}\label{sec:existence}

%%%%%%%%%%%%%%%%%%%%%%%%%%%%%%%%%%%%%%%%%%%%%%%%%%%%%%%%

\subsection{The Cauchy problem}

For $H\ge 2$, we consider the nonlinear and nonlocal system~\eqref{eq:sys H}, where the functions $\x\mapsto r_i(\x)$ are as in \eqref{def ri}, while the functions $f_i(\x,u_i(t,\cdot))$ are as in \eqref{eq:f_type2}. System~\eqref{eq:sys H} is supplemented with an initial datum
\begin{equation}\label{data}
\ug^0=(u_1^0,\ldots,u_H^0),\ \ \text{$u_i^0\in C(\R^n)\cap L^\infty(\R^n)\cap L^1(\R^n)$ and $u_i^0\ge0$ for all $1\leq i\leq H$},
\end{equation}
and we focus on the well-posedness of the Cauchy problem \eqref{eq:sys H}-\eqref{data}, under some additional decay properties on $\ug^0$, see~\eqref{init bound} below. The proofs are very different from that for the case $H=2$ in~\cite{HamLavRoq20}, where symmetry arguments were used to reduce the study of the nonlinear system~\eqref{eq:sys H} to that of a linear scalar equation. Such arguments cannot be applied in the general case $H> 2$.

\begin{theorem}[Well-posedness]\label{thm:well-pos}
Assume that there exist positive constants $K$ and $\theta$ such that the initial condition $\ug^0$ in~\eqref{data} satisfies
\begin{equation}\label{init bound}
\forall\,1\leq i\leq H,\ \ \forall\,\x\in\R^n,\quad 0\le u_i^0(\x)\le K\,e^{-\theta\norm{\x}}.
\end{equation}
Then, there is a unique solution $\ug=(u_1,\ldots,u_H)$ of~\eqref{eq:sys H}-\eqref{data} in $C([0,+\infty)\times\R^n,\R^H)\cap C^{1;2}_{t;\x}((0,+\infty)\times\R^n,\R^H)$ such that
\begin{equation}\label{ineqexp}
\forall\,1\leq i\leq H,\ \ \forall\,t\ge0,\ \ \forall\,\x\in\R^n,\quad  0\le u_i(t,\x)  \leq K\,e^{(\rmax+1)t -\min(\theta,1/\mu)\norm{\x}},
\end{equation}
and the maps $t\mapsto\int_{\R^n}u_i(t,\x)\,\md\x$ are locally Lipschitz-continuous in $[0,+\infty)$.
\end{theorem}

%%%%%%%%%%%%%%%%%%%%%%%%%%%%%%%%%%%%%%%%%%%%%%%%%%%%%%%%

\subsection{The principal eigenvalue}

Here, we present some linear material, namely the principal eigenvalue $\lambda_H$ and the  principal eigenvector $\Phi=(\varphi_1,\ldots,\varphi_H)$ solving $\mathcal A \Phi^T=\lambda_H\Phi^T$, where the operator
\begin{equation}\label{eq:defA}
\A:=-\frac{\mu^2}{2} \Delta-\begin{pmatrix}
r_1(\x)-\delta & \ds\frac{\delta}{H-1} & \cdots & \ds\frac{\delta }{H-1} \\
\ds\frac{\delta }{H-1} & r_2(\x)-\delta & \ddots & \vdots \\
\vdots & \ddots & \ddots & \ds\frac{\delta }{H-1} \\
\ds\frac{\delta }{H-1} & \cdots & \ds\frac{\delta }{H-1} & r_H(\x)-\delta \\
\end{pmatrix}
\end{equation}
is obtained by linearizing system \eqref{eq:sys H} around the trivial solution $(0,\ldots,0)$. Since $\delta>0$ and since the fitness functions, defined in \eqref{def ri}, satisfy $r_i(\x)\to-\infty$ as $\norm{\x}\to +\infty$, $\A$ can be seen as a cooperative Schr\"odinger operator with confining potentials, of which  the linear analysis is classical, see \cite{Agm-82,damascelli2013symmetry} in a bounded domain, or \cite{Car-09, cardoulis2015principal} in a slightly different setting. Another approach, used in~\cite{HamLavRoq20}, consists in defining $\lambda_H$ as the limit, as $R\to+\infty$, of the Dirichlet principal eigenvalue $\lambda_H^R$ in the open Euclidean ball $B(\mathcal{O},R)$ of center the origin
$$\mathcal{O}:=(0,\ldots,0)$$
and radius $R>0$. Similarly, $\Phi=(\varphi_1,\ldots,\varphi_H)$ can be defined as in~\cite{HamLavRoq20} as the locally uniform limit of the functions $\Phi^R$ (after suitable multiplication, say, with $\Phi^R_1(\mathcal{O})=1$). The functions $\Phi^R$ belong to $C^\infty_0(\overline{B(\mathcal{O},R)})^H$ from standard elliptic estimates, where~$C^\infty_0(\overline{B(\mathcal{O},R)})$ is the space of $C^\infty(\overline{B(\mathcal{O},R)})$ functions vanishing on $\partial B(\mathcal{O},R)$. It is also known that the functions $\varphi_i$'s are positive and decay exponentially as $\|\x\|\to+\infty$, thanks to the confining property of the potentials $r_i$. We thus consider the principal eigenvalue $\lambda_H$, the normalised\footnote{In the sequel, we say that $\Psi\in L^2(\R^n)^H$ is normalised whenever $\int _{\R^{n}}\Vert \Psi(\x)\Vert  ^2\,\md\x=1.$} and positive (that is, positive componentwise) principal eigenvector $\Phi:=(\varphi_1,\ldots,\varphi _H) \in\big(C^\infty_0(\R^n)\cap L^1(\R^n)\big)^H$ (with~$C^\infty_0(\R^n)$ being the space of~$C^\infty(\R^n)$ functions converging to~$0$ as~$\|\x\|\to+\infty$), satisfying
\begin{equation}\label{eq:eigenvalue_pb}
-\frac{\mu^2}{2}\Delta \varphi_i(\x)-r_i(\x) \, \varphi_i(\x)-  \delta \left(\sum \limits_{\substack{k=1 \\k \neq i}}^H\frac {\varphi_k(\x)}{H-1}-\varphi_i(\x) \right)=\lambda_H\, \varphi_i(\x)\ \hbox{ in }\R^n, \quad 1\leq i \leq H.
\end{equation}
We also know that $\Phi\in(H^1(\R^n)\cap L^2_{w}(\R^n))^H$, with
$$L^2_w(\R^n):=\left\{\psi:\R^n\to\R\hbox{ such that }\x\mapsto \Vert \x\Vert\,\psi(\x) \in L^2(\R^n)\right\},$$
and that the following Rayleigh formula is available:
\begin{equation}\label{Rayleigh}
\lambda _H=Q_H(\Phi)=\min\Big\{Q_H(\Psi): \Psi\in(H^1(\R^n)\cap L^2_{w}(\R^n))^H,\,\int _{\R^{n}}\!\!\Vert \Psi(\x)\Vert  ^2\,\md\x=1\Big\},
\end{equation}
where
\begin{equation}\label{eq:rayleigh_Q}\begin{array}{rcl}
Q_H(\Psi)=Q_H(\psi_1,\ldots,\psi_H) & \!\!\!:=\!\!\! & \ds\sum _{i=1}^{H} \lp \frac{\mu ^2}{2}\!\int_{\R^{n}}\!\|\nabla \psi_i(\x)\|^2\,\md\x -\!\int_{\R^{n}}\!r_i (\x)\,(\psi_i(\x))^2\, \md\x \rp\\
& & \ds+\,\delta\left(1-\sum_{1\le i< j\le H}\frac{2}{H-1}\int_{\R^n}\psi_i(\x)\,\psi_j(\x)\,\md\x\right).\end{array}
\end{equation}
Furthermore, the principal eigenvector $\Phi$ and its opposite $-\Phi$ are the unique normalised minima of $Q_H$ in $(H^1(\R^n)\cap L^2_{w}(\R^n))^H$, and $\Phi$ is the unique nonnegative (componentwise) eigenvector of $\mathcal{A}$ in $(H^1(\R^n)\cap L^2_{w}(\R^n))^H$.

The principal eigenvalue $\lambda_H$ depends on the parameters $\delta$, $\alpha$, $\mu$, $\rmax$ (on a trivial additional manner, since $\lambda_H+\rmax$ is independent of $\rmax$), as well as on the optima $(\opt_i)_{1\le i\le H}$. In the following sections, depending on the context, we will also use the notations $\lambda_H(\delta)$, $\lambda_H(\opt_1,\ldots,\opt_H)$, or $\lambda_H(\delta,\alpha,\mu,\opt_1,\ldots,\opt_H)$ in order to emphasise the effect of the various parameters on $\lambda_H$.

For the particular case $H=1$, we refer to Section \ref{sec:reference-case}.

%%%%%%%%%%%%%%%%%%%%%%%%%%%%%%%%%%%%%%%%%%%%%%%%%%%%%%%%

\subsection{Extinction vs persistence}

Let $\ug=(u_1,\ldots,u_H)$ denote the solution of~\eqref{eq:sys H}-\eqref{ineqexp} given in Theorem~\ref{thm:well-pos}. Remember that~$u_i\ge0$ in $[0,+\infty)\times\R^n$, for each $1\le i\le H$. Assume without loss of generality that~$\ug^0\not\equiv(0,\ldots,0)$ in $\R^n$, that is, there is $1\le j\le H$ such that $u_j\not\equiv0$ in $\R^n$. The strong parabolic maximum principle applied to $u_j$ implies that $u_j>0$ in $(0,+\infty)\times\R^n$, and then~$u_i>0$ in $(0,+\infty)\times\R^n$ from the strong parabolic maximum principle applied to each $u_i$ with $i\neq j$. For $t\ge0$, we define the population size within each host $1 \leq i \leq H$ as
$$N_i(t):=\int_{\R^n} u_i(t,\y) \, \md\y,$$
which is a positive real number for $t>0$, by~\eqref{ineqexp} and the previous observations. Furthermore, the Lebesgue dominated convergence theorem implies that each function $N_i$ is continuous in $[0,+\infty)$. The total population size in the system is defined as
$$N(t):=N_1(t)+ \ldots + N_H(t),$$
and the mean growth rate within each host as
\begin{equation}\label{eq:def_rbar}
\ds\rb_i(t):=\frac{\int_{\R^n} r_i(\y) \, u_i(t,\y) \, \md \y }{\int_{\R^n} u_i(t,\y) \,  \md\y }=\frac{\int_{\R^n} r_i(\y) \, u_i(t,\y) \, \md \y }{N_i(t)}.
\end{equation}
By integrating~\eqref{eq:sys H} over $[t_1,t_2]\times B(\mathcal{O},R)$ for any $0<t_1<t_2<+\infty$ and $R>0$, by passing to the limit as $R\to+\infty$ and then as $t_2-t_1\to0$, and by using~\eqref{ineqexp} as in the proof of Theorem~\ref{thm:well-pos} in Section~\ref{sec51} (see especially~\eqref{integrals} below), it follows that the population sizes are differentiable in $(0,+\infty)$ and satisfy
$$N_i'(t)=\rb_i(t)N_i(t)-N_i(t)^2+\frac{\delta}{H-1}\sum\limits_{\substack{k=1 \\k \neq i}}^H(N_k(t)-N_i(t))$$
for every $1\le i\le H$ and $t>0$ (and even at $t=0$ by continuity as $t\to0$ in the above formula).

The next result shows that the sign of the principal  eigenvalue $\lambda_H$ fully determines the fate of the population at large times.

\begin{proposition}[Extinction vs persistence]\label{theo pers log}
Let  $\lambda_H$ be given by~\eqref{Rayleigh}, and let $\ug$ be the solution of~\eqref{eq:sys H}-\eqref{ineqexp} given by Theorem~$\ref{thm:well-pos}$, with a non-trivial initial condition $\ug^0$, let $N_i(t)$ be the corresponding population sizes in each habitat, and let $N(t)$ the total population size in the system.
\begin{enumerate}
\item [(i)] If $\lambda_H>0$ and $\ug^0$ is compactly supported, then
\[ N(t) \to 0 \text{ as }t\to+\infty\ \ (\text{\textit{extinction of the population}}).\]
\item [(ii)] If $\lambda_H=0$ and $\ug^0$ is compactly supported, then
\[\liminf_{t\to+\infty}\Big(\min_{1\le i\le H}N_i(t)\Big)=0\ \ (\text{\textit{partial extinction of the population}}).\]
\item [(iii)] If $\lambda_H<0$, then
 \end{enumerate}
\begin{equation}\label{persistence}
\limsup_{t\to+\infty} N(t)\!\ge\!\limsup_{t\to+\infty}\Big(\!\max_{1\le i\le H}N_i(t)\!\Big)\!\ge\!-\lambda_H\!>\!0\ (\text{\textit{persistence of the population}}).
\end{equation}
\end{proposition}

%%%%%%%%%%%%%%%%%%%%%%%%%%%%%%%%%%%%%%%%%%%%%%%%%%%%%%%%

\subsection{Stationary states}

In this subsection, we are interested in positive bounded stationary states of~\eqref{eq:sys H}. By posi\-tive, we mean positive componentwise in $\R^n$ since, from the strong elliptic maximum principle applied to this cooperative system, if a bounded stationary state is nonnegative componentwise in $\R^n$, it is either positive componentwise, or identically $0$ in $\R^n$. From the confining properties of the fitnesses $r_i$, any positive bounded stationary state of~\eqref{eq:sys H} necessarily decays to $0$ exponentially as $\|\x\|\to+\infty$. Therefore, it follows from Theorem~\ref{thm:well-pos} and the proof of Proposition~\ref{theo pers log} that~\eqref{eq:sys H} has no positive bounded stationary state if $\lambda_H\ge0$ (see Remark~\ref{remstat} below for further details).

Furthermore, it turns out that, when persistence occurs, that is, $\lambda_H<0$, the nature of the stationary states deeply depends on the number of hosts. Namely, when $H=1,\, 2$, the stationary states are proportional to the principal eigenfunctions (and the study of their ``shapes'' has recently received a lot of attention, see e.g.   \cite{AlfVer19,DjiDucFab17,Mir17,MirGan20}). This is no longer true in general in the case $H>2$, due to a possible symmetry breaking. More precisely, we have the following results.

\begin{proposition}[Stationary states]\label{prop:statio_states}
\noindent (i) Assume that $H=1$ and $\lambda_1<0$. The positive bounded stationary state of~\eqref{eq:sys H} is unique and is given by:
$$\ds p_1(\x)= -\lambda_1 \frac{\varphi_1(\x)}{\int_{\R^n}\varphi_1(\y)\,\md\y},\ \ \x\in \R^n.$$
\noindent (ii) Assume that $H=2$ and $\lambda_2<0.$ Then the principal eigenvector $(\varphi_1,\varphi_2)$ of~$\mathcal{A}$ in~$(H^1(\R^n)\cap L^2_w(\R^n))^2$ satisfies $\int_{\R^n}\varphi_1(\y)\,\md\y=\int_{\R^n}\varphi_2(\y)\,\md\y$, and
\begin{equation}\label{eq:2hosts_statio_state}
\x\mapsto(p_1(\x),p_2(\x)):=\lp-\lambda_2 \frac{\varphi_1(\x)}{\int_{\R^n}\varphi_1(\y)\,\md\y} ,\ -\lambda_2 \frac{\varphi_2(\x)}{\int_{\R^n}\varphi_2(\y)\,\md\y} \rp
\end{equation}
is a positive bounded stationary state of \eqref{eq:sys H} proportional to $(\varphi_1,\varphi_2)$, and satisfying $\int_{\R^n}p_1(\x)\,\md\x=\int_{\R^n}p_2(\x)\,\md\x=-\lambda_2$.\\
\noindent (iii) Assume that $H>2$ and $\lambda_H<0$. Then positive bounded stationary states of~\eqref{eq:sys H}, if any, are in general not proportional to the principal eigenvector $(\varphi_1,\ldots,\varphi_H)$ of $\mathcal A$.
\end{proposition}

%%%%%%%%%%%%%%%%%%%%%%%%%%%%%%%%%%%%%%%%%%%%%%%%%%%%%%%%
%%%%%%%%%%%%%%%%%%%%%%%%%%%%%%%%%%%%%%%%%%%%%%%%%%%%%%%%

\section{Effect of the main parameters on persistence}\label{sec:effect}

In this section, we investigate the dependence of $\lambda_H$ with respect to the parameters $\delta>0$, $\alpha>0$, $\mu>0$, and the optima $(\opt_i)_{1\le i\le H}\in(\R^n)^H$. It is useful to start with the ``reference case'', namely $H=1$.

%%%%%%%%%%%%%%%%%%%%%%%%%%%%%%%%%%%%%%%%%%%%%%%%%%%%%%%%

\subsection{The reference case with a single host}\label{sec:reference-case}

In that case of a single host at position $\opt _1$, there is no transfer rate, that is,~\eqref{eq:eigenvalue_pb} reduces to a single equation with $\delta=0$. The principal eigenpair $(\lambda_1,G_1)$ solving
\begin{align}\label{spectral  H=1}
-\frac{\mu^2}{2}\Delta G_1(\x)-r_1(\x)\,G_1(\x)=\lambda_1\,G_1(\x)\ \hbox{ in }\R^n,
\end{align}
is explicitly given by
\begin{equation}\label{def:gaussian}
\lambda_1=\lambda_1(\alpha,\mu)=\frac{\mu n\sqrt{\alpha}}{2}-\rmax, \quad G_1(\x)=G(\x-\opt_1),\quad G(\y):=C_G\,e^{-\sqrt{\alpha}\|\y\|^2/(2\mu)},
\end{equation}
where $C_G>0$ is a normalising positive constant that guarantees that $\int_{\R^n} G^2(\y) d\y =1$. Notice that, as the problem~\eqref{spectral  H=1} is left invariant by any translation of the phenotypic space, the eigenvalue $\lambda_1$ does not depend on the position of the optimum $\opt_1$.

%%%%%%%%%%%%%%%%%%%%%%%%%%%%%%%%%%%%%%%%%%%%%%%%%%%%%%%%

\subsection{The case of $H$ hosts with $H\ge2$}

In this subsection, we assume $H\geq 2$ and denote $\lambda_H=\lambda_H(\delta,\alpha,\mu,\opt_1,\ldots,\opt_H)$ the principal eigenvalue of $\mathcal{A}$ in $(H^1(\R^n)\cap L^2_w(\R^n))^H$, depending on $\delta>0$, $\alpha>0$, $\mu>0$, and $(\opt_i)_{1\le i\le H}$. As above, the corresponding normalised eigenvector is denoted $\Phi=(\varphi_1,\ldots,\varphi_H)$. Observe that, if we consider~\eqref{eq:eigenvalue_pb} with $\delta=0$, the $H$ equations are decoupled and we then define
\begin{equation}\label{deflambdaH0}
\lambda_H(0,\alpha,\mu,\opt_1,\ldots,\opt_H):=\lambda_1(\alpha,\mu)=\frac{\mu n\sqrt{\alpha}}{2}-\rmax.
\end{equation}
For later use, we define, for $1\le i\le H$,
$$G_i(\x):=G(\x-\opt_i)=C_G\,e^{-\sqrt{\alpha}\|\x-\opt_i\|^2/(2\mu)}$$
and it is convenient to straightforwardly compute
\begin{equation}\label{prod-sca-gauss}
\int _{\R^n}G_i(\x)\,G_j(\x)\,\md\x=e^{-\sqrt{\alpha}\,\|\opt_i-\opt_j\|^2/(4\mu)}.
\end{equation}

The following first result of this section asserts that the principal eigenvalue depends continuously on the parameters, and that an increase of the migration rate $\delta>0$, or of the intensity of selection $\alpha>0$, or of the mutational parameter $\mu>0$ reduces the chances of persistence of the population.

\begin{proposition}[Continuity and monotonicity w.r.t. parameters]\label{prop monolambda}
Let $H\geq 2$ be given. The function $(\delta,\alpha,\mu,\opt_1,\ldots,\opt_H)\mapsto\lambda_H(\delta,\alpha,\mu,\opt_1,\ldots,\opt_H)$ is continuous in $[0,+\infty)\times(0,+\infty)^2\times(\R^n)^H$, and concave with respect to $(\delta,\alpha,\mu)\in[0,+\infty)\times(0,+\infty)^2$. Furthermore, $\lambda_H(\delta,\alpha,\mu,\opt_1,\ldots,\opt_H)$ is increasing with respect to $\alpha>0$ and to $\mu>0$. Lastly, either the optima $\opt_i$ are all identical and $\lambda_H(\delta,\alpha,\mu,\opt_1,\ldots,\opt_H)=\lambda_1(\alpha,\mu)=\mu\,n\,\sqrt{\alpha}/2-\rmax$ for all $\delta\ge0$, or $\lambda_H(\delta,\alpha,\mu,\opt_1,\ldots,\opt_H)$ is increasing with respect to $\delta\ge0$.
\end{proposition}

The following bounds on $\lambda_H$ shows, in particular, that the case with many hosts is always less favourable for the population than the case with a single host.

\begin{proposition}[Bounds on $\lambda_H$ w.r.t. $\lambda_1$]\label{prop bound lambda}
Let $H\geq 2$ be given. Then, for all $\delta>0$, $\alpha>0$, $\mu>0$, and $(\opt_i)_{1\le i\le H}\in(\R^n)^H$,
\begin{equation}\label{ineqs}\begin{array}{rcl}
\ds\frac{\mu n\sqrt{\alpha}}{2}\!-\!\rmax\!=\!\lambda_1(\alpha,\mu) & \!\!\!\leq\!\!\! & \ds\lambda _1(\alpha,\mu)+\delta\left(1-\sum_{1\le i< j\le H}\frac{2}{H-1}\int_{\R^n}\varphi_i(\x)\varphi_j(\x)\md\x\right)\vspace{3pt}\\
& \!\!\!\leq\!\!\! & \ds\lambda _H(\delta,\alpha,\mu,\opt_1,\ldots,\opt_H)\vspace{3pt}\\
& \!\!\!\leq\!\!\! & \ds\lambda _1(\alpha,\mu)\!+\!\delta\!\left(\!1-\!\!\sum_{1\le i< j\le H}\!\frac{2}{H(H\!-\!1)}\!\int_{\R^n}\!\!G_i(\x)G_j(\x)\md\x\!\right)\vspace{3pt}\\
& \!\!\!\le\!\!\! & \ds\lambda_1(\alpha,\mu)+\delta=\frac{\mu n\sqrt{\alpha}}{2}-\rmax+\delta.\end{array}
\end{equation}
Moreover, for every $\mu>0$, $\alpha>0$ and $(\opt_i)_{1\le i\le H}\in(\R^n)^H$, $\lambda_H(\delta,\alpha,\mu,\opt_1,\ldots,\opt_H)$ is bounded independently of $\delta$.
\end{proposition}

In the sequel, and without loss of generality, we fix the optima $\opt_1$ and $\opt_2$ at
\begin{equation}\label{O1-O2}
\opt _1:=(-\beta,0,\ldots,0), \quad  \opt _2:=(\beta,0,\ldots,0),
\end{equation}
for some $\beta\geq 0$. In the case $H=2$, we now denote $\lambda_2(\delta,\alpha,\mu,\beta)$ the principal eigenvalue of~\eqref{eq:eigenvalue_pb}. In~\cite{HamLavRoq20}, we defined the \textit{habitat difference} by $m_D := \|\opt_1-\opt_2\|^2/2=2\,\beta^2$, and we proved that $\lambda_2(\delta,\alpha,\mu,\beta)$ was an increasing function of $m_D$ (and therefore of $\beta$), see \cite[Proposition~4]{HamLavRoq20}. We also proved that $\lambda_2(\delta,\alpha,\mu,\beta)\to\lambda_1(\alpha,\mu)=\mu n\sqrt{\alpha}/2-\rmax$ as $\delta\to0^+$. The result below, which is a direct consequence of Proposition~\ref{prop bound lambda} (together with  \eqref{prod-sca-gauss}), extends this last result to the case of $H\ge 2$ hosts.

\begin{corollary}\label{cor bound lambda bis}
Let $H\geq 2$ be given. Then, for every $\delta>0$, $\alpha>0,$ $\mu>0$, $\beta\geq 0$, and~$(\opt_i)_{1\le i\le H}\in(\R^n)^H$ with $\opt_1$ and $\opt_2$ as in~\eqref{O1-O2}, there holds
$$\lambda_1(\alpha,\mu)\le\lambda _H(\delta,\alpha,\mu,\opt_1,\ldots,\opt_H)\ \leq \lambda _1(\alpha,\mu)+\delta\left(1-\frac 2{H(H-1)}e^{-\sqrt{\alpha}\beta^2/\mu}\right).$$
Incidentally, for $H=2$, $\lambda_2(\delta,\alpha,\mu,\beta)-\lambda_1(\alpha,\mu)\to0$ when $\alpha \to 0^+$, $\mu\to+\infty$ or $\beta\to0$.
\end{corollary}

Having understood, by Propositions~\ref{prop monolambda}-\ref{prop bound lambda}, the small migration rate regime $\delta\to 0$, we now turn to the asymptotics $\delta \to +\infty$. The case $H=2$ was investigated in \cite[Proposition~4]{HamLavRoq20}, where we obtained that $\lambda_2(\delta,1,\mu,\beta)\to \lambda_1(1,\mu)+\beta^2/2=\mu n/2-\rmax+\beta^2/2$ as $\delta\to+\infty$ (in the case $\alpha=1$). We generalise here this result to any number $H\ge 2$ of hosts.

\begin{proposition}[Large migration rate]\label{prop:lambda-infini}
Let $H\geq 2$, $\alpha>0$, $\mu>0$, and $(\opt_i)_{1\le i\le H}\in(\R^n)^H$ be given. Then $\lambda_H(\delta,\alpha,\mu,\opt_1,\ldots,\opt_H)\nearrow \lambda_{H,\infty}(\alpha,\mu,\opt_1,\ldots,\opt_H)$ as $\delta\to+\infty$, where $\lambda_{H,\infty}(\alpha,\mu,\opt_1,\ldots,\opt_H)$ is the principal eigenvalue for a population living in a single host and with ``effective'' fitness equal to the mean of the $H$ original fitness functions, that~is,
\begin{equation}\label{eq-large-delta}
-\frac{\mu ^2}{2}\Delta \varphi (\x)-R_H(\x)\,\varphi(\x)=\lambda_{H,\infty}(\alpha,\mu,\opt_1,\ldots,\opt_H)\,\varphi(\x), \quad R_H(\x):=\frac 1 H \sum_{i=1}^H r_i(\x),
\end{equation}
with principal eigenfunction $\varphi\in C^\infty_0(\R^n)$. As a consequence, given the center of mass
$$\mathbf{M}:=\frac{1}{H}\sum\limits_{i=1}^H \opt_i,$$
we have
\begin{equation}\label{lambda-delta-gd-centre-masse}
\lambda_{H,\infty}(\alpha,\mu,\opt_1,\ldots,\opt_H)=\lambda_1(\alpha,\mu)-R_H(\mathbf{M})+\rmax.
\end{equation}
\end{proposition}

%%%%%%%%%%%%%%%%%%%%%%%%%%%%%%%%%%%%%%%%%%%%%%%%%%%%%%%%
%%%%%%%%%%%%%%%%%%%%%%%%%%%%%%%%%%%%%%%%%%%%%%%%%%%%%%%%

\section{Effect of a third host \label{sec:H3vsH2}}

As seen in Proposition~\ref{prop bound lambda}, the inequality $\lambda_1(\alpha,\mu)=\lambda_1\leq \lambda_2=\lambda_2(\delta,\alpha,\mu,\opt_1,\opt_2)$ is always true (as is the inequality $\lambda_1\le\lambda_H$ for any $H\ge2$), meaning that the presence of a second host always penalises the population. In this section, given $\delta>0$, $\alpha>0$, $\mu>0$, we investigate how the introduction of a third host affects the chances of survival of the population, compared to the case of only two hosts. To do so, we denote $\lambda_2=\lambda_2(\opt_1,\opt_2)$ the principal eigenvalue for two hosts with optima $\opt_1$ and $\opt_2$ as in \eqref{O1-O2}, and~$\lambda_3=\lambda_3(\opt_1,\opt_2,\opt_3)$ the principal eigenvalue for three hosts with optima $\opt_1$, $\opt_2$ and~$\opt_3$. Our goal is to compare $\lambda_2(\opt_1,\opt_2)$ and $\lambda_3(\opt_1,\opt_2,\opt_3)$. As revealed in the following, the situation is very rich and the outcomes dramatically depend not only on the different parameters but also the position of the third optimum $\opt_3$. In the following, we denote
\begin{equation}\label{def:Lambda}
\Lambda_\delta:=\big\{\opt _3 \in \R^n: \lambda_3(\opt_1,\opt _2,\opt _3)<\lambda_2(\opt_1,\opt_2)\big\}
\end{equation}
the region in the phenotypical space where the presence of a third host corres\-ponding to an optimum $\opt_3$ causes a springboard effect, leading to higher chances of persistence of the pathogen, compared to an environment with the two hosts with optima~$\opt_1$ and~$\opt_2$.

We first deal in Section~\ref{ss:small-and-large} with the case of small or large migration regimes ($\delta$ small or large), while Section~\ref{sec42} provides some comparisons which hold for any given $\delta>0$. Some complementing numerical simulations are presented in Section~\ref{sec43}.

%%%%%%%%%%%%%%%%%%%%%%%%%%%%%%%%%%%%%%%%%%%%%%%%%%%%%%%%

\subsection{Large and small migration regimes}\label{ss:small-and-large}

%%%%%%%%%%%%%%%%%%%%%%%%%%%%%%%%%%%%%%%%%%%%%%%%%%%%%%%%

\subsubsection*{Large migration regime}

Let us first consider the regime corresponding to the  limit case $\delta\to +\infty$. We denote
\begin{equation}\label{defLambdainfty}
\Lambda_\infty:=\{\opt_3\in \R^n: \lambda_{3,\infty}(\opt_1,\opt_2,\opt_3)<\lambda_{2,\infty}(\opt_1,\opt_2)\},
\end{equation}
where the quantities $\lambda_{H,\infty}$ are defined in Proposition~\ref{prop:lambda-infini} (there, the dependence of $\lambda_{H,\infty}$ on $\alpha$ and $\mu$ was also emphasised; here $\alpha>0$ and $\mu>0$ are fixed, and we are mainly interested in the dependence with respect to the position of the optima, this is why we use the notations $\lambda_{3,\infty}(\opt_1,\opt_2,\opt_3)$ and $\lambda_{2,\infty}(\opt_1,\opt_2)$). We also recall that $\lambda_1=\mu n\sqrt{\alpha}/2-\rmax$ is independent of $\delta$.

For two hosts, the center of mass is $\mathbf{M}_2=(\opt_1+\opt_2)/2=(0,\ldots,0)=\mathcal O$, and the effective fitness defined in Proposition~\ref{prop:lambda-infini} is $ R_2(\x)=(r_1(\x)+r_2(\x))/2=\rmax-\alpha\Vert \x\Vert ^2/2-\alpha\beta^2/2$. As a result
$$\lambda_{2,\infty}(\opt_1,\opt_2)=\lambda_1-R_2(\mathbf M_2)+\rmax=\lambda_1+\frac{\alpha\,\beta^2}{2}.$$

For three hosts, the third optimum is denoted $\opt_3=(a_1,\ldots,a_n)$, and the effective fitness of Proposition \ref{prop:lambda-infini} is
$$R_3(\x)=\frac{r_1(\x)+r_2(\x)+r_3(\x)}{3}=\rmax-\frac{\alpha}{2}\Vert \x -\M_3\Vert ^2-\frac \alpha 2 \left(\frac 23\beta ^2+\frac 2 9(a_1^2+\cdots+a_n^2)\right),$$
where $\M_3:=\frac 13 \opt _3$ is the center of mass of the triangle $\opt_1\opt_2 \opt_3$. As a result
$$\lambda_{3,\infty}(\opt_1,\opt_2,\opt_3)=\lambda_1-R_3(\M_3)+\rmax=\lambda_1+\frac \alpha 2 \left(\frac 23 \beta^2+\frac 29(a_1^2+\cdots+a_n^2)\right).$$

From this, we immediately identify the asymptotic region (as $\delta \to +\infty$) where the third host increases the chances of persistence (see right column of Figure~\ref{fig:diff scenar}).

\begin{proposition}[Region leading to higher chances of persistence for large migration rate]\label{Alfaro balls}
In the phenotypical space $\R^n$, the region $\Lambda_\infty$ defined by~\eqref{defLambdainfty} is the open Euclidean ball $B(\mathcal{O},\sqrt{3/2}\,\beta)$ centered at the origin $\mathcal O$ and of radius $\sqrt{3/2}\,\beta$.
\end{proposition}

Notice that both $\opt_1$ and $\opt _2$ belong to $\Lambda_\infty$. On the other hand, $\opt_3=(0,\pm \sqrt 3 \beta,0,\ldots,0)$, which makes $\opt_1\opt_2\opt_3$ an equilateral triangle, does not belong to $\Lambda_\infty$. This, as revealed below (see Proposition~\ref{prop equilateral}), is actually true for all $\delta>0$.

\begin{remark}{\rm Consider the ``host in the middle'' situation, that is $\opt_3=\mathcal{O}=(0,\ldots,0)$ so that $\lambda_{3,\infty}(\opt_1,\opt_2,\mathcal{O})=\lambda_1+\alpha\beta^2/3$. We now introduce a fourth host whose optimum is denoted $\opt _4=(b_1,\ldots,b_n)$. As above we compute
$$\lambda_{4,\infty}(\opt_1,\opt_2,\mathcal{O},\opt_4)=\lambda_1+\frac \alpha 2 \left(\frac 12 \beta ^2+\frac{3}{16}(b_1^2+\cdots+b_n^2)\right).$$
Then the region where $\lambda_{4,\infty}(\opt_1,\opt_2,\mathcal{O},\opt_4)<\lambda_{3,\infty}(\opt_1,\opt_2,\mathcal{O})$ is the ball $B(\mathcal{O},\sqrt{8/9}\,\beta)$ and, this time, $\opt_1$ and $\opt_2$ do not belong to this ball.}
\end{remark}

%%%%%%%%%%%%%%%%%%%%%%%%%%%%%%%%%%%%%%%%%%%%%%%%%%%%%%%%

\subsubsection*{Small migration regime}

Let us fix $\alpha>0$, $\mu>0$, $\beta> 0$, $\opt_1$ and $\opt_2$ as in~\eqref{O1-O2}, and let us now investigate $\Lambda_\delta$ given in~\eqref{def:Lambda} for $0<\delta \ll 1$. For given positions $\opt_1,\opt_2,\opt_3$, using the  notations $\lambda_3=\lambda_3(\delta)$ and $\lambda_2=\lambda_2(\delta)$, we focus on the difference $\lambda_3(\delta)-\lambda_2(\delta)$. As mentioned above, by~\eqref{deflambdaH0}, $\lambda_3(0)-\lambda_2(0)=\lambda_1-\lambda_1=0$, and $\lambda_3(\delta)-\lambda_2(\delta)\to\lambda_1-\lambda_1=0$ by Proposition~\ref{prop bound lambda}. Hence, if we define $\Lambda_0$ as in~\eqref{def:Lambda} with $\delta=0$, we have $\Lambda_0=\emptyset$. Thus, the quantity $\lambda_3'(0)-\lambda_2'(0)$, if it exists, would be an important information. For a general $H\ge2$, we denote $\lambda_H=\lambda_H(\delta)$, all other parameters $\alpha>0$, $\mu>0$ and $(\opt_i)_{1\le i\le H}\in(\R^n)^H$ being fixed.

\begin{proposition}[Region leading to higher chances of persistence for small migration rate]\label{prop:fitness-gain-small-delta}
For any $H\ge2$, the function $\delta\mapsto\lambda_H(\delta)$ is differentiable in $[0,+\infty)$, with $\lambda'_H(\delta)<1$ for all $\delta\ge0$. Furthermore, either the optima $\opt_i$ are all identical and $\lambda_H(\delta)=\lambda_1=\mu n\sqrt{\alpha}/2-\rmax$ for all $\delta\ge0$, or~$\lambda'_H(\delta)>0$ for all $\delta\ge0$.\footnote{Notice that the inequalities $0\le\lambda_H'(0)<1$ are coherent with Corollary~\ref{cor bound lambda bis}.} Lastly,
\begin{equation}\label{formulelambda'}
\lambda_H'(0)=\frac{H-\mu_A}{H-1},
\end{equation}
where
$$\mu_A:=\max_{\qg=(q_1,\ldots,q_H)\in\R^H,\,\|\qg\|=1}\sum_{1\le i,j\le H}a_{ij}q_iq_j$$
is the largest eigenvalue of the $H\times H$ symmetric matrix $A=(a_{ij})_{1\le i,j\le H}$ defined by
$$a_{ij}:=e^{-\sqrt{\alpha}\|\opt_i-\opt_j\|^2/(4\mu)}>0.$$
In particular, with $\opt_1$ and $\opt_2$ as in~\eqref{O1-O2}, there holds
\begin{equation}\label{patatoide}\begin{array}{rcl}
\mathcal{P} & := & \{\opt_3\in\R^n:\lambda'_3(0)<\lambda'_2(0)\}\vspace{3pt}\\
& = & \ds\Big\{\opt_3\in\R^n:k(\opt_3):=g(\opt_3)\cos\Big(\frac{\arccos(h(\opt_3)\,g(\opt_3)^{-3})}{3}\Big)>1\Big\}\end{array}
\end{equation}
with
$$g(\opt_3):=\sqrt{\frac{1+2h(\opt_3)\cosh(\sqrt{\alpha}\,a_1\beta/\mu)}{3}},\ \ h(\opt_3):=e^{\sqrt{\alpha}(3\beta^2-\|\opt_3\|^2)/(2\mu)},$$
and $\opt_3=(a_1,\ldots,a_n)$. Thus, if $\opt_3\in \mathcal P$, then there is $\delta_0>0$ such that $\opt_3\in \Lambda_\delta$ for every $0<\delta<\delta_0$, with $\Lambda_{\delta}$ as in~\eqref{def:Lambda}. On the other hand, if $\opt_3\not\in\overline{\mathcal P}$, then there is $\delta_0>0$ such that $\opt_3\not \in \Lambda_\delta$ for every $0<\delta<\delta_0$.
\end{proposition}

Let us now briefly comment on the properties of $\mathcal P$ given in~\eqref{patatoide}. First of all, since the function $k$ is continuous in $\R^n$ and converges to $\cos(\pi/6)/\sqrt{3}=1/2<1$ as~$\|\opt_3\|\to+\infty$, the set $\mathcal{P}$ is open and bounded. Furthermore, it is symmetric with respect to the hyperplane $\{0\}\times\R^{n-1}$ and axisymmetric with respect to the axis $\R\times\{0\}^{n-1}$. When~$\opt_3=(0,\sqrt{3}\beta,0,\ldots,0)$, then $\opt_1\opt_2\opt_3$ is an equilateral triangle, and $k(\opt_3)=1$, hence $\opt_3\in\partial\mathcal{P}$. We will actually prove in Proposition~\ref{prop equilateral} below that, for any $\delta>0$, $\opt_3=(0,\sqrt 3 \beta,0,\ldots,0)\not\in\Lambda_\delta$. But the previous observation shows that, for small $\delta$, it was so close!

On the other hand, for an arbitrary point $\opt_3=(a_1,\ldots,a_n)\in\R^n$, formula~\eqref{formulelambda'} implies that $\lambda'_2(0)=1-e^{-\sqrt{\alpha}\beta^2/\mu}$ and, with $\qg:=(1/\sqrt{3},1/\sqrt{3},1/\sqrt{3})$,
$$\lambda'_3(0)\le\frac{3-\qg A\qg}{2}=1-\frac{e^{-\sqrt{\alpha}\beta^2/\mu}+2\,e^{-\sqrt{\alpha}(\beta^2+\|\opt_3\|^2)/(4\mu)}\cosh(\sqrt{\alpha}a_1\beta/(2\mu))}{3}.$$
Therefore,
$$\mathcal{P}\ \supset\ \Big\{\opt_3\in\R^n:e^{-3\sqrt{\alpha}\beta^2/(4\mu)}<e^{-\sqrt{\alpha}a_1^2/(4\mu)}\cosh\Big(\frac{\sqrt{\alpha}a_1\beta}{2\mu}\Big)\,e^{-\sqrt{\alpha}(a_2^2+\cdots+a_n^2)/(4\mu)}\Big\}.$$
By studying the variations of the function $a_1\mapsto e^{-\sqrt{\alpha}a_1^2/(4\mu)}\cosh(\sqrt{\alpha}a_1\beta/(2\mu))$, one infers that
$$\mathcal{P}\supset[-\sqrt{3}\beta,\sqrt{3}\beta]\times\{0\}^{n-1},$$
hence $\opt_1\in\mathcal{P}$ and $\opt_2\in\mathcal{P}$. Lastly,
$$\mathcal{P}\supset\{0\}\times B'(\mathcal{O}',\sqrt{3}\beta),$$
where $B'(\mathcal{O}',\sqrt{3}\beta)$ is the open Euclidean ball of radius $\sqrt{3}\beta$ and centered at the origin $\mathcal{O}':=(0,\ldots,0)$ in $\R^{n-1}$.

%%%%%%%%%%%%%%%%%%%%%%%%%%%%%%%%%%%%%%%%%%%%%%%%%%%%%%%%

\subsection{The general case $\delta>0$}\label{sec42}

In this subsection, we fix the parameters $\delta>0$, $\alpha>0$, and $\mu>0$. We recall that, without loss of generality, $\opt_1$ and $\opt_2$ are as in~\eqref{O1-O2} with $\beta\ge0$, and that $\Lambda_\delta$ denotes the phenotypical region where the presence of $\opt_3$ leads to higher chances of persistence, as defined in \eqref{def:Lambda}.

%%%%%%%%%%%%%%%%%%%%%%%%%%%%%%%%%%%%%%%%%%%%%%%%%%%%%%%%

\subsubsection*{Some fitness loss situations}

We show that if the third optimum $\opt_3$ is such that $\opt_1\opt_2\opt_3$ forms an equilateral triangle, then $\lambda_2-\lambda_3$ is always negative (reduced chances of persistence). Since, by uniqueness of the principal eigenpair of~\eqref{eq:eigenvalue_pb}, the map $\opt_3\mapsto\lambda_3(\opt_1,\opt_2,\opt_3)$ is invariant by rotation around the line $(\opt_1\opt_2)$ containing $\opt_1$ and $\opt_2$, we can assume without loss of generality that $\opt_3$ lies in $\R^2\times\{0\}^{n-2}$, that is, $\opt_3=(0,\pm\sqrt{3}\beta,0,\ldots,0)$.

\begin{proposition}[Equilateral triangle configuration]\label{prop equilateral}
Assume that $\delta>0$, $\alpha>0$, $\mu>0$ are fixed, together with $\opt_1$ and $\opt_2$ as in~\eqref{O1-O2} with $\beta>0$. If $\opt_3= (0, \sqrt 3 \beta,0, \ldots,0)$ or~$\opt_3= (0, -\sqrt 3 \beta,0, \ldots,0)$, then $\lambda_3(\opt_1,\opt_2,\opt_3)>\lambda_2(\opt_1,\opt_2)$, hence $\opt_3\not \in\overline{\Lambda_\delta}$.
\end{proposition}

Now, still for any fixed $\delta>0$, $\alpha>0$, $\mu>0$, and $\opt_1$, $\opt_2$ as in~\eqref{O1-O2} with $\beta\ge0$, we aim at describing the behaviour of $\lambda_3$ when the third optimum $\opt_3$ is far away from the two other optima. To do so, we introduce an auxiliary eigenvalue problem with two hosts, corresponding to a situation with {\it migration loss}: half of the individuals are lost during the migration process. Namely, we define the principal eigenvalue $\tilde{\lambda}_2(\opt_1,\opt_2)$ satisfying
\begin{equation}\label{eq:eigenvalue_pb_loss}
-\frac{\mu^2}{2}\Delta \tilde \varphi_i(\x)-r_i(\x)\,\tilde\varphi_i-\delta \lp \frac{1}{2}\tphi_j(\x)-\tphi_i(\x)\rp=\tilde{\lambda}_2(\opt_1,\opt_2)\,\tilde\varphi_i,
\end{equation}
for $(i,j)=(1,2)$ and $(i,j)=(2,1)$ with $( \tphi_1, \tphi_2) \in\big(H^1(\R^n)\cap L^2_w(\R^n)\cap C^\infty_0(\R^n)\cap L^1(\R^n)\big)^2$ the corresponding normalised principal eigenvector. Note the factor $1/2$ in front of $\tphi_j$ which describes the migration loss.

Before going further on, we note that an evaluation of the Rayleigh quotient associated with $\tilde{\lambda}_2(\opt_1,\opt_2)$ at $(G_1,0)$ (with $G_1=G(\cdot-\opt_1)$ the principal eigenfunction associated with $\lambda_1$) implies that $\tilde{\lambda}_2(\opt_1,\opt_2)\le\lambda_1+\delta$. Additionally, multiplying~\eqref{eq:eigenvalue_pb_loss} by~$\tphi_1$ for $(i,j)=(1,2)$ (resp. by~$\tphi_2$ for~$(i,j)=(2,1)$), integrating by parts and adding the two equations, and using this time the Rayleigh formula associated with $\lambda_2(\opt_1,\opt_2)$, we observe that $\tilde{\lambda}_2(\opt_1,\opt_2)\ge \lambda_2(\opt_1,\opt_2)+\delta \int_{\R^n}\tphi_1\,\tphi_2>\lambda_2(\opt_1,\opt_2)$ since $\delta>0$ and $\tphi_1>0$, $\tphi_2>0$ in~$\R^n$. Finally,
\begin{equation} \label{eq:ineglambda2tilde_lambda1}
\lambda_2(\opt_1,\opt_2)<\tilde{\lambda}_2(\opt_1,\opt_2)\le\lambda_1+\delta.
\end{equation}

When the third optimum ``tends to infinity'', we expect the system to resemble this case of two hosts with migration loss. Indeed, roughly speaking, the third host disappears from the system taking with him migrants arriving from hosts 1 and 2, thus acting as a {\it well} or a {\it black hole}. Precisely, the following holds.

\begin{proposition}[Far away third optimum] \label{prop:byebyeO3}
Assume that $\delta>0$, $\alpha>0$, $\mu>0$ are fixed, together with $\opt_1$ and $\opt_2$ as in~\eqref{O1-O2} with $\beta\ge0$. For $\|\opt_3\|$ large enough, there holds
\begin{equation}\label{loin}
\tilde{\lambda}_2(\opt_1,\opt_2)\ge\lambda_3(\opt_1,\opt_2,\opt_3)\ge\tilde{\lambda}_2(\opt_1,\opt_2)-\frac{\delta\,\sqrt{48\,(\tilde{\lambda}_2(\opt_1,\opt_2)+\rmax)}}{\sqrt{\alpha} \,\min(\norm{\opt_3-\opt_1},\norm{\opt_3-\opt_2})},
\end{equation}
so that $\lambda_3(\opt_1,\opt_2,\opt_3)\to\tilde{\lambda}_2(\opt_1,\opt_2)$ as $\|\opt_3\|\to+\infty$, and $\lambda_3(\opt_1,\opt_2,\opt_3)>\lambda_2(\opt_1,\opt_2)$ for all $\|\opt_3\|$ large enough.
\end{proposition}

%%%%%%%%%%%%%%%%%%%%%%%%%%%%%%%%%%%%%%%%%%%%%%%%%%%%%%%%

\subsubsection*{Some situations where the presence of a third host leads to higher chances of persistence}

In Section \ref{ss:small-and-large}, the set $\Lambda_\delta$ where the presence of $\opt_3$ leads to higher chances of persistence,  defined in \eqref{def:Lambda}, was captured for both regimes $\delta\to +\infty$ and $\delta\to 0^+$. For arbitrary $\delta>0$, a sharp characterisation is very involved (see Proposition \ref{prop:hmid} and the numerical simulations below). Nevertheless, we can provide the following  information, saying that a third host with phenotypic optimum close to that of one of the other two hosts is globally beneficial for the chances of persistence.

\begin{proposition}[When the third host resembles one of the two others]\label{prop:O3=O1}
Assume that $\delta>0$, $\alpha>0$, $\mu>0$ are fixed, together with $\opt_1$ and $\opt_2$ as in~\eqref{O1-O2} with $\beta>0$. Then there is $\rho>0$ such that
$$B(\opt_1,\rho)\cup B(\opt_2,\rho)\subset \Lambda_\delta.$$
\end{proposition}

%%%%%%%%%%%%%%%%%%%%%%%%%%%%%%%%%%%%%%%%%%%%%%%%%%%%%%%%

\subsubsection*{A host in the middle}

Here we consider a different approach: we fix the parameters $\delta>0$, $\alpha>0$, $\mu>0$ and the third optimum $\opt_3$ as $(\opt_1+\opt_2)/2$. In other words, $\opt_3=(0,\ldots,0)=\Oc$, with our hypothesis~\eqref{O1-O2}.  We show that, when the two  optima $\opt_1$ and $\opt_2$ are far enough from each other (meaning $\beta>0$ large), this ``host in the middle'' configuration does not maximise the difference $\lambda_2(\opt_1,\opt_2)-\lambda_3(\opt_1,\opt_2,\opt_3)$, that is, it does not minimise $\lambda_3(\opt_1,\opt_2,\opt_3)$. This result, consistent with the simulations in Figure~\ref{fig:diff scenar}, is slightly surprising (at least at first glance) and shows the richness of the outcomes.

\begin{proposition}[Host in the middle]\label{prop:hmid}
Assume that $\delta>0$, $\alpha>0$ and $\mu>0$ are fixed. For $\beta>0$ large enough in~\eqref{O1-O2}, we have
$$\lambda_3(\opt_1,\opt_2,\Oc)>\lambda_3(\opt_1,\opt_2,\opt_1).$$
More precisely, $\lambda_3(\opt_1,\opt_2,\Oc)\to \lambda_1+\delta$ while $\lambda_3(\opt_1,\opt_2,\opt_1)\to \lambda_1+\delta/2$ as $\beta\to+\infty$.
\end{proposition}

We point out that, since $\lambda_3(\opt_1,\opt_2,\opt_1)=\lambda_3(\opt_1,\opt_1,\opt_2)$ by symmetry, the limit $\lambda_1+\delta/2=\lim_{\beta\to+\infty}\lambda_3(\opt_1,\opt_2,\opt_1)=\lim_{\beta\to+\infty}\lambda_3(\opt_1,\opt_1,\opt_2)$ is coherent with the limit $\lim_{\beta\to+\infty}\lambda_3(\opt_1,\opt_1,\opt_2)=\tilde{\lambda}_2(\opt_1,\opt_1)=\lambda_1+\delta/2$, with $\tilde{\lambda}_2(\opt_1,\opt_1)$ being the principal eigenvalue of~\eqref{eq:eigenvalue_pb_loss} in the case of identical optima (see the proof of Proposition~\ref{prop:hmid} in Section~\ref{sec53} for further details).

%%%%%%%%%%%%%%%%%%%%%%%%%%%%%%%%%%%%%%%%%%%%%%%%%%%%%%%%

\subsubsection*{Decreasing $\lambda_3$ by projecting}

Next, still with $\opt_1$ and $\opt_2$ as in~\eqref{O1-O2} without loss of generality, we show that replacing~$\opt_3$  by its projection on the line $\R\times\{0\}^{n-1}$ containing $\opt_1$ and $\opt_2$ always decreases the principal eigenvalue $\lambda_3$, and thus increases the difference $\lambda_2-\lambda_3$, leading to higher chances of persistence.

\begin{proposition}[Better on the line]\label{prop:projection}
Let $\opt_3\in\R^n$, and let $\opt_3^\sharp$ be its projection on the line $\R\times\{0\}^{n-1}$, with $\opt_1$ and $\opt_2$ as in~\eqref{O1-O2}. Then
$$\lambda_3(\opt_1,\opt_2,\opt_3)\ge \lambda_3(\opt_1,\opt_2,\opt_3^\sharp).$$
\end{proposition}

%%%%%%%%%%%%%%%%%%%%%%%%%%%%%%%%%%%%%%%%%%%%%%%%%%%%%%%%

\subsubsection*{Existence of a best third optimum}

With $\opt_1$ and $\opt_2$ being fixed, the following corollary asserts the existence of a minimum for the function $\opt_3\mapsto\lambda_3(\opt_1,\opt_2,\opt_3)$ defined in $\R^n$, thus ensuring the existence of a point in the phenotypic space maximizing the difference $\lambda_2(\opt_1,\opt_2)-\lambda_3(\opt_1,\opt_2,\opt_3)$.

\begin{corollary}[Minimizing $\lambda_3(\opt_1,\opt_2,\cdot)$]\label{coropt}
Assume that $\delta>0$, $\alpha>0$, $\mu>0$, $\opt_1\in\R^n$ and $\opt_2\in\R^n$ are fixed. Then the function $\opt_3\mapsto\lambda_3(\opt_1,\opt_2,\opt_3)$ has a minimum in $\R^n$, equal to $\opt_1=\opt_2$ if these two optima are identical, or lying in the line $(\opt_1\opt_2)$ if $\opt_1\neq\opt_2$.
\end{corollary}

\begin{proof}
Without loss of generality, one can assume that $\opt_1$ and $\opt_2$ are given as in~\eqref{O1-O2}. On the one hand, if $\beta=0$, then $\opt_1=\opt_2=\mathcal{O}$ and $\lambda_3(\opt_1,\opt_2,\mathcal{O})=\lambda_2(\opt_1,\opt_2)=\lambda_1$, by Proposition~\ref{prop monolambda}, while $\lambda_3(\opt_1,\opt_2,\opt_3)>\lambda_1$ for all $\opt_3\neq\mathcal{O}$, by Proposition~\ref{prop monolambda} again. The conclusion of Corollary~\ref{coropt} is then immediate in this case. On the other hand, if $\beta>0$, the conclusion easily follows from the continuity property in Proposition~\ref{prop monolambda}, together with Propositions~\ref{prop:byebyeO3},~\ref{prop:O3=O1} and~\ref{prop:projection}.
\end{proof}

\begin{remark}{\rm From Proposition~\ref{prop:hmid}, it follows that the middle optimum $(\opt_1+\opt_2)/2$ is not a minimum of the map $\opt_3\mapsto\lambda_3(\opt_1,\opt_2,\opt_3)$ when $\|\opt_1-\opt_2\|$ is large enough. Hence, for all $\|\opt_1-\opt_2\|$ large enough, by symmetry, the map $\lambda_3(\opt_1,\opt_2,\cdot)$ has at least two different minima, of the type $(\pm a_\beta,0,\ldots,0)$, for a certain $a_\beta>0$, under the notation~\eqref{O1-O2}. For all~$\|\opt_1-\opt_2\|$ large enough, these optima maximise the difference $\lambda_2(\opt_1,\opt_2)-\lambda_3(\opt_1,\opt_2,\cdot)$ (which is positive, by Proposition~\ref{prop:O3=O1}).}
\end{remark}

%%%%%%%%%%%%%%%%%%%%%%%%%%%%%%%%%%%%%%%%%%%%%%%%%%%%%%%%

\subsection{Numerical simulations}\label{sec43}

We present in Figure~\ref{fig:diff scenar} a map of the positions of the third optimum $\opt_3$ in the phenotypic plane $\R^2$ with a colorbar picturing the corresponding difference $\lambda_2(\opt_1,\opt_2)-\lambda_3(\opt_1,\opt_2,\opt_3)$, with $\opt_1$ and $\opt_2$ as in~\eqref{O1-O2}. We recall that the difference $\lambda_2(\opt_1,\opt_2)-\lambda_3(\opt_1,\opt_2,\opt_3)$ measures the gain (or loss, if negative) in the chances of persistence, when the third host is added to the system.  All Matlab source code used to generate the analyses performed here is available at \url{https://doi.org/10.17605/OSF.IO/QAV2M}.

\setmysize{height=4.8cm, width=4.8cm}
\begin{figure}
\centering	
 \setlength{\tabcolsep}{0pt}
\begin{tabular}{C{1.cm} C{0.31\textwidth} C{0.31\textwidth} C{0.31\textwidth} }
& $\delta = 0.01$  & $\delta=1$ & $\delta=10$ \\ \\
\rotatebox{90}{$\beta= 0.5$}         &  		 \includegraphics[mysize]{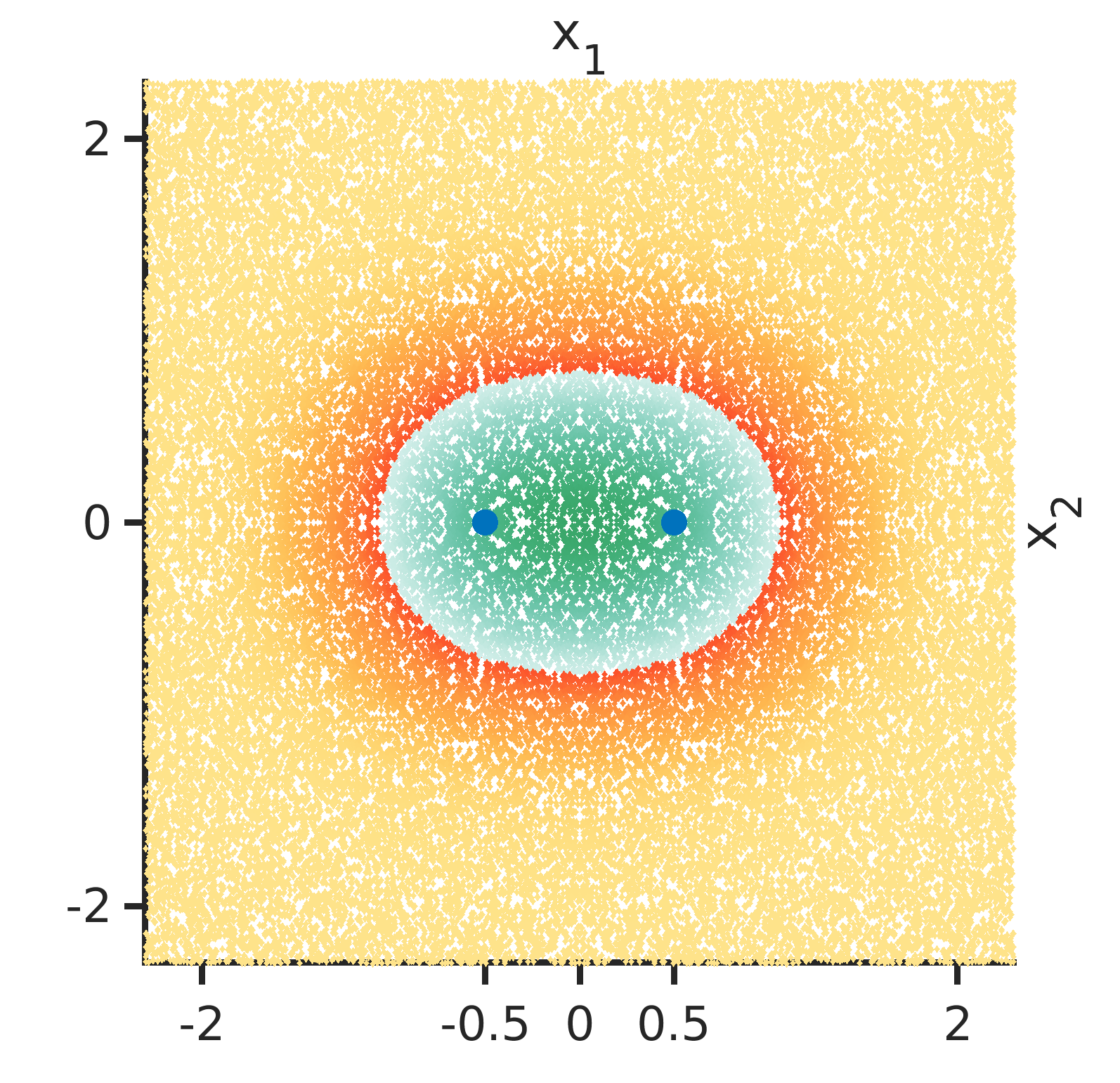}
    &\includegraphics[mysize]{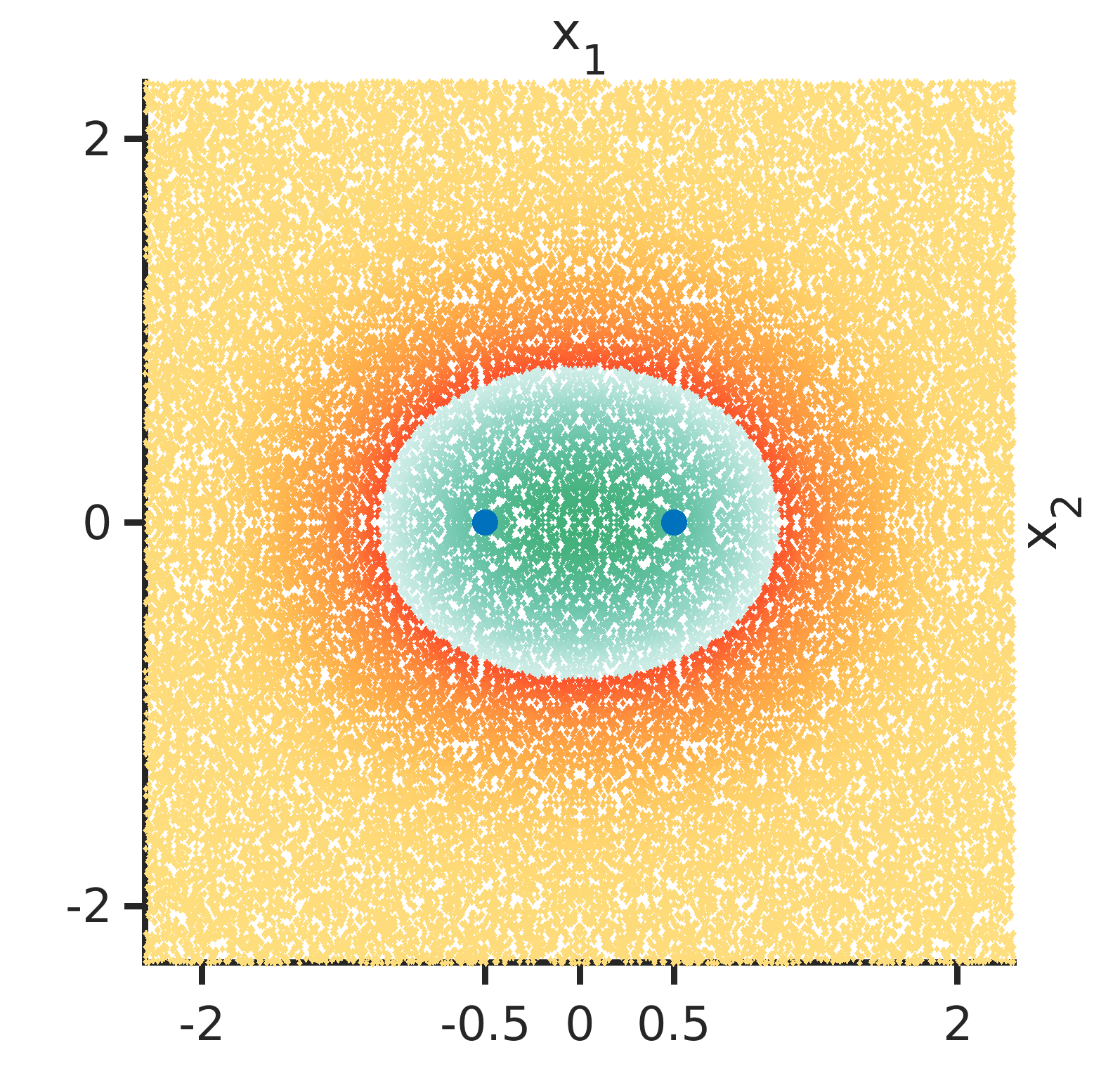} & \includegraphics[mysize]{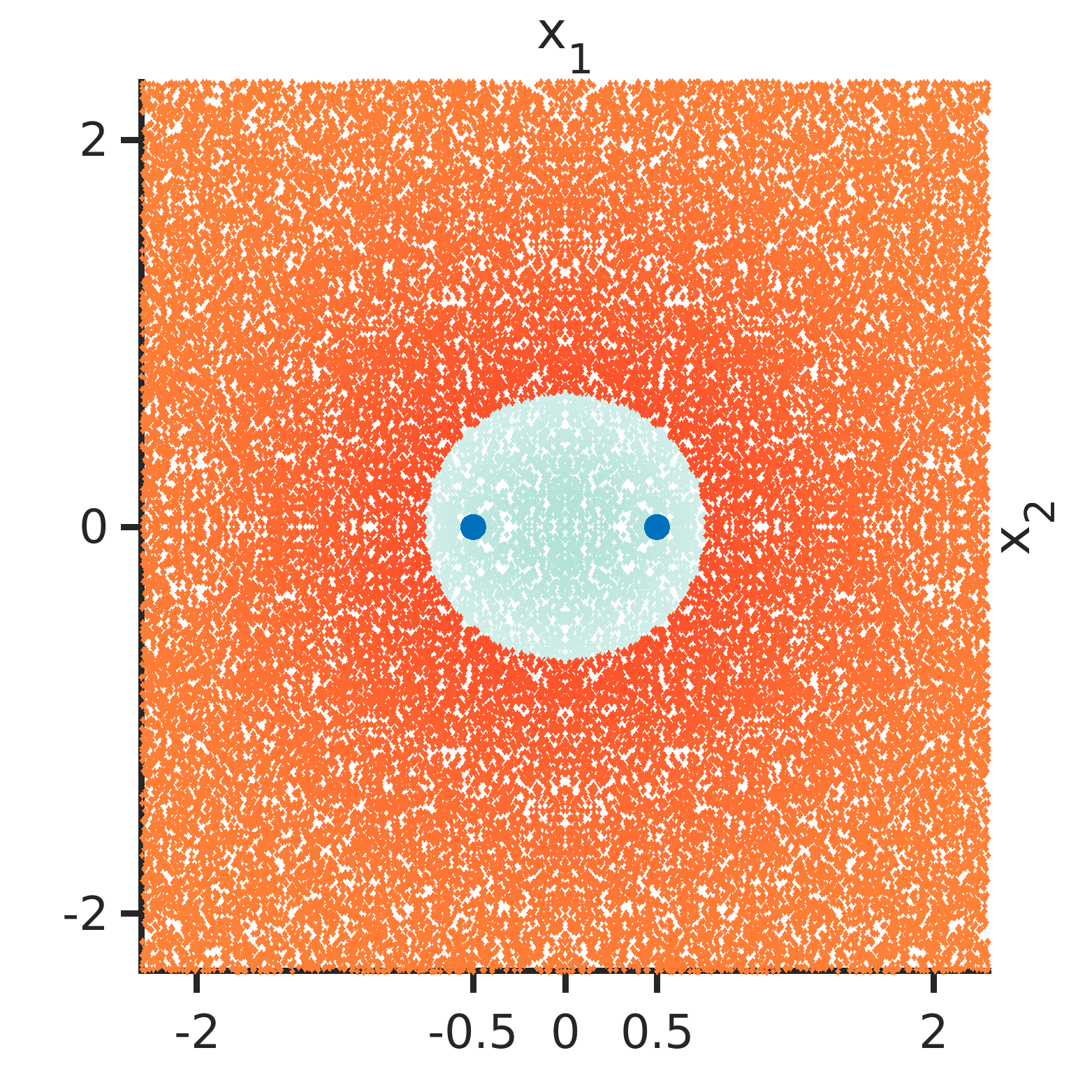}  \\
                            \rotatebox{90}{$\beta= 0.7$}                   &   \includegraphics[mysize]{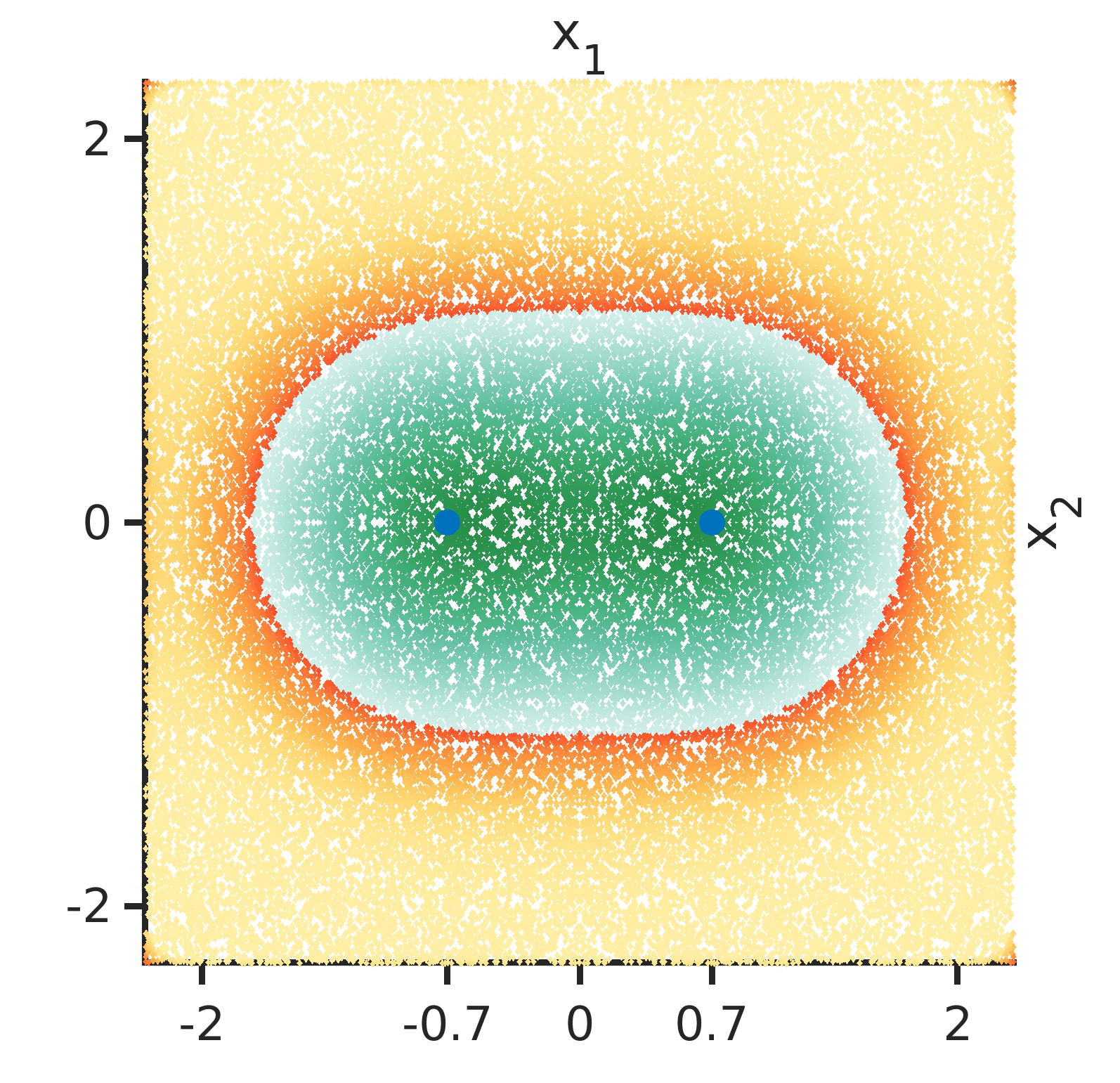}   &  \includegraphics[mysize]{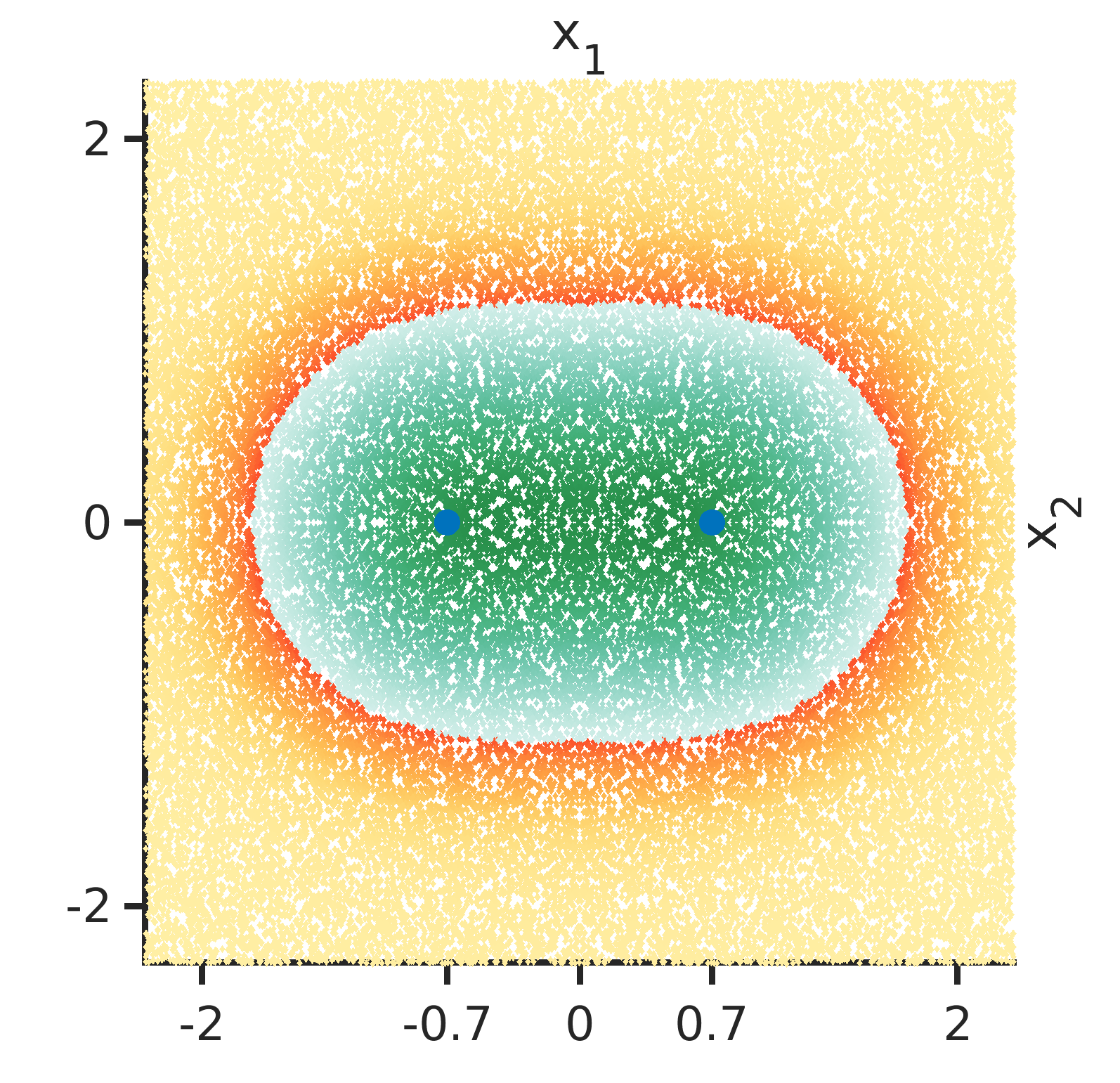} & \includegraphics[mysize]{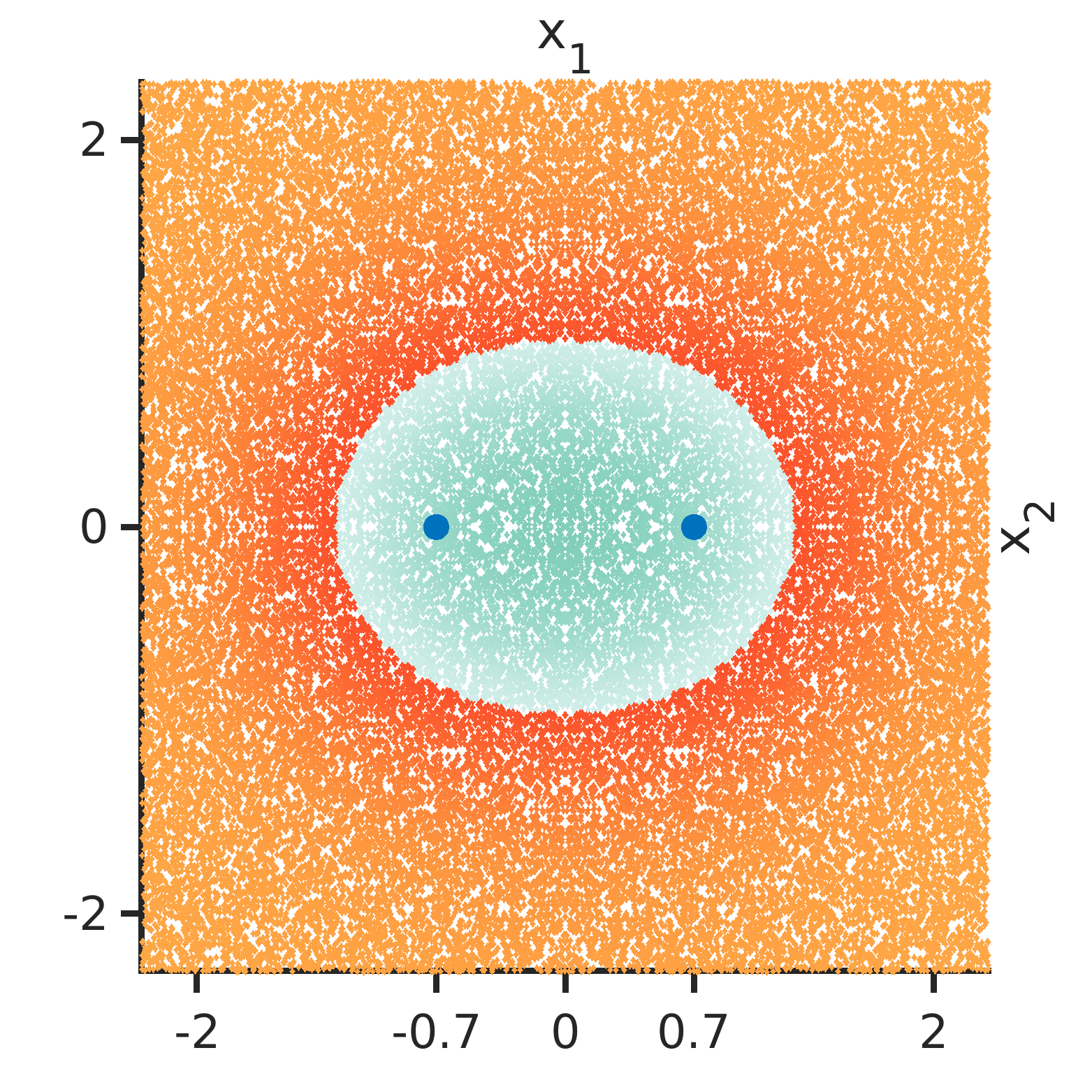}\\
                            \rotatebox{90}{$\beta= 1$}                     & \includegraphics[mysize]{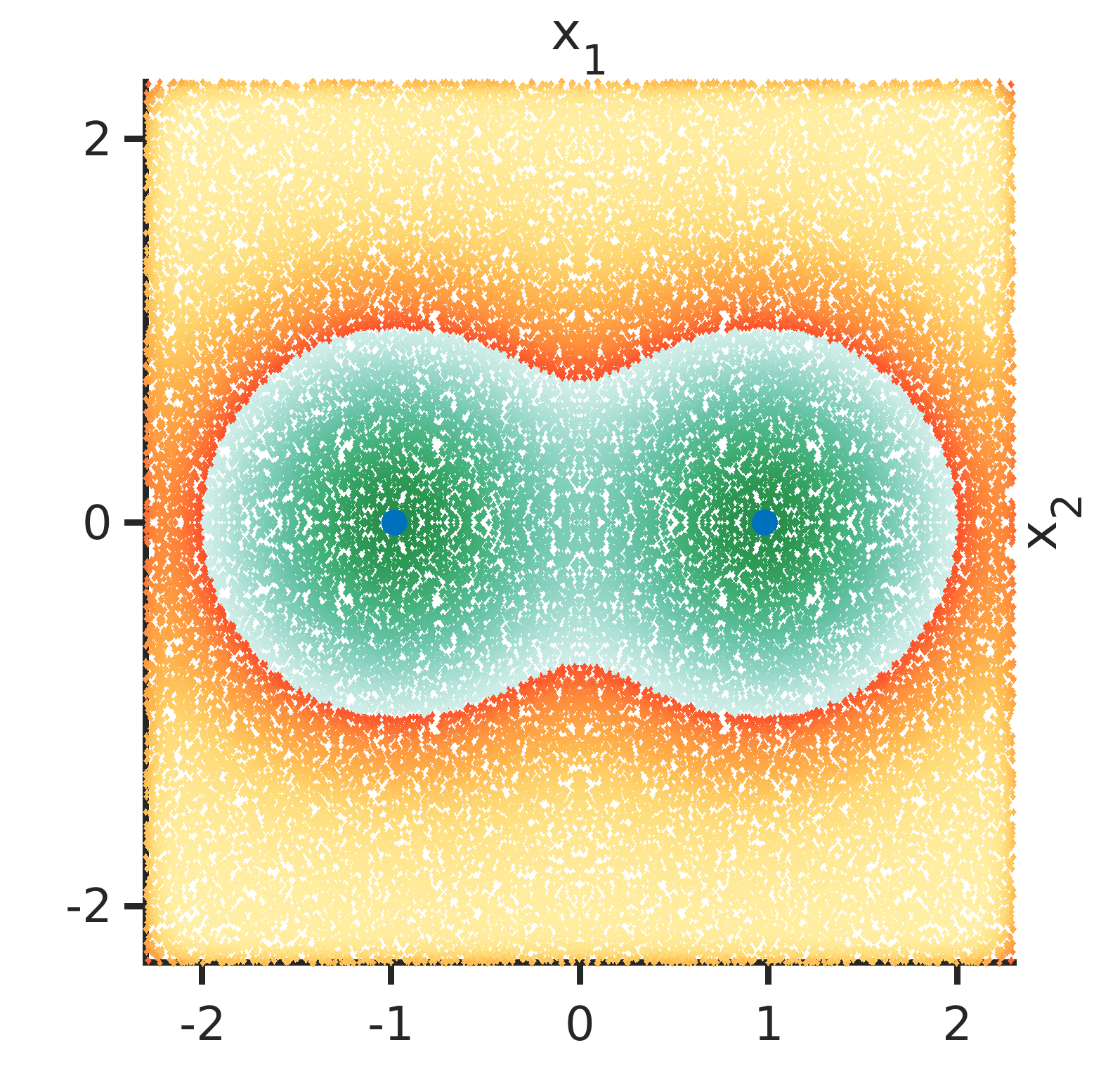}   &    \includegraphics[mysize]{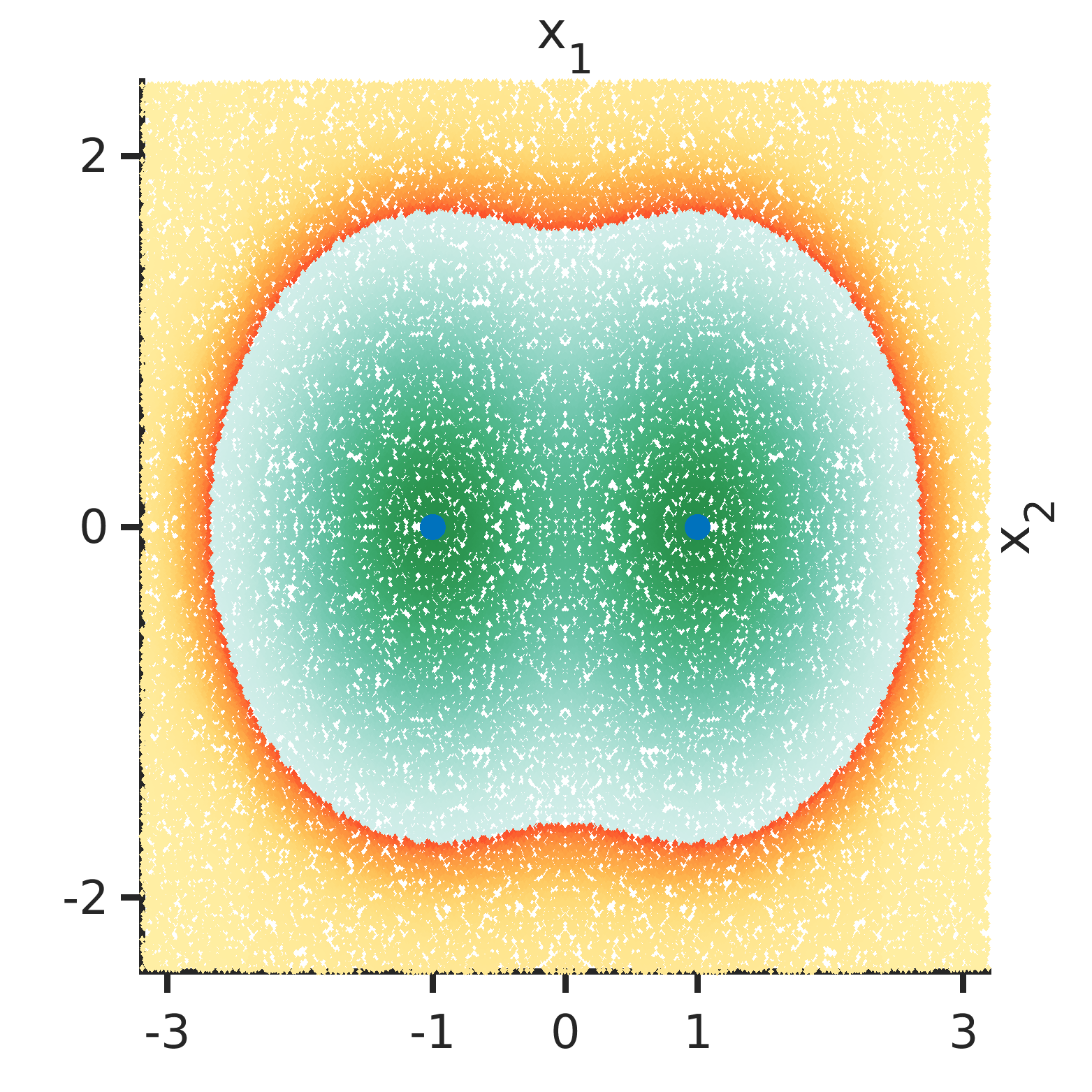}  & \includegraphics[mysize]{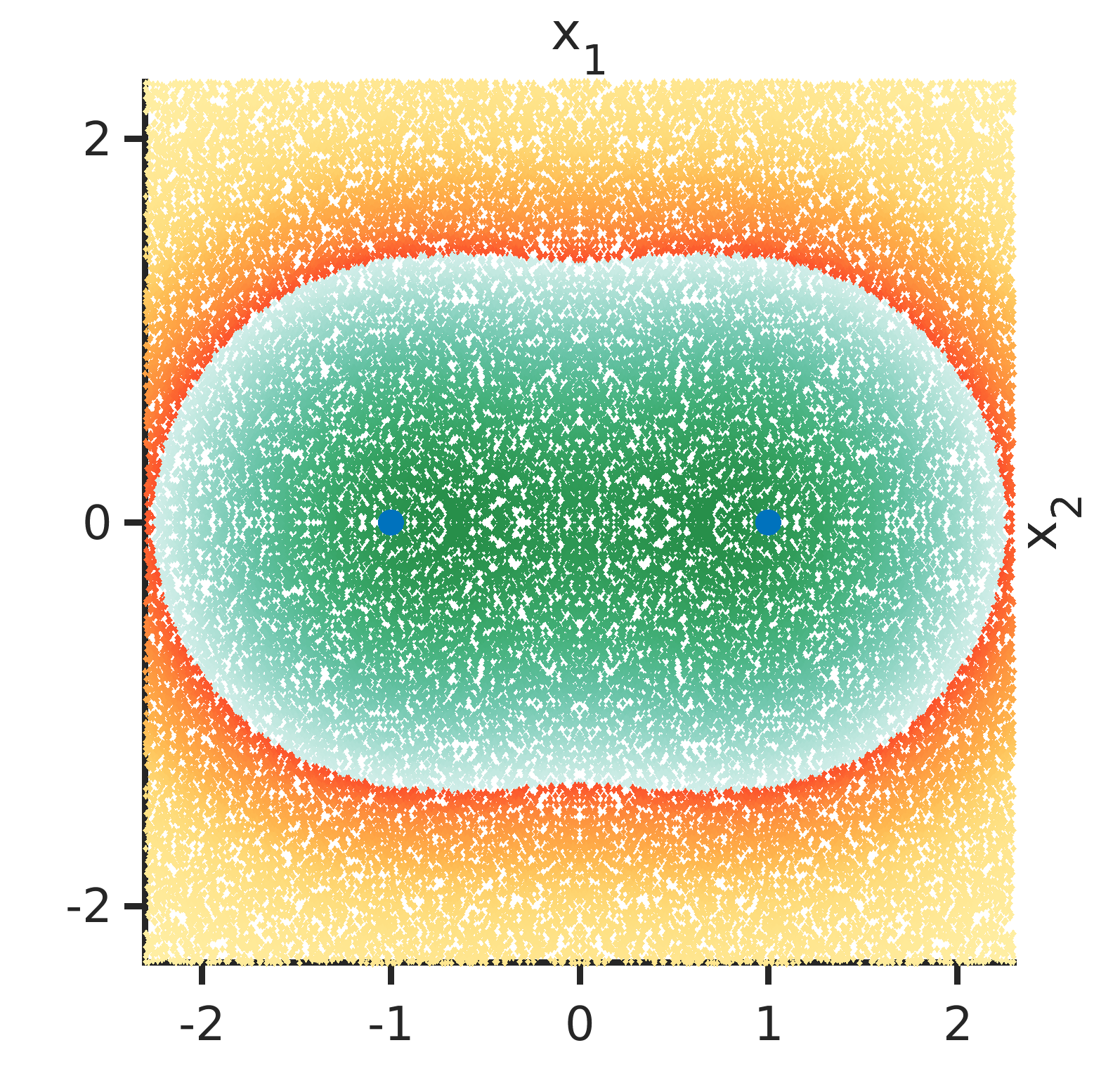} \\
                         \rotatebox{90}{$\beta= 2$}            &     \includegraphics[mysize]{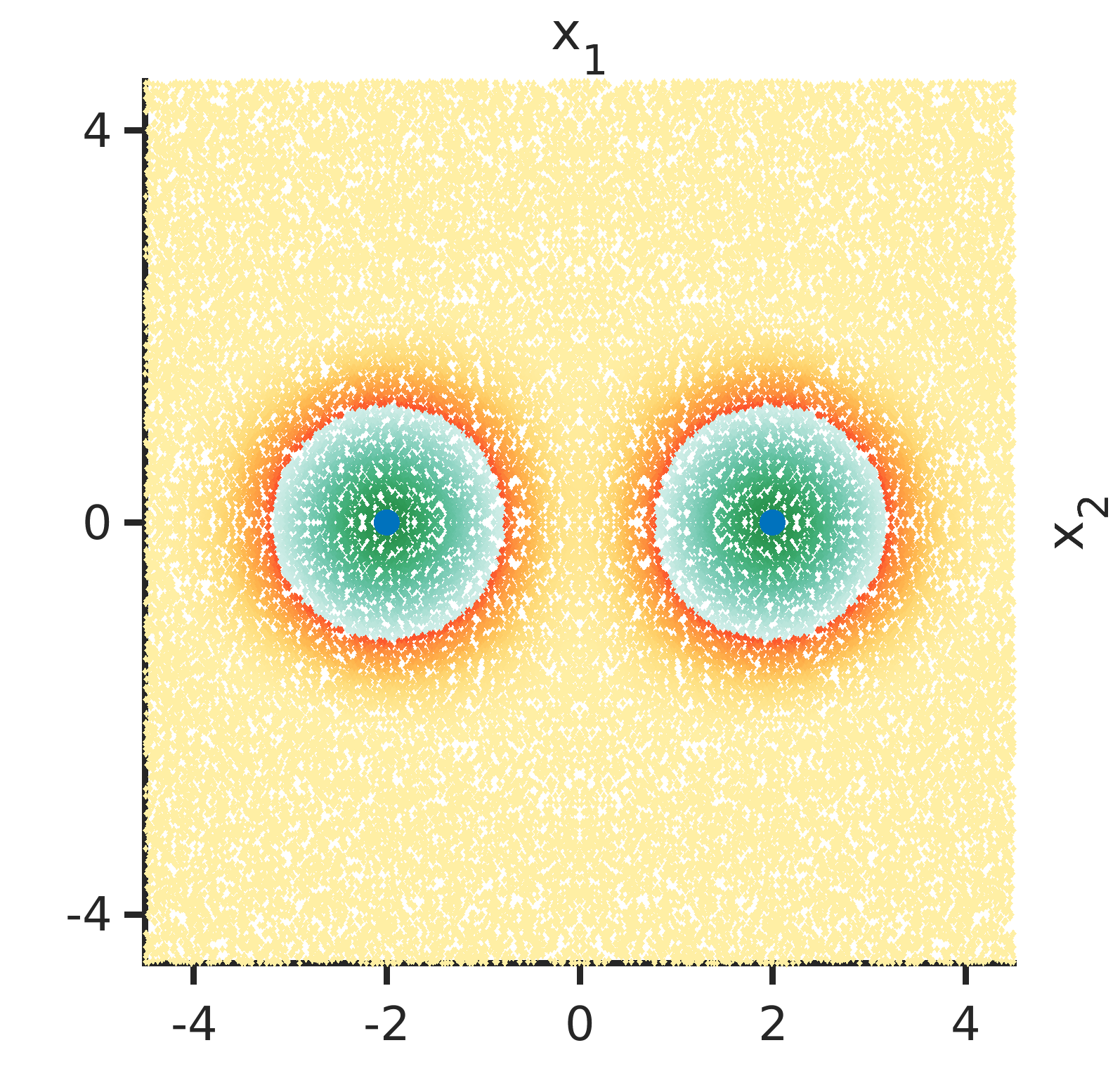} & \includegraphics[mysize]{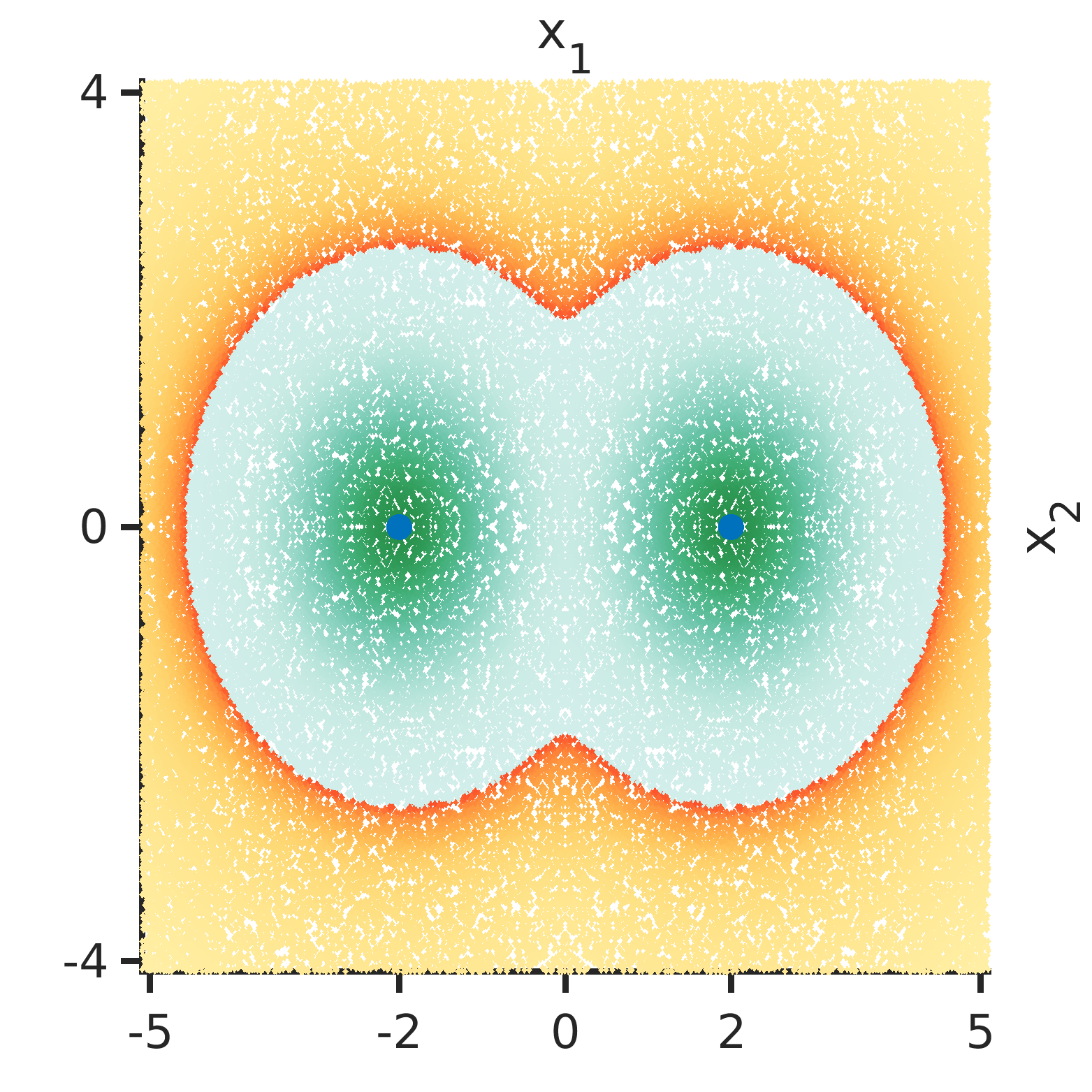} & \includegraphics[mysize]{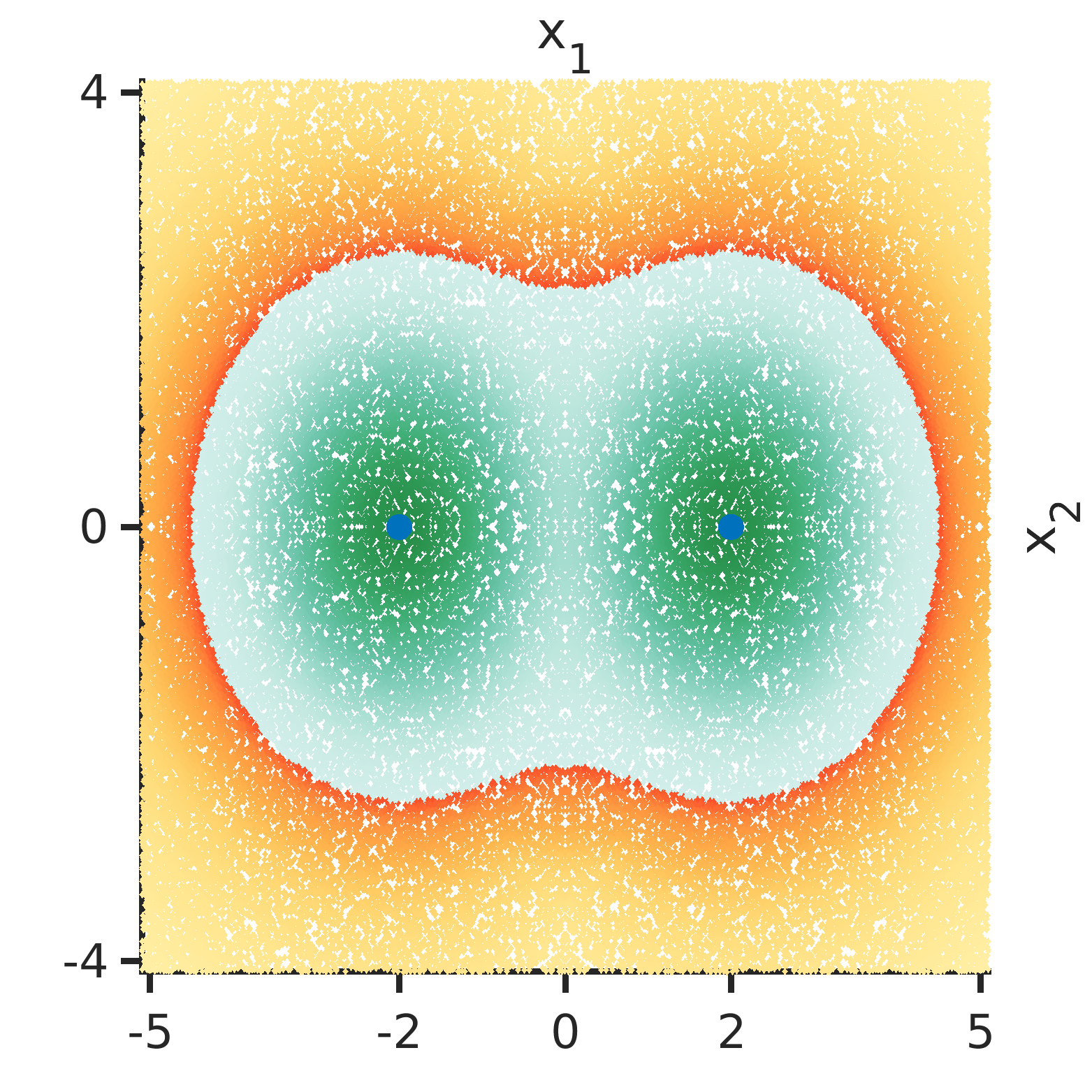}\\
& \multicolumn{3}{c}{\includegraphics[width=0.84\textwidth,height=0.8cm]{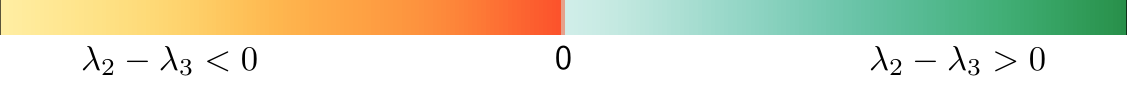}
}
\end{tabular}
\caption{\footnotesize{ \textbf{Effect of adding a third host on persistence.} Each point represent the position of a third optimum, while the two other optima are fixed, at the respective positions $(-\beta,0)$ and $(\beta,0)$, pictured in blue. The colormap corresponds to the value taken by the difference $\lambda_2-\lambda_3$, depending on the position of the third optimum. Each row corresponds to a different value of $\delta$ and each line to a different value of $\beta$. We assumed here that $\mu^2/2 =\alpha=1$. }}
\label{fig:diff scenar}
\end{figure}

This figure  serves as an illustration of the different results of this section. In each figure, the set $\Lambda_\delta$, defined in \eqref{def:Lambda}, corresponds to the green regions. The last column illustrates the behaviour $\delta\to+\infty$ of Proposition~\ref{Alfaro balls}.   In particular, on the first line, for $\beta=1/2$, there is a perfect match between the ball given in the asymptotic regime $\delta\to+\infty$ and the numerical simulations. For the other values of parameter $\beta$ pictured here,  $\delta=10$ seems not large enough to be considered infinite, but one can still observe on the last two lines, going from the left to the right a ``convexification'' of the set~$\Lambda_\delta$, for the lack of the better word, that points in the direction of  the  convergence towards the eventual ball, denoted $\Lambda_\infty$ in Proposition~\ref{Alfaro balls}.

On the other hand, the first column represents the regime $\delta \to 0^+$, detailed in Proposition~\ref{prop:fitness-gain-small-delta}.  In particular, it can be observed that, for small $\beta$, the set $\Lambda_\delta$ has the shape of an ``athletic stadium''. Then, for larger $\beta$, the set $\Lambda_\delta$ becomes ``oxygen molecule shaped'', and thus no longer convex. Interestingly, increasing $\beta$ leads to a loss of connectedness of the set $\Lambda_\delta$, see the picture for $\delta=0.01$ and $\beta=2$. In particular, in agreement with Proposition~\ref{prop:hmid}, this shows that the central point~$(0,0)$ is not optimal in terms of persistence.

Actually, this (numerically observed) loss of connectedness suggests that the introduction of the third host \lq\lq in the middle'' may lead to a decrease in pathogen fitness, compared to the case of only two hosts. One should be careful that this rather surprising result  is not implied by Proposition~\ref{prop:hmid}, which compares the introduction of a third host having an optimum in the middle of the first two hosts with that of a third host having the same optimum as one of the first two hosts. Notice that, as seen in the computations after Proposition~\ref{prop:fitness-gain-small-delta}, $\opt _3=(0,0)$ always belongs to $\mathcal P$ defined in \eqref{patatoide}. However, from the numerical simulations, it is natural to conjecture that $(0,0)\notin \Lambda_\delta$ should occur for some \lq\lq particular combinations of large $\beta$-small $\delta$''  that should depend on $\mu$ (mutations) and $\alpha$ (selection). This makes the proof of this conjecture very challenging.

Finally, we can verify that, as proved in Proposition~\ref{prop:O3=O1}, if the third optimum is very close (or indeed at the same position) than one of the other two optima, then the fitness of the system with three optima is better than the one with two. This corresponds  to the concept of ``evolutionary springboard'', as adaptation to one optimum is favoured by the presence of another one nearby. Incidentally, in accordance with Proposition~\ref{prop:projection}, the closer the third optimum to the axis $(\opt_1\opt_2)$, the larger the chances of persistence. This means that on the $x_2$-axis the third optimum is best positioned at the same phenotypic height than the two fixed optima. On the other axis (the $x_1$-axis), despite the symmetry of the two fixed optima, it is much harder to get qualitative information  as exemplified above by the case where the third host is in the middle.

\section{Discussion}\label{s:discussion}

We proposed a model to describe the adaptation dynamics of a phenotypically structured population in a $H$-patch environment, with each patch associated with a different phenotypic optimum, and performed a rigorous mathematical analysis of this model. This analysis sheds new light on the effect of increasing the number of hosts on the persistence of a (say, pathogenic) population. As is already known with two hosts, a pathogen population will always have more difficulty persisting on multiple hosts connected by migration than on a single host (Proposition~\ref{prop bound lambda}). Furthermore, increasing the migration rate $\delta>0$ further reduces the chances of population persistence (Proposition~\ref{prop monolambda}). However, in some configurations where pathogen persistence is not possible with two hosts, adding a third host achieves persistence. In such case, the presence of a third host causes a springboard effect, by modifying the geometric configuration of the phenotypic space. Thus, as we already know experimentally with the examples of zoonoses mentioned in the Introduction, the increase in the number of hosts does not necessarily mean a decrease in the chances of persistence as long as the initial number of hosts is at least equal to 2. However, the occurrence of this springboard effect, or on the contrary of a detrimental effect of increasing the number of hosts on the persistence of the pathogen, depends in a rather complex way on the respective positions of the optimal phenotypes associated with each host. Our results, which are based on a fixed point theorem,  comparison principles, integral estimates, variational arguments and rearrangement techniques,  provide a better understanding of these dependencies.

In the asymptotic regimes of small and large migration rates, we were able to characterise very precisely the situations where the addition of a third host increases or decreases the chances of persistence of the pathogen, compared to a situation with only two hosts. Proposition~\ref{prop:lambda-infini} shows that when the migration rate is large, the system with  $H$ host behaves as a single population, living in a single host and with ``effective'' fitness equal to the mean of the $H$ original fitness functions. In this case, where the population can be qualified as ``generalist'', we were able to show that the addition of a third host is favourable to persistence if and only if the optimum associated with this host is inside a certain ball containing the two other optima (Proposition~\ref{Alfaro balls}). When the migration rate is low, as shown in Proposition~\ref{prop:fitness-gain-small-delta} (and Figure~\ref{fig:diff scenar}), the shape of the region $\Lambda_\delta$ in the phenotypical space where the presence of a third host leads to higher chances of persistence of the pathogen is more complex and depends on the parameters. In particular, we see that this region is not necessarily convex especially if the two initial optima are far from each other. Moreover, while the addition of a third host always decreases the chances of persistence when the optima associated with the three hosts form an equilateral triangle in the phenotypic space (Proposition~\ref{prop equilateral}), we observe that when migration is low this equilateral configuration is very close to the $\Lambda_\delta$ region. Of course, some of these observations are only possible if the dimension of the phenotype space is at least 2.  This illustrates the importance of taking pleiotropy into account (in the sense that a single mutation can affect multiple phenotypes), as already noted in \cite{WaxPec98} in the case $H=1$ (but with other assumptions on the mutational operator).

In the general case (i.e., with an arbitrary migration rate $\delta$), and when the positions $\opt_1$ and $\opt_2$ of the two other hosts are fixed, we proved the existence of a position in the phenotypic space of the optimum associated with the third host maximizing the chances of pathogen persistence (Corollary~\ref{coropt}). This position is necessarily on the line $(\opt_1\, \opt_2)$ (Proposition~\ref{prop:projection}). However, contrary to a natural intuition, it is not necessarily located in the exact mean position between $\opt_1$ and $\opt_2$ (Proposition~\ref{prop:hmid}). In particular, if the two initial optima $\opt_1$ and $\opt_2$ are far apart, it is more beneficial for the persistence of the pathogen to introduce a third host associated with an optimum $\opt_3$ that is in a neighbourhood of $\opt_1$ or $\opt_2$. This result which is consistent with the simulations in Figure~\ref{fig:diff scenar} shows the richness of the outcomes.

From a mathematical point of view, an important part of the difficulty in dealing with $H\ge 3$ hosts, compared to $H=1$ or $H=2$ hosts, comes from the lack of symmetry. In the $H=2$ case, provided that the initial conditions are themselves symmetric about $(\opt_1+\opt_2)/2$, the solutions $u_1$ and $u_2$ remain symmetric about this point. This type of symmetry argument was used in \cite{HamLavRoq20} to reduce the study of the nonlinear system~\eqref{eq:sys H} to that of a linear scalar equation. Such arguments cannot be applied in the general case $H \ge 3$ where we had to treat the complete nonlinear and nonlocal system. This has led to very different proofs, starting with the well-posedness of the Cauchy problem (Theorem~\ref{thm:well-pos}). The stationary states of the system illustrate well the complexity induced by the absence of symmetry: when there is persistence and $H=1,\, 2$, the stationary states are proportional to the principal eigenfunctions, but this is no longer true in general in the case $H\ge 3$, because of a possible symmetry breaking (Proposition~\ref{prop:statio_states}).  Again because of the lack of  symmetry, the populations at equilibrium have no reason to be equal in each host. Even when $H=2$, the eigenfunctions are not radially symmetric (they are biased towards the other optimum, see formula 46 in \cite{HamLavRoq20}). When $H\ge 3$, the qualitative description of their profile is even more involved, leading to new difficulties in the comparison between different geometric configurations of the positions of the optima. Although we were able to show that the ``best position'' of $\opt_3$ when $\opt_1$ and $\opt_2$ are fixed should belong to the line $(\opt_1\, \opt_2)$, the rearrangement techniques
that we used for this result did not allow to determine the best position exactly. This is because the rearrangement result is only true in the directions where $\opt_1$ and $\opt_2$ are stable by rearrangement (and thus not in the direction $x_1$ with our conventions; otherwise the position $\opt_3=(\opt_1+\opt_2)/2$ would always be optimal, which would contradict the result of the Proposition~\ref{prop:hmid}). Note that the rearrangement inequalities used in the proof of Proposition~\ref{prop:projection} are not standard, because of the  unbounded domain and of the sign changes of $r_i$.

As mentioned in the Introduction, one of the fundamental principles in agroecology is to promote diversified agroecosystems rather than uniform cultures.  Our results show, with a rigorous mathematical analysis, that what is true when moving from a monoculture to a varietal mixture is not always true when moving from a varietal mixture with $H\ge 2$ species to a mixture with $H+1$ species. In such a situation, it is necessary to know the respective positions of the optima associated with the different hosts, from the point of view of the pathogen, to predict the positive or negative impact of diversification. From an experimental point of view, it would therefore be necessary to be able to estimate these positions, by comparing, for example, via cross-inoculation experiments, the effect of pre-adaptation of a pathogen on each host $i$ before its introduction on a host $j$, for each pair $i,j$. Of course, the vision here, where each host is described by a unique optimum (FGM model) and where the different parameters do not depend on the host $i$ is very schematic,  and the comparison with quantitative data would require more complex approaches. However, the model \eqref{eq:sys H} has the advantage of being mathematically tractable while capturing the antagonistic effects of host diversification.

%%%%%%%%%%%%%%%%%%%%%%%%%%%%%%%%%%%%%%%%%%%%%%%%%%%%%%%%
%%%%%%%%%%%%%%%%%%%%%%%%%%%%%%%%%%%%%%%%%%%%%%%%%%%%%%%%

\section{Proofs of the main results}\label{sec:proofs}

This section is devoted to the proofs of all main results of the paper which have not been shown along the way in the previous sections. The results of Section~\ref{sec:existence} are proved in Section~\ref{sec51}, while Section~\ref{sec52} contains the proofs of the results of Section~\ref{sec:effect}, and Section~\ref{sec53} those of Section~\ref{sec:H3vsH2}.

%%%%%%%%%%%%%%%%%%%%%%%%%%%%%%%%%%%%%%%%%%%%%%%%%%%%%%%%

\subsection{Proofs of Theorem~\ref{thm:well-pos} and Propositions~\ref{theo pers log}-\ref{prop:statio_states}}\label{sec51}

\begin{proof}[Proof of Theorem~$\ref{thm:well-pos}$]
Our strategy involves firstly a fixed point argument to construct a unique local-in-time solution, and secondly the derivation of some estimates to conclude that the solution is defined for all time $t\ge0$. One recalls that $\ug^0=(u^0_1,\ldots,u^0_H)$ satis\-fies~\eqref{data}-\eqref{init bound}. Having in mind the assumption~\eqref{init bound} and the desired inequality~\eqref{ineqexp}, we can assume without loss of generality that
$$0<\theta\le\frac{1}{\mu}$$
in~\eqref{init bound}. Set $\rmax^+:=\max(\rmax,0)$, let $\omega_n$ be the $(n-1)$-dimensional Lebesgue measure of the unit Euclidean sphere of $\R^n$, let $B:=\max_{1\le i\le H}\|\opt_i\|$ and
\begin{equation}\label{defC1}
C_1\!:=\!\frac{e}{3}\,\theta^{-n}\omega_n(n-1)!\times\big[|\rmax|+\alpha B^2+\alpha\theta^{-2}n(n+1)+Ke^{\rmax^+}\theta^{-n}\omega_n(n-1)!(e+3)+2\delta\big].
\end{equation}
Denote
\begin{equation}\label{cond T}
T_1:=\Big(1+\frac{\theta^2\mu^2}{2}+\rmax^++3\,\theta^{-n}\,\omega_n\,(n-1)!\,K\,e^{\rmax^++1}+9\,C_1^2\Big)^{-1}\ \in(0,1),
\end{equation}
and
\begin{equation}\label{defE}\begin{array}{rcl}
E & \!\!:=\!\! & \Big\{\ug=(u_1,\ldots,u_H)\in C([0,T_1]\times\R^n,\R^H)\ \hbox{ such that }u_i(0,\cdot)=u_i^0\hbox{ in $\R^n$,}\vspace{3pt}\\
& & \ \ \ \ 0\le u_i(t,x)\le K\,e^{(\rmax+1)t-\theta\norm{\x}}\hbox{ for all }(t,\x)\in[0,T_1]\times\R^n,\vspace{3pt}\\
& & \ds\ \ \ \ \hbox{and }t\mapsto\int_{\R^n}u_i(t,\x)\,\md\x\ \in C^{0,1/2}([0,T_1]),\ \hbox{ for all $1\le i\le H$}\Big\}.\end{array}
\end{equation}
We equip $E$ with the distance
\begin{equation}\label{defdE}\begin{array}{rl}
d_E(\ug,\vg):=\ds\max\Big( & \!\!\!\ds \max_{1\le i\le H}\big\|(t,\x)\mapsto(u_i(t,\x)-v_i(t,\x))\,e^{\theta\|\x\|}\big\|_{L^\infty([0,T_1]\times\R^n)},\vspace{3pt}\\
& \!\!\!\ds\max_{1\le i\le H}\Big\|t\mapsto\int_{\R^n}(u_i(t,\x)-v_i(t,\x))\md\x\Big\|_{C^{0,1/2}([0,T_1])}\Big),\end{array}
\end{equation}
where $\|g\|_{C^{0,1/2}([0,T_1])}=\max_{t\in[0,T_1]}|g(t)|+\sup_{t\neq t'\in[0,T_1]}|g(t)-g(t')|/|t-t'|^{1/2}$ for $g\in C^{0,1/2}([0,T_1])$. The space $(E,d_E)$ is a non-empty complete metric space.

Define now a mapping $\Fm$ as follows:
$$\begin{array}{rcl}
\mathcal{F}:E & \to & E \\
\ug & \mapsto & \Fm(\ug):=\Ug,\end{array}$$
where, for a given $\ug\in E$, $\Ug= (U_1,\ldots,U_H)\in C([0,T_1]\times\R^n,\R^H)\cap C^{1;2}_{t;\x}((0,T_1]\times\R^n,\R^H)$ is the classical bounded nonnegative solution of the linear system
$$\p_t U_i (t,\x)=\frac{\mu^2}{2} \Delta U_i(t,\x)+\!\left(\!r_i(\x)\!-\!\!\int_{\R^n}\!\!u_i(t,\y)\,\md\y\!\right) U_i(t,\x)+\frac{\delta}{H\!-\!1}\sum \limits_{\substack{k=1 \\k \neq i}}^H(U_k(t,\x)\!-\!U_i(t,\x))$$
in $(0,T_1]\times\R^n$, for $1\le i\le H$, with initial condition $U_i(0,\cdot)=u_i^0$ in $\R^n$. Using the H\"older continuity of the maps $t\mapsto\int_{\R^n}u_i(t,\x)\,\md\x$ in $[0,T_1]$, such a solution $\Ug$ exists and is unique, by~\cite{Bes79}.

To show that $\Fm(E)\subset E$, let us first check that, for each $\ug\in E$, the image $\Ug=\Fm(\ug)$ satisfies the pointwise bounds of $E$. Notice that
\begin{equation}\label{ineqU}
\p_t U_i(t,\x)  \leq  \frac{\mu^2}2 \Delta U_i(t,\x)+\rmax U_i(t,\x)  + \frac{\delta}{H-1} \sum \limits_{\substack{k=1 \\k \neq i}}^H (U_k(t,\x)-U_i(t,\x))
\end{equation}
for every $1\le i\le H$ and $(t,\x)\in(0,T_1]\times\R^n$. Take any vector $\xi\in\R^n$ such that $\|\xi\|=1$ and define
\begin{equation}\label{defolUi}
\overline{U}_i(t,\x):=K\,e^{(\rmax+1)t-\theta\xi\cdot\x},\ \ 1\le i\le H,\ \ t\in\R,\ \ \x\in\R^n.
\end{equation}
Notice that the positive functions $\overline{U}_i$ actually do not depend on $i$, and satisfy
\begin{equation}\label{eqolUi}
\p_t\overline{U}_i(t,\x)=\frac{\mu^2}{2}\Delta\overline{U}_i(t,\x)+(\rmax+1)\overline{U}_i(t,\x)-\frac{\theta^2\mu^2}{2}\,\overline{U}_i(t,\x)
\end{equation}
in $\R\times\R^n$. Since we had assumed that $0<\theta\le1/\mu$ without loss of generality, we get that
\begin{equation}\label{ineqU+}
\p_t \overline{U}_i(t,\x)\geq\frac{\mu^2}2 \Delta\overline{U}_i(t,\x)+\rmax\overline{U}_i(t,\x)+\frac{\delta}{H-1} \sum \limits_{\substack{k=1 \\k \neq i}}^H (\overline{U}_k(t,\x)-\overline{U}_i(t,\x))
\end{equation}
for every $1\le i\le H$ and $(t,\x)\in\R\times\R^n$ (recall that $\overline{U}_k-\overline{U}_i\equiv 0$). Moreover, we know from \eqref{init bound} that $U_i(0,\cdot) = u^0_i\leq\overline{U}_i(0,\cdot)$ in $\R^n$. Therefore, by~\eqref{ineqU} and~\eqref{ineqU+}, $\Ug$ and $\overline{\Ug}:=(\overline{U}_1,\ldots,\overline{U}_H)$ are sub- and super-solutions of a linear cooperative system, and the parabolic maximum principle implies that $U_i(t,\x)\leq\overline{U}_i(t,\x)$ for all $1\le i\le H$ and $(t,\x)\in[0,T_1]\times\R^n$, that is, $U_i(t,\x) \leq K\,e^{(\rmax+1) t-\theta\xi\cdot\x}$. Since $\xi$ was arbitrary in the Euclidean unit sphere of~$\R^n$, and remembering the nonnegativity of each $U_i$, we get that
\begin{equation}\label{ineqU2}
0\le U_i(t,\x)\le Ke^{(\rmax+1)t-\theta\|\x\|}
\end{equation}
for all $1\le i\le H$ and $(t,\x)\in[0,T_1]\times\R^n$. The Lebesgue dominated convergence theorem then implies that each map $t\mapsto\int_{\R^n}U_i(t,\x)\md\x$ is well defined and continuous in $[0,T_1]$. It remains to show that each such map is in $C^{0,1/2}([0,T_1])$. Since the functions~$U_i$ and~$r_iU_i$ decay exponentially to $0$ as $\|\x\|\to+\infty$ uniformly in $t\in[0,T_1]$, it follows from standard parabolic estimates that, for each $0<a\le b\le T_1$, the functions $\p_{x_j}U_i$ decay exponentially to $0$ as $\|\x\|\to+\infty$ uniformly in $t\in[a,b]$, for every $1\le i\le H$ and $1\le j\le n$. Therefore, for every $0<t\le t'\le T_1$, by integrating the equation satisfied by~$U_i$ over $[t,t']\times B(\mathcal{O},R)$ and passing to the limit as $R\to+\infty$, one infers that
\begin{equation}\label{integrals}\begin{array}{rcl}
\ds\int_{\R^n}\!\!U_i(t',\x)\md\x-\!\int_{\R^n}\!\!U_i(t,\x)\md\x & \!\!\!=\!\!\! & \ds\int_t^{t'}\Big(\int_{\R^n}r_i(\x)U_i(s,\x)\md\x\Big)\md s\vspace{3pt}\\
& & \ds-\int_t^{t'}\Big(\int_{\R^n}u_i(s,\x)\md\x\Big)\Big(\int_{\R^n}U_i(s,\x)\md\x\Big)\md s\vspace{3pt}\\
& & \ds+\frac{\delta}{H-1}\sum \limits_{\substack{k=1 \\k \neq i}}^H\int_t^{t'}\!\!\!\!\int_{\R^n}\!(U_k(s,\x)-U_i(s,\x))\md\x\,\md s.\end{array}
\end{equation}
By~\eqref{ineqU2} and the definitions of $r_i$ and $E$, the integrals $\int_{\R^n}r_i(\x)U_i(s,\x)\md\x$, $\int_{\R^n}u_i(s,\x)\md\x$ and $\int_{\R^n}U_i(s,\x)\md\x$ are all bounded in absolute value, say by a constant $C$, uniformly in $s\in[0,T_1]$ and $1\le i\le H$. Thus,
\begin{equation}\label{uLip}
\Big|\int_{\R^n}U_i(t',\x)\md\x-\int_{\R^n}U_i(t,\x)\md\x\Big|\le (C+C^2+2\delta C)\,|t-t'|
\end{equation}
for every $1\le i\le H$ and $0<t<t'\le T_1$ (and also with $t=0$ by continuity). As a consequence, the maps $t\mapsto\int_{\R^n}U_i(t,\x)\md\x$ are of class $C^{0,1}([0,T_1])\subset C^{0,1/2}([0,T_1])$. Finally, $\Ug\in\E$.

Next, we show that $\Fm$ is a contraction mapping, that is, there exists $0<\kappa<1$ such that
\begin{align}\label{eq:contractionmapping}
\forall\,(\ug,\vg)\in E\times E,\ \ d_E(\Fm(\ug),\Fm(\vg)) \leq \kappa\,d_E(\ug,\vg).
\end{align}
Take any $\ug,\vg \in E$, and denote $\Ug:= \Fm(\ug)$, $\Vg:=\Fm(\vg)$, $\wg:=\ug-\vg$ and $\Wg:=\Ug-\Vg$. With these notations, for each $1\le i\le H$, the function $W_i$ belongs to $C([0,T_1]\times\R^n,\R^H)\cap C^{1;2}_{t;\x}((0,T_1]\times\R^n,\R^H)$ and solves
%\begin{equation}\label{sys W+}\begin{array}{rcl}
%\p_t W_i (t,\x)  & = & \ds\frac{\mu^2}2 \Delta W_i (t,\x) +  \left( r_i(\x)  -   \int_{\R^n} u_i(t,\y)\md\y \right) W_i (t,\x)  \\
%& & \ds-V_i (t,\x) \int_{\R^n} w_i(t,\y)\md\y + \frac{\delta}{H-1} \sum \limits_{\substack{k=1 \\k \neq i}}^H ( W_k (t,\x) -W_i (t,\x))\end{array}
%\end{equation}
\begin{multline}\label{sys W+}
\p_t W_i (t,\x)   =  \ds\frac{\mu^2}2 \Delta W_i (t,\x) +  \left( r_i(\x)  -   \int_{\R^n} u_i(t,\y)\md\y \right) W_i (t,\x)  \\
 \ds-V_i (t,\x) \int_{\R^n} w_i(t,\y)\md\y + \frac{\delta}{H-1} \sum \limits_{\substack{k=1 \\k \neq i}}^H ( W_k (t,\x) -W_i (t,\x))
 \end{multline}
in $(0,T_1]\times\R^n$, together with $W_i(0,\cdot)=0$ in $\R^n$. Consider any $\xi\in\R^n$ with $\|\xi\|=1$. Our supersolution candidate for this system will be $\overline{\Wg}=(\overline{W}_1,\ldots,\overline{W}_H)$ defined by
$$\overline{W}_i(t,\x):= \frac 13\,d_E(\ug,\vg)\,e^{t/T_1-\theta\xi\cdot\x}, \quad 1\le i\le H,\ 0 <t\leq T_1,\  \x\in\R^n.$$
Notice that the functions $\overline{W}_i$ actually do not depend on $i$. First of all, we recall that~$\omega_n$ is the $(n-1)$-dimensional Lebesgue measure of the unit Euclidean sphere of $\R^n$. Now, since $\Vg=\Fm(\vg)\in E$, one has, for every $1\leq i\leq H$ and $(t,\x)\in[0,T_1]\times\R^n$,
\begin{equation}\label{ineqViwi}\begin{array}{rcl}
\ds\Big|V_i(t,\x)\int_{\R^n}w_i(t,\y)\md\y\Big| & \le & \ds V_i(t,\x)\int_{\R^n}|u_i(t,\y)-v_i(t,\y)|\,\md\y\vspace{3pt}\\
& \le & \ds d_E(\ug,\vg)\,V_i(t,\x)\int_{\R^n}e^{-\theta\|\y\|}\md\y\vspace{3pt}\\
& = & \ds \theta^{-n}\,\omega_n\,(n-1)!\,d_E(\ug,\vg)\,V_i(t,\x)\vspace{3pt}\\
& \le & \ds \theta^{-n}\,\omega_n\,(n-1)!\,d_E(\ug,\vg)\,K\,e^{(\rmax+1)t-\theta\norm{\x}}\vspace{3pt}\\
& \le & \ds \theta^{-n}\,\omega_n\,(n-1)!\,d_E(\ug,\vg)\,K\,e^{(\rmax^++1)T_1-\theta\xi\cdot\x}.\end{array}
\end{equation}
On the other hand, the definition~\eqref{cond T} of $T_1$ yields
\begin{equation}\label{ineqT1}
\frac{1}{T_1}\ge\frac{\theta^2\mu^2}{2}+\rmax^++3\,\theta^{-n}\,\omega_n\,(n-1)!\,K\,e^{(\rmax^++1)T_1-t/T_1}
\end{equation}
for all $t\in[0,T_1]$. Since $r_i\le\rmax\le\rmax^+$ in $\R^n$ and $u_i\ge0$ in $[0,T_1]\times\R^n$, it then follows from~\eqref{ineqViwi}-\eqref{ineqT1} that each nonnegative function $\overline{W}_i$ satisfies
$$\begin{array}{rcl}
\p_t\overline{W}_i(t,\x) & \ge & \ds\frac{\mu^2}{2}\Delta\overline{W}_i(t,\x)+\rmax^+\overline{W}_i(t,\x)-V_i(t,\x)\int_{\R^n}w_i(t,\y)\,\md\y\vspace{3pt}\\
& \ge & \ds\frac{\mu^2}{2}\Delta\overline{W}_i(t,\x)+\Big(r_i(\x)-\int_{\R^n}u_i(t,\y)\,\md\y\Big)\,\overline{W}_i(t,\x)\vspace{3pt}\\
& & \ds-V_i(t,\x)\int_{\R^n}w_i(t,\y)\,\md\y+\frac{\delta}{H-1} \sum \limits_{\substack{k=1 \\k \neq i}}^H (\overline{W}_k (t,\x)-\overline{W}_i (t,\x))\end{array}$$
for all $(t,\x)\in(0,T_1]\times\R^n$ (recall that $\overline{W}_k-\overline{W}_i\equiv 0$). Therefore, the functions $Z_i:=\overline{W}_i-W_i$ are continuous in $[0,T_1]\times\R^n$, nonnegative at $t=0$, and satisfy
$$\p_t Z_i (t,\x)\ge\frac{\mu^2}2 \Delta Z_i (t,\x)\!+\!\Big(r_i(\x)\!-\!\!\int_{\R^n}\!u_i(t,\y)\md\y\!\Big)Z_i (t,\x)+\frac{\delta}{H\!-\!1}\!\sum \limits_{\substack{k=1 \\k \neq i}}^H(Z_k (t,\x)\!-\!Z_i (t,\x))$$
in $(0,T_1]\times\R^n$. The comparison principle for this linear cooperative system implies that $Z_i(t,\x)\ge0$, that is, $W_i(t,\x)\le\overline{W}_i(t,\x)$, for every $1\le i\le H$ and $(t,\x)\in[0,T_1]\times\R^n$. By changing $W_i$ into $-W_i$ in~\eqref{sys W+} (and the sign of the first integral in the second line) and arguing similarly, one gets that $|W_i(t,\x)|\le\overline{W}_i(t,\x)=(d_E(\ug,\vg)/3)\times e^{t/T_1-\theta\xi\cdot\x}$ for every $1\le i\le H$ and $(t,\x)\in[0,T_1]\times\R^n$. Since $\xi$ was arbitrary in the unit Euclidean sphere of~$\R^n$, one concludes that
\begin{equation}\label{ineqWi}
|U_i(t,\x)-V_i(t,\x)|=|W_i(t,\x)|\le\frac13\,d_E(\ug,\vg)\,e^{t/T_1-\theta\|\x\|}\le\frac e3\,d_E(\ug,\vg)\,e^{-\theta\|\x\|}
\end{equation}
for every $1\le i\le H$ and $(t,\x)\in[0,T_1]\times\R^n$.

In order to establish~\eqref{eq:contractionmapping}, it remains to estimate the $C^{0,1/2}([0,T_1])$ norm of the functions $t\mapsto\int_{\R^n}(U_i(t,\x)-V_i(t,\x))\,\md\x$. For every $1\le i\le H$ and $0<t<t'\le T_1$, by using~\eqref{integrals} for both~$U_i$ and~$V_i$, one infers that
\begin{equation}\label{ineqWitt'}\begin{array}{rcl}
\ds\int_{\R^n}\!\!\!W_i(t',\x)\md\x-\!\!\int_{\R^n}\!\!\!W_i(t,\x)\md\x & \!\!\!\!=\!\!\!\! & \ds\int_t^{t'}\Big(\int_{\R^n}r_i(\x)W_i(s,\x)\md\x\Big)\md s\vspace{3pt}\\
& & \ds-\int_t^{t'}\Big(\int_{\R^n}u_i(s,\x)\md\x\Big)\Big(\int_{\R^n}W_i(s,\x)\md\x\Big)\md s\vspace{3pt}\\
& & \ds-\int_t^{t'}\Big(\int_{\R^n}w_i(s,\x)\md\x\Big)\Big(\int_{\R^n}V_i(s,\x)\md\x\Big)\md s\vspace{3pt}\\
& & \ds+\frac{\delta}{H-1}\sum \limits_{\substack{k=1 \\k \neq i}}^H\int_t^{t'}\!\!\!\!\int_{\R^n}\!(W_k(s,\x)\!-\!W_i(s,\x))\md\x\,\md s.\end{array}
\end{equation}
Let us now bound in absolute value each term of the right-hand side of~\eqref{ineqWitt'}. Since $|r_i(\x)|\le|\rmax|+\alpha\|\x\|^2+\alpha B^2$ by~\eqref{def ri} (we recall that $B=\max_{1\le i\le H}\|\opt_i\|$), it follows from~\eqref{ineqWi} that
$$\begin{array}{l}
\ds\Big|\int_t^{t'}\Big(\int_{\R^n}r_i(\x)W_i(s,\x)\md\x\Big)\md s\Big|\vspace{3pt}\\
\qquad\ds\le\frac{e}{3}\,d_E(\ug,\vg)\,\big[(|\rmax|+\alpha B^2)\theta^{-n}\omega_n(n-1)!+\alpha\,\theta^{-n-2}\omega_n(n+1)!\big]\times|t-t'|
\end{array}$$
for all $0<t<t'\le T_1$. Similarly, using~\eqref{ineqWi} and $\ug\in E$, together with $T_1\le1$, the second term of the right-hand side of~\eqref{ineqWitt'} can be bounded by
$$\Big|\!\int_t^{t'}\!\!\!\!\Big(\int_{\R^n}\!\!u_i(s,\x)\md\x\Big)\Big(\int_{\R^n}\!\!W_i(s,\x)\md\x\Big)\md s\Big|\!\le\!\frac{1}{3}\,d_E(\ug,\vg)\,K\,e^{\rmax^++2}\theta^{-2n}\omega_n^2(n-1)!^2\times|t\!-\!t'|,$$
for all $0<t<t'\le T_1$. Furthermore, since $|w_i(s,\x)|\le d_E(\ug,\vg)\,e^{-\theta\|\x\|}$ and $|V_i(s,\x)|\le K\,e^{(\rmax+1)s-\theta\|\x\|}\le K\,e^{\rmax^++1-\theta\|\x\|}$ for every $1\le i\le H$ and $(s,\x)\in[0,T_1]\times\R^n$ (because $\Vg=\Fm(\vg)\in E$ and $T_1\le1$), the integral in the third line of the right-hand side of~\eqref{ineqWitt'} can be estimated as
$$\Big|\int_t^{t'}\!\!\!\Big(\!\int_{\R^n}\!\!w_i(s,\x)\md\x\Big)\Big(\!\int_{\R^n}\!\!V_i(s,\x)\md\x\Big)\md s\Big|\le d_E(\ug,\vg)\,K\,e^{\rmax^++1}\theta^{-2n}\omega_n^2(n-1)!^2\times|t-t'|.$$
Lastly,~\eqref{ineqWi} implies that
$$\Big|\frac{\delta}{H-1}\sum \limits_{\substack{k=1 \\k \neq i}}^H\int_t^{t'}\!\!\!\!\int_{\R^n}\!(W_k(s,\x)\!-\!W_i(s,\x))\md\x\,\md s\Big|\le\frac23\,\delta\,e\,d_E(\ug,\vg)\,\theta^{-n}\omega_n(n-1)!\times|t-t'|$$
for every $1\le i\le H$ and $0<t<t'\le T_1$. Finally, putting all the above inequalities into~\eqref{ineqWitt'} and remembering the definition~\eqref{defC1} of~$C_1$, it follows that
$$\ds\Big|\int_{\R^n}\!\!\!W_i(t',\x)\md\x-\!\!\int_{\R^n}\!\!\!W_i(t,\x)\md\x\Big|\le C_1\,d_E(\ug,\vg)\times|t-t'|$$
for all $0<t<t'\le T_1$, and then for all $0\le t\le t'\le T_1$ by continuity of $s\mapsto\int_{\R^n}W_i(s,\x)\,\md\x$ in $[0,T_1]$. Remember now that $C_1T_1\le C_1\sqrt{T_1}\le 1/3$ by~\eqref{cond T}, and that $W_i(0,\cdot)=0$ in $\R^N$ for every $1\le i\le H$. As a consequence,
$$\Big\|t\mapsto\int_{\R^n}W_i(t,\x)\,\md\x\Big\|_{C^{0,1/2}([0,T_1])}\le\frac23\,d_E(\ug,\vg)$$
for all $1\le i\le H$. Together with~\eqref{ineqWi}, and using the definition~\eqref{defdE} of $d_E$, one obtains $d_E(\Ug,\Vg)\le\kappa\,d_E(\ug,\vg)$, with $\kappa:=\max(e/3,2/3)=e/3<1$, that is,~\eqref{eq:contractionmapping} is proved.

The Banach-Picard fixed point theorem then implies the existence and uniqueness of~$\ug\in E$ such that $\Fm(\ug)=\ug$. Therefore, $\ug$ belongs to $C([0,T_1]\times\R^n,\R^H)\cap C^{1;2}_{t;\x}((0,T_1]\times\R^n,\R^H)$ and it solves~\eqref{eq:sys H} in $(0,T_1]\times\R^n$ with initial condition $\ug^0$. Furthermore, owing to the definition of $E$ in~\eqref{defE} and the inequality $\theta\le1/\mu$, $\ug$ satisfies~\eqref{ineqexp} in $[0,T_1]\times\R^n$ and, by~\eqref{uLip}, the functions $t\mapsto\int_{\R^n}u_i(t,\x)\,\md\x$ are Lipschitz-continuous in $[0,T_1]$.

Let us now show that $\ug$ can be extended to a global solution, that is, defined for all $t\ge0$. First of all, observe that $T_1$ in~\eqref{cond T} can be written as
$$T_1=\big(A_1+A_2K+A_3K^2\big)^{-1},$$
for some positive constants $A_1$, $A_2$ and $A_3$ that are independent of $K$. Consider then any $T'_1\in(0,T_1)$. From~\eqref{ineqexp} satisfied by~$\ug$ at time $T'_1$, one has $0\le u_i(T'_1,\x)\le K\,e^{(\rmax+1)T'_1}\,e^{-\theta\|\x\|}$ for all $1\le i\le H$ and $\x\in\R^n$. In other words, the continuous functions $u_i(T'_1,\cdot)$ satisfy properties similar to~\eqref{data}-\eqref{init bound}, with $K$ replaced by $K\,e^{(\rmax+1)T'_1}$. From the arguments of the previous paragraphs, there exists a unique solution $\tilde{\ug}=(\tilde{u}_1,\ldots,\tilde{u}_H)\in C([T'_1,T'_2]\times\R^n,\R^H)\cap C^{1;2}_{t,\x}((T'_1,T'_2]\times\R^n,\R^H)$ such that, for every $1\le i\le H$, $\tilde{u}_i(T'_1,\cdot)=u_i(T'_1,\cdot)$ in~$\R^n$,
$$0\le\tilde{u}_i(t,x)\le K\,e^{(\rmax+1)T'_1}\,e^{(\rmax+1)(t-T'_1)-\theta\norm{\x}}=K\,e^{(\rmax+1)t-\theta\norm{\x}}$$
for all $(t,\x)\in[T'_1,T'_2]\times\R^n$ and the functions $t\mapsto\int_{\R^n}\tilde{u}_i(t,\x)\md\x$ are in $C^{0,1}([T'_1,T'_2])$, with
$$T'_2-T'_1=\big(A_1+A_2Ke^{(\rmax+1)T'_1}+A_3K^2e^{2(\rmax+1)T'_1}\big)^{-1}.$$
Since the uniqueness also holds in any interval $[T'_1,T''_2]$ with $T'_1<T''_2\le T'_2$ and since~$\ug$ restricted to $[T'_1,\min(T_1,T'_2)]\times\R^n$ satisfies the same properties as the restriction of~$\tilde{\ug}$ to this set, one gets that these two restrictions are equal. Therefore, the function $\ug=(u_1,\ldots,u_H)$ extended by $\tilde{\ug}$ in $(T_1,T'_2]\times\R^n$ (provided this set is not empty), is then of class $C([0,T'_2]\times\R^n,\R^H)\cap C^{1;2}_{t;\x}((0,T'_2]\times\R^n,\R^H)$. Furthermore, $0\le u_i(t,x)\le K\,e^{(\rmax+1)t-\theta\norm{\x}}$ for all $(t,\x)\in[0,T'_2]\times\R^n$ and $1\le i\le H$, and the functions $t\mapsto\int_{\R^n}u_i(t,\x)\md\,\x$ are in~$C^{0,1}([0,T'_2])$. Since this argument holds for every $T'_1\in(0,T_1)$, it follows that $\ug$ can then be extended in $[0,T_2)\times\R^n$, with
$$T_2-T_1=\big(A_1+A_2Ke^{(\rmax+1)T_1}+A_3K^2e^{2(\rmax+1)T_1}\big)^{-1},$$
and the above properties hold with $[0,T'_2]$ and $(0,T'_2]$ replaced by $[0,T_2)$ and $(0,T_2)$, respectively, with the functions $t\mapsto\int_{\R^n}u_i(t,\x)\md\x$ belonging to $C^{0,1}_{loc}([0,T_2))$.

By an immediate induction, there exists a sequence $(T_m)_{m\ge1}$ of positive real numbers such that, for every $m\ge1$,
\begin{equation}\label{Tm}
T_{m+1}-T_m=\big(A_1+A_2Ke^{(\rmax+1)T_m}+A_3K^2e^{2(\rmax+1)T_m}\big)^{-1}
\end{equation}
and $\ug$ can be extended as a $C([0,T_m)\times\R^n,\R^H)\cap C^{1;2}_{t;\x}((0,T_m)\times\R^n,\R^H)$ function satisfying $0\le u_i(t,x)\le K\,e^{(\rmax+1)t-\theta\norm{\x}}$ for all $(t,\x)\in[0,T_m)\times\R^n$ and $1\le i\le H$, and the functions $t\mapsto\int_{\R^n}u_i(t,\x)\md\,\x$ are in $C^{0,1}_{loc}([0,T_m))$. The sequence $(T_m)_{m\ge1}$ is increasing and is not bounded (otherwise, it would converge to a positive real number while $T_{m+1}-T_m$ would converge to $0$, which would contradict~\eqref{Tm}). Therefore, $T_m\to+\infty$ and the function $\ug$ is then defined in $[0,+\infty)\times\R^n$, and it satisfies all the desired properties of the conclusion of Theorem~\ref{thm:well-pos}.

Finally, if $\vg$ is another solution satisfying all the properties of $\ug$ listed in Theorem~\ref{thm:well-pos}, it follows that the restrictions of $\vg$ and $\ug$ to $[0,T_1]\times\R^n$ belong to $E$ defined in~\eqref{defE}, and that $\Fm(\vg)=\vg$, from the equations satisfied by the functions $v_i$. Therefore, by uniqueness, the restrictions of $\vg$ and $\ug$ to $[0,T_1]\times\R^n$ are equal. By an immediate induction, using the construction of the above sequence $(T_m)_{m\ge1}$, one gets that $\vg$ and $\ug$ are equal in~$[0,T_m)\times\R^n$ for every $m\ge1$, hence $\ug=\vg$ in $[0,+\infty)\times\R^n$. The proof of Theorem~\ref{thm:well-pos} is thereby complete.
\end{proof}

\begin{remark}{\rm It follows from the above proof that, for any $\varepsilon\in(0,1)$, the quantity $\rmax+1$ could be replaced by $\rmax+\varepsilon$ in~\eqref{ineqexp}, at the expense of replacing $\min(\theta,1/\mu)$ by $\min(\theta,\sqrt{\varepsilon}/\mu)$. Indeed, in the above proof, by assuming that $0<\theta\le\sqrt{\varepsilon}/\mu$ without loss of generality, the functions $\overline{U}_i$ in~\eqref{defolUi} would now be defined as $\overline{U}_i(t,\x):=K\;e^{(\rmax+\varepsilon)t-\theta\xi\cdot\x}$ and would satisfy~\eqref{eqolUi}, with $\rmax+\varepsilon$ instead of $\rmax+1$, but they would still fulfill~\eqref{ineqU+}. The rest of the proof is unchanged, with just $(\rmax+1,\rmax^++1,\rmax^++2)$ replaced by $(\rmax+\varepsilon,\rmax^++\varepsilon,\rmax^++\varepsilon+1)$ in the definitions or calculations. As a consequence, if~$\rmax<0$, the solution $\ug$ constructed in Theorem~$\ref{thm:well-pos}$ is such that, for every $1\le i\le H$, $\|u_i(t,\cdot)\|_{L^\infty(\R^n)}\to0$ as $t\to+\infty$ exponentially fast.}
\end{remark}

\begin{proof}[Proof of Proposition~$\ref{theo pers log}$]
$(i)$ We first assume $\lambda_H >0$. Consider the following linear system, for $1 \leq i \leq H$:
\begin{equation}\label{eq:sys H lin}
\partial_t v_i(t,\x)=\frac{\mu^2}{2}\Delta v_i(t,\x)+r_i(\x) v_i(t,\x)+\frac{\delta}{H-1}\sum \limits_{\substack{k=1 \\k \neq i}}^H(v_k(t,\x)-v_i(t,\x)).
\end{equation}
It corresponds to a Malthusian growth of the population. Since $N_i(t)=\int_{\R^n}u_i(t,\x)\,\md\x\ge0$ for all $t\ge0$ by~\eqref{ineqexp}, the solution $\ug$ of~\eqref{eq:sys H}-\eqref{init bound} is a subsolution of~\eqref{eq:sys H lin}, while the function $(t,\x)\mapsto\Phi(\x)\,e^{-\lambda_Ht}$ is a solution, by definition of the principal eigenpair $(\lambda_H,\Phi)$ of~\eqref{eq:eigenvalue_pb}. Since $\ug^0$ is assumed to be compactly supported and $\Phi$ is continuous and positive componentwise, there is a constant $A>0$ such that $\ug^0\le A\,\Phi$ componentwise in $\R^n$. The compa\-rison principle applied to the cooperative system~\eqref{eq:sys H lin} implies that $\ug(t,\x)\le A\,\Phi(\x)\,e^{-\lambda_Ht}$ componentwise for all $(t,\x)\in[0,+\infty)\times\R^n$, hence $N_i(t)\to0$ as $t\to+\infty$ for every $1\le i\le H$.
\vskip 0.3cm
\noindent $(ii)$ We now assume that $\lambda_H=0$. Notice that the previous arguments imply that the continuous functions $N_i$ are bounded in $[0,+\infty)$. More precisely, as in the previous paragraph, since $\ug^0$ is assumed to be compactly supported, there is a constant $A>0$ such that $\ug^0\le A\,\Phi$ componentwise in $\R^n$, hence $\ug(t,\x)\le A\,\Phi(\x)$ componentwise for all $(t,\x)\in[0,+\infty)\times\R^n$. Assume now by way of contradiction that $\liminf_{t\to+\infty}\min_{1\le i\le H}N_i(t)>0$. Then there are $T>0$ and $\varepsilon>0$ such that $N_i(t)=\int_{\R^n}u_i(t,\x)\,\md\x\ge\varepsilon$ for every $1\le i\le H$ and $t\ge T$. For $t\ge T$, the solution $\ug$ is a subsolution of the linear system
$$\partial_t v_i(t,\x)=\frac{\mu^2}{2}\Delta v_i(t,\x)+(r_i(\x)-\varepsilon)v_i(t,\x)+\frac{\delta}{H-1}\sum \limits_{\substack{k=1 \\k \neq i}}^H(v_k(t,\x)-v_i(t,\x)),$$
whereas $(t,\x)\mapsto A\,\Phi(\x)\,e^{-\varepsilon(t-T)}$ is a solution (since $\lambda_H=0$). The maximum principle applied to this cooperative system implies that $\ug(t,\x)\le A\,\Phi(\x)\,e^{-\varepsilon(t-T)}$ for all $(t,\x)\in[T,+\infty)\times\R^n$. Hence, $N_i(t)\to0$ as $t\to+\infty$, for every $1\le i\le H$, contradicting our assumption. As a consequence, $\liminf_{t\to+\infty}\min_{1\le i\le H}N_i(t)=0$.
\vskip 0.3cm
\noindent $(iii)$ Assume finally that $\lambda_H<0$ and assume that the desired conclusion~\eqref{persistence} is not satisfied. Then there are $T>0$ and $\varepsilon\in(0,-\lambda_H)$ such that $N_i(t)=\int_{\R^n}u_i(t,\x)\,\md\x\le-\lambda_H-\varepsilon$ for every $1\le i\le H$ and $t\ge T$. Since $\lambda_H=\lim_{R\to+\infty}\lambda_H^R$, where $\lambda_H^R$ denotes the principal eigenvalue of the operator $\mathcal{A}$ acting on $(C^\infty_0(\overline{B(\mathcal{O},R)}))^H$, there is $R>0$ such that $\lambda_H^R\le\lambda_H+\varepsilon/2$, hence $N_i(t)=\int_{\R^n}u_i(t,\x)\,\md\x\le-\lambda_H^R-\varepsilon/2$ for every $1\le i\le H$ and $t\ge T$. Now, each function $u_i(T,\cdot)$ is continuous and positive in $\R^N$, hence there is $\eta>0$ such that $\ug(T,\x)\ge\eta\,\Phi^R(\x)$ componentwise for all $\x\in\overline{B(\mathcal{O},R)}$, where $\Phi^R\in (C^\infty_0(\overline{B(\mathcal{O},R)}))^H$ denotes a principal eigenvector of $\mathcal{A}$, associated to the principal eigenvalue $\lambda_H^R$ (that is, $\Phi^R$ is positive componentwise in $B(\mathcal{O},R)$, it vanishes on $\partial B(\mathcal{O},R)$ and it satisfies $\mathcal{A}\Phi^R=\lambda^R_H\Phi^R$ in $\overline{B(\mathcal{O},R)}$). The function $(t,\x)\mapsto\eta\,\Phi^R(\x)\,e^{(\varepsilon/2)(t-T)}$ is a solution of the linear system
$$\partial_t v_i(t,\x)=\frac{\mu^2}{2}\Delta v_i(t,\x)+(r_i(\x)+\lambda_H^R+\varepsilon/2)v_i(t,\x)+\frac{\delta}{H-1}\sum \limits_{\substack{k=1 \\k \neq i}}^H(v_k(t,\x)-v_i(t,\x))$$
in $\R\times\overline{B(\mathcal{O},R)}$, with Dirichlet boundary conditions on $\R\times\partial B(\mathcal{O},R)$, whereas $\ug$ is a supersolution in $[T,+\infty)\times\overline{B(\mathcal{O},R)}$, and the functions are compared at time $T$. The maximum principle applied to this cooperative system implies that $\ug(t,\x)\ge\eta\,\Phi^R(\x)\,e^{(\varepsilon/2)(t-T)}$ for all $(t,\x)\in[T,+\infty)\times\overline{B(\mathcal{O},R)}$. Hence, $N_i(t)\ge\int_{B(\mathcal{O},R)}u_i(t,\x)\,\md\x\to+\infty$ as $t\to+\infty$, for every $1\le i\le H$, contradicting our assumption. As a consequence,~\eqref{persistence} is proved, and the proof of Proposition~\ref{theo pers log} is complete.
\end{proof}

\begin{remark}\label{remstat}{\rm We here show that, when $\lambda_H\ge0$, the system~\eqref{eq:sys H} does not admit any positive bounded stationary state. Assume by way of contradiction that such a stationary state $\pg=(p_1,\ldots,p_H)$ exists. It necessarily converges to $0$ exponentially as $\|\x\|\to+\infty$, from the confining properties of the fitnesses $r_i$. Let $R>0$ be such that $r_i(\x)<0$ for all $1\le i\le H$ and $\|\x\|\ge R$. Let also $A>0$ be such that $\pg(\x)\le A\,\Phi(\x)$ componentwise for all $\|\x\|\le R$, where $\Phi=(\varphi_1,\ldots,\varphi_H)$ denotes as usual the, positive, principal eigenvector of~\eqref{eq:eigenvalue_pb}. Notice that, from its positivity, the stationary state $\pg$ of~\eqref{eq:sys H} is a subsolution of the following linear system
$$\frac{\mu^2}{2}\Delta p_i(\x)+r_i(\x)p_i(\x)+\frac{\delta}{H-1}\sum \limits_{\substack{k=1 \\k \neq i}}^H(p_k(\x)-p_i(\x))\ge0,$$
while $\Phi$ is a supersolution of the same linear system, because $\lambda_H\ge0$. The weak maximum principle for this system then yields $\pg(\x)\le A\,\Phi(\x)$ componentwise for all $\|\x\|\ge R$. Hence, $\pg(\x)\le A\,\Phi(\x)$ componentwise for all $\x\in\R^n$. Let finally $\ug$ be the solution of~\eqref{eq:sys H} given by Theorem~\ref{thm:well-pos}, with initial condition $\ug^0:=\pg$. On the one hand, by the uniqueness property of that theorem, there holds $\ug(t,\x)\equiv\pg(\x)$ for all $t\ge0$ and $\x\in\R^n$. On the other hand, the same arguments as in the proof of Proposition~\ref{theo pers log} in the cases $\lambda_H>0$ or $\lambda_H=0$ then imply that, in both cases, one has at least $\liminf_{t\to+\infty}\big(\min_{1\le i\le H}\int_{\R^n}u_i(t,\x)\,\md\x\big)=0$. This leads to a contradiction, since $\int_{\R^n}u_i(t,\x)\,\md\x=\int_{\R^n}p_i(\x)\,\md\x>0$ is independent of $t\ge0$, for every $1\le i\le H$. As a conclusion,~\eqref{eq:sys H} does not admit any positive bounded stationary state when $\lambda_H\ge0$.
}
\end{remark}

\begin{proof}[Proof of Proposition~$\ref{prop:statio_states}$]
\noindent $(i)$ Assume that $H=1$ and $\lambda_1<0$. In this case, any positive bounded stationary state $p_1$ necessarily decays to $0$ at least exponentially as $\|\x\|\to+\infty$, and then belongs to $H^1(\R^n)\cap L^2_w(\R^n)$ from standard elliptic estimates. It is also a principal eigenfunction of the operator $\A$ defined in~\eqref{eq:defA}, with associated principal eigenvalue $-\int_{\R^n}p_1$. By uniqueness of the eigenpair $(\lambda_1,\varphi_1 )$, up to a multiplication of $\varphi_1$ by positive constants, we deduce that there exists $\kappa>0$ such that $p_1=\kappa\,\varphi_1$ in $\R^n$, and that $\int_{\R^n}p_1=-\lambda_1$. We get that $\kappa=-\lambda_1/\lp \int_{\R^n}\varphi_1\rp$. As it is immediate to check that $\kappa\,\varphi_1$ is a positive bounded stationary state of~\eqref{eq:sys H}, this shows~$(i)$.
\vskip 0.3cm
\noindent $(ii)$ Assume that $H=2$ and $\lambda_2<0.$ Without loss of generality, up to rotation and translation of the phenotypic space, we can fix the optima $\opt_1$ and $\opt_2$ at $\opt _1:=(-\beta,0,\cdots,0),$ $\opt _2:=(\beta,0,\cdots,0),$ for some $\beta\geq 0$. First, we recall from~\cite{HamLavRoq20} that, by uniqueness of the principal eigenfunctions (up to multiplication) and by symmetry of the problem~\eqref{eq:sys H} with respect to the origin, we have $\varphi_2(\x)=\varphi_1(\iota(\x))>0$ for all $\x\in\R^n$, with
$$\iota(\x):=\iota(x_1, x_2,\ldots,x_n)=(-x_1, x_2,\ldots,x_n).$$
It follows that $\int_{\R^n}\varphi_1=\int_{\R^n}\varphi_2>0$. Furthermore, the principal eigenfunctions satisfy:
\begin{equation} \label{eq:eigenpb_H2}
-\ds \frac{\mu^2} 2\,\Delta \varphi_i(\x)-r_i(\x)\varphi_i(\x)+\delta(\varphi_i(\x)-\varphi_j(\x))=\lambda_2\,\varphi_i,\ \ \x\in\R^n,
\end{equation}
for $(i,j)=(1,2)$ and $(i,j)=(2,1)$. Define $(p_1,p_2)$ as in~\eqref{eq:2hosts_statio_state} and $\rb_i$ as in~\eqref{eq:def_rbar}, with $u_i(t,\y)$ replaced by $p_i(\y)$. Owing to the definition~\eqref{eq:2hosts_statio_state}, integrating~\eqref{eq:eigenpb_H2} over $\R^n$ and using that the $\varphi_i$'s and its first order derivatives decay at least exponentially as $\|\x\|\to+\infty$, we infer that $\rb_1=\rb_2=-\lambda_2$. We also note that $\int_{\R^n}p_1=\int_{\R^n}p_2=-\lambda_2$. Dividing \eqref{eq:eigenpb_H2} by $\int_{\R^n}\varphi_i=\int_{\R^n}\varphi_j$, and multiplying by $-\lambda_2$ we then get:
$$ -\ds \frac{\mu^2} 2\,\Delta p_i(\x)-r_i(\x)p_i (\x)+\delta (p_i(\x)-p_j(\x)) = -\Big(\int_{\R^n}p_i(\y)\,\md\y\Big)\,p_i(\x),\ \ \x\in \R^n,$$
for $(i,j)=(1,2)$ and $(i,j)=(2,1)$. This proves $(ii)$.
\vskip 0.3cm
\noindent $(iii)$ Assume that $H\ge3$ and consider the optima $\opt_j$'s located at $\opt_1:=(-\beta,0,\ldots,0)$ and $\opt_i:=(\beta,0,\cdots,0)$ for all $2\le i\le H$, with $\beta>0$. Let $(\lambda_H,\Phi)$ be the principal eigenpair of~\eqref{eq:eigenvalue_pb}, with $\Phi=(\varphi_1,\ldots,\varphi_H)$ here normalised so that $\int_{\R^n}\varphi_1=1$. First of all, by uniqueness, we readily check that $\varphi_i=\varphi_2$ for all $2\le i\le H$, hence the eigenfunctions satisfy:
\begin{equation}\label{eq:phi1phi2}
\left\{\begin{array}{l}
-\ds\frac{\mu^2} 2\,\Delta \varphi_1(\x) - r_1(\x) \varphi_1(\x)+\delta(\varphi_1(\x)-\varphi_2(\x))=\lambda_H \, \varphi_1(\x), \vspace{3mm} \\
-\ds \frac{\mu^2} 2 \ \Delta \varphi_2(\x) - r_2(\x) \varphi_2 (\x) +\frac{\delta}{H-1}(\varphi_2(\x)-\varphi_1(\x))  = \lambda_H \, \varphi_2(\x).\end{array}\right.
\end{equation}

We now claim that, for all $\beta>0$ large enough, no positive multiple of $\Phi$ is a statio\-nary state of the system~\eqref{eq:sys H}. Assume by way of contradiction that there exist a sequence $(\beta_m)_{m\in\N}$ diverging to $+\infty$ and a sequence $(\kappa_m)_{m\in\N}$ of positive real numbers such that $(\kappa_m\varphi_{1,m},\kappa_m\varphi_{2,m},\ldots,\kappa_m\varphi_{2,m})$ is a positive bounded stationary state of~\eqref{eq:sys H}, associated with the optima $\opt_{1,m}:=(-\beta_m,0,\ldots,0)$ and $\opt_{2,m}=\cdots=\opt_{H,m}:=(\beta_m,0,\ldots,0)$. Call~$(\lambda_{H,m})_{m\in\N}$ the sequence of associated principal eigenvalues, and normalise the principal eigenvectors so that $\int_{\R^n}\varphi_{1,m}(\x)\md\x=1$. For each $m\in\N$, plugging the solution $(\kappa_m\varphi_{1,m},\kappa_m\varphi_{2,m},\ldots,\kappa_m\varphi_{2,m})$ into~\eqref{eq:sys H} yields $\kappa_m\int_{\R^n}\varphi_{i,m}(\x)\,\md\x=-\lambda_H$ for every $1\le i\le H$, hence $\int_{\R^n} \varphi_{2,m}=\int_{\R^n} \varphi_{1,m}=1$ in~\eqref{eq:phi1phi2} (with $\varphi_1$ and $\varphi_2$ replaced by~$\varphi_{1,m}$ and~$\varphi_{2,m}$). From the conclusion of Proposition~\ref{prop bound lambda} (whose proof is independent of this one), each $\lambda_{H,m}$ satisfies
$$\lambda_1\le\lambda_{H,m}\le\lambda_1+\delta=\frac{\mu n\sqrt{\alpha}}{2}-\rmax+\delta.$$
Therefore, up to extraction of a subsequence one can assume that
$$\lambda_{H,m}\to\lambda_\infty\in\R\ \hbox{ as $m\to+\infty$}.$$

Call now $\phi_m:=\varphi_{1,m}(\cdot+\opt_{1,m})$ and $\psi_m:=\varphi_{2,m}(\cdot+\opt_{1,m})=\varphi_{2,m}(\cdot-\opt_{2,m})$. From~\eqref{eq:phi1phi2}, one has
\begin{equation}\label{eqphim}
-\frac{\mu^2}2\,\Delta \phi_m(\x)-\Big(\rmax-\alpha\,\frac{\|\x\|^2}{2}\Big)\phi_m(\x)+\delta(\phi_m(\x)-\psi_m(\x))=\lambda_{H,m}\,\phi_m(\x)
\end{equation}
for all $m\in\N$ and $\x\in\R^n$. Since $\int_{\R^n}\phi_m=\int_{\R^n}\psi_m=1$ for all $m\in\N$, it follows that the sequence $(\Delta\phi_m)_{m\in\N}$ is bounded in $L^1_{loc}(\R^n)$ and then (see~\cite{Pon16}\footnote{The authors are grateful to L.~Dupaigne and F.~Murat for helpful discussions on this point.}) that the sequence $(\phi_m)_{m\in\N}$ is bounded in $W^{1,p}_{loc}(\R^n)$ for every $1\le p<n/(n-1)$ if $n\ge2$ (respectively bounded in $W^{1,\infty}_{loc}(\R)$ if $n=1$). Therefore, up to extraction of a subsequence, there is a nonnegative function $\phi_\infty$ in $W^{1,p}_{loc}(\R^n)$ for every $1\le p<n/(n-1)$ if $n\ge2$ (respectively for every $1\le p\le+\infty$ if $n=1$) such that $\phi_m\to\phi_\infty$ almost everywhere in $\R^n$ and in~$L^q_{loc}(\R^n)$ for every $1\le q<n/(n-2)$ if $n\ge3$ (respectively for every $1\le q<+\infty$ if $n=2$, respectively in $L^\infty_{loc}(\R)$ if $n=1$), and $\p_{x_j}\phi_m\rightharpoonup\p_{x_j}\phi_\infty$ weakly in $L^p_{loc}(\R^n)$ for every $1\le p<n/(n-1)$ if $n\ge2$ (respectively for every $1\le p<+\infty$ if $n=1$). Since $\int_{B(\mathcal{O},R)}\phi_\infty=\lim_{m\to+\infty}\int_{B(\mathcal{O},R)}\phi_m\le1$ for every $R>0$, it follows from the monotone convergence theorem that $\int_{\R^n}\phi_\infty\le1$. On the other hand, by integrating~\eqref{eqphim} over $\R^n$ and using that $\int_{\R^n}\phi_m=\int_{\R^n}\psi_m=1$, one gets that
\begin{equation}\label{phim}
\frac{\alpha R^2}{2}\int_{\R^n\setminus B(\mathcal{O},R)}\phi_m(\x)\,\md\x\le\frac{\alpha}{2}\int_{\R^n}\|\x\|^2\phi_m(\x)\,\md\x=\lambda_{H,m}+\rmax\le\frac{\mu n\sqrt{\alpha}}{2}+\delta
\end{equation}
for all $m\in\N$ and $R>0$. Since the nonnegative functions $\phi_m$ converge to $\phi_\infty$ in $L^1_{loc}(\R^n)$ and have unit $L^1(\R^n)$ norm, it follows that $\int_{\R^n}\phi_\infty(\x)\,\md\x=1$. As in~\eqref{phim}, by using now $\psi_m(\cdot+2\opt_{2,m})=\varphi_{2,m}(\cdot+\opt_{2,m})$ instead of $\phi_m=\varphi_{1,m}(\cdot+\opt_{1,m})$, one also has
$$\frac{\alpha R^2}{2}\!\!\int_{\R^n\setminus B(2\opt_{2,m},R)}\!\psi_m(\x)\,\md\x\le\frac{\alpha}{2}\!\int_{\R^n}\!\|\x\|^2\psi_m(\x\!+\!2\opt_{2,m})\,\md\x=\lambda_{H,m}\!+\!\rmax\le\frac{\mu n\sqrt{\alpha}}{2}\!+\!\delta$$
for all $m\in\N$ and $R>0$. Since $\psi_m\ge0$ in $\R^n$ and since $\opt_{2,m}=(\beta_m,0,\ldots,0)$ with $\lim_{m\to+\infty}\beta_m=+\infty$, it follows that $\psi_m\to0$ as $m\to+\infty$ in $L^1_{loc}(\R^n)$. Passing to the limit as $m\to+\infty$ in~\eqref{eqphim} after multiplying against any $C^\infty(\R^n)$ function with compact support, it follows that $\phi_\infty$ is a distributional solution of
\begin{equation}\label{eqphiinfty}
-\frac{\mu^2}2\,\Delta \phi_\infty-\Big(\rmax-\alpha\,\frac{\|\x\|^2}{2}\Big)\phi_\infty(\x)+\delta\phi_\infty=\lambda_\infty\,\phi_\infty
\end{equation}
in $\R^n$. Since $\phi_\infty\in W^{1,p}_{loc}(\R^n)$ for every $1\le p<n/(n-1)$ if $n\ge2$ (respectively for every $1\le p\le+\infty$ if $n=1$), one infers from the weak formulation of~\eqref{eqphiinfty} that $\partial_{x_j}\phi_\infty$ actually belongs to $W^{1,p}_{loc}(\R^n)$ for every $1\le j\le n$ and then $\phi_\infty\in W^{2,p}_{loc}(\R^n)$, for every $1\le p<n/(n-1)$ if $n\ge2$ (respectively for every $1\le p\le+\infty$ if $n=1$). Therefore, by bootstrapping and differentiating, it follows that $\phi_\infty$ is a $C^\infty(\R^n)$ solution of~\eqref{eqphiinfty}. Since it is nonnegative with unit $L^1(\R^n)$ mass, the strong elliptic maximum principle implies that~$\phi_\infty>0$ in $\R^n$. In particular, there is $R_0>0$ such that the function~$\phi_\infty$ is subharmonic in~$\R^n\setminus B(\mathcal{O},R_0)$. Hence,
$$0<\phi_\infty(\x)\le n\,\omega_n^{-1}\int_{B(\x,1)}\phi_\infty(\y)\,\md\y\le n\,\omega_n^{-1}\int_{\R^n\setminus B(\mathcal{O},\|\x\|-1)}\phi_\infty(\y)\,\md\y$$
for all $\|\x\|\ge R_0+1$, where $\omega_n/n$ is the Lebesgue measure of the Euclidean unit ball of $\R^n$. Since $\phi_\infty\in L^1(\R^n)$, one then deduces that $\phi_\infty(\x)\to0$ as $\|\x\|\to+\infty$. On the other hand, for every $\sigma>0$, there is $R_1>0$ such that $\alpha\|\x\|^2/2-\rmax+\delta\ge\lambda_\infty$ for all $\|\x\|\ge R_1$ and such that the function $\x\mapsto e^{-\sigma\|\x\|}$ is a supersolution of~\eqref{eqphiinfty} in $\R^n\setminus B(\mathcal{O},R_1)$. Therefore, by choosing $A>0$ such that $\max_{\partial B(\mathcal{O},R_1)}\phi_\infty\le A\,e^{-\sigma R_1}$, the weak maximum principle implies that $0<\phi_\infty(\x)\le A\,e^{-\sigma\|\x\|}$ for all $\|\x\|\ge R_1$. As a consequence, $\phi_\infty$ converges to $0$ as $\|\x\|\to+\infty$ faster than any exponentially decaying function. So does then the function $\x\mapsto\|\x\|^2\phi_\infty(\x)$, hence standard elliptic estimates imply that $\|\nabla\phi_\infty\|$ also converges to $0$ as $\|\x\|\to+\infty$ faster than any exponentially decaying function. Finally, $\phi_\infty\in H^1(\R^n)\cap L^2_w(\R^n)$ and, since it is positive, it is therefore the principal eigenfunction of the operator $-(\mu^2/2)\Delta-(\rmax-\alpha\|\x\|^2/2)+\delta$, with principal eigenvalue $\lambda_\infty$. Hence,
\begin{equation}\label{lambdainfty}
\lambda_\infty=\lambda_1+\delta=\frac{\mu n\sqrt{\alpha}}{2}-\rmax+\delta.
\end{equation}

Finally, by arguing similarly with the functions $\tilde\phi_m:=\varphi_{1,m}(\cdot+\opt_{2,m})$ and $\tilde\psi_m:=\varphi_{2,m}(\cdot+\opt_{2,m})$, and passing to the limit as $m\to+\infty$ in the equation
$$-\frac{\mu^2}2\,\Delta\tilde\psi_m(\x)-\Big(\rmax-\alpha\,\frac{\|\x\|^2}{2}\Big)\tilde\psi_m(\x)+\frac{\delta}{H-1}(\tilde\psi_m(\x)-\tilde\phi_m(\x))=\lambda_{H,m}\,\tilde\psi_m(\x),$$
one similarly gets that
\begin{equation}\label{conclusion}
\lambda_\infty=\lambda_1+\frac{\delta}{H-1}=\frac{\mu n\sqrt{\alpha}}{2}-\rmax+\frac{\delta}{H-1},
\end{equation}
contradicting~\eqref{lambdainfty}, since $\delta>0$ and $H\ge 3$. The proof of Proposition~\ref{prop:statio_states} is thereby complete.
\end{proof}

%%%%%%%%%%%%%%%%%%%%%%%%%%%%%%%%%%%%%%%%%%%%%%%%%%%%%%%%

\subsection{Proofs of Propositions~\ref{prop monolambda}-\ref{prop bound lambda} and~\ref{prop:lambda-infini}}\label{sec52}

We actually start with the proof of Proposition~\ref{prop bound lambda}, since its conclusion is used in the proof of Proposition~\ref{prop monolambda}.

\begin{proof}[Proof of Proposition~$\ref{prop bound lambda}$] First of all, since
\begin{equation}\label{young}
2\int_{\R^n}\phi(\x)\,\psi(\x)\,\md\x\le\int_{\R^n}(\phi(\x))^2\,\md\x+\int_{\R^n}(\psi(\x))^2\,\md\x
\end{equation}
for any two $L^2(\R^n)$ functions $\phi$ and $\psi$, and since the principal eigenvector $\Phi=(\varphi_1,\ldots,\varphi_H)$ of $\mathcal{A}$ is normalised in $L^2(\R^n)^H$, the first inequality in~\eqref{ineqs} follows immediately. The second inequality follows from $\lambda_H(\delta,\alpha,\mu,\opt_1,\ldots,\opt_H)=Q_H(\Phi)$ and the fact that $\lambda_1$ is the minimum of $Q_H$ when $\delta=0$:
$$\sum_{i=1}^H\lp\!\frac{\mu^2}{2\!\!}\int_{\R^n}\!\!\|\nabla\varphi_i(\x)\|^2\,\md\x-\!\int_{\R^n}r_i(\x)\,(\varphi_i(\x))^2\,\md\x\!\rp\!\!\ge\sum_{i=1}^H\lambda_1(\alpha,\mu)\!\int_{\R^n}\!\!(\varphi_i(\x))^2\,\md\x=\lambda_1(\alpha,\mu).$$
Now, recalling that $G_i=G(\cdot-\opt_i)$ and testing the Rayleigh quotient $Q_H$ at $\Psi:=(G_1/\sqrt{H},\ldots,G_H/\sqrt{H})$, we get that
$$\begin{array}{rcl}
\lambda_H(\delta,\alpha,\mu,\opt_1,\ldots,\opt_H) & \leq & Q_H(\Psi)\vspace{3pt}\\
& = & \ds\lambda _1(\alpha,\mu)+\delta\left(1-\sum_{1\le i< j\le H}\frac{2}{H(H-1)}\int_{\R^n}G_i(\x)\,G_j(\x)\,\md\x\right)\vspace{3pt}\\
& \le & \lambda _1(\alpha,\mu)+\delta.\end{array}$$
Lastly, for any $\Gamma \in H^1(\R^n)\cap L^{2}_w(\R^n)$ such that $\int_{\R^n} \Gamma^2=1/H$, we observe that the quantity $Q_H(\Gamma,\ldots,\Gamma)$ is independent of $\delta$ and therefore the Rayleigh formula \eqref{Rayleigh}-\eqref{eq:rayleigh_Q} (together with the lower bound $\lambda_H(\delta,\alpha,\mu,\opt_1,\ldots,\opt_H)\ge\mu\,n\,\sqrt{\alpha}/2-\rmax$) implies that $\lambda_H(\delta,\alpha,\mu,\opt_1,\ldots,\opt_H)$ is bounded independently of $\delta$ (for every given $\alpha>0$, $\mu>0$ and $(\opt_i)_{1\le i\le H}\in(\R^n)^H$).
\end{proof}

\begin{proof}[Proof of Proposition~$\ref{prop monolambda}$]
Let us begin with the proof of the continuity of the map
$$(\delta,\alpha,\mu,\opt_1,\ldots,\opt_H)\mapsto\lambda_H(\delta,\alpha,\mu,\opt_1,\ldots,\opt_H)$$
in $(0,+\infty)^3\times(\R^n)^H$. Fix $(\delta,\alpha,\mu,\opt_1,\ldots,\opt_H)\in(0,+\infty)^3\times(\R^n)^H$ and consider any sequence $(\delta^m,\alpha^m,\mu^m,\opt^m_1,\ldots,\opt^m_H)_{m\in\N}$ in $(0,+\infty)^3\times(\R^n)^H$ and converging to~$(\delta,\alpha,\mu,\opt_1,\ldots,\opt_H)$. Since the sequence $(\lambda_H(\delta^m,\alpha^m,\mu^m,\opt^m_1,\ldots,\opt^m_H))_{m\in\N}$ is bounded by Proposition~\ref{prop bound lambda}, one can assume that it converges to some real number $\lambda$, up to extraction of a subsequence. For each $m\in\N$, call $\Phi^m:=(\varphi_1^m,\ldots,\varphi_H^m)$ the normalised principal eigenvector associated to the principal eigenvalue $\lambda_H(\delta^m,\alpha^m,\mu^m,\opt^m_1,\ldots,\opt^m_H)$. As at the beginning of the proof of Proposition~\ref{prop bound lambda}, the inequality~\eqref{young} applied to $\int_{\R^n}\varphi_i^m\varphi_j^m$ implies that
$$\begin{array}{r}
\ds\sum _{i=1}^{H}\lp\frac{(\mu^m)^2}{2}\int_{\R^{n}} \|\nabla\varphi_i^m(\x)\|^2\,\md\x+\frac{\alpha^m}{2}\int_{\R^{n}}\|\x-\opt_i^m\|^2(\varphi_i^m(\x))^{2}(\x)\,\md\x\rp\vspace{3pt}\\
\le Q_H(\Phi^m)+\rmax=\lambda_H(\delta^m,\alpha^m,\mu^m,\opt^m_1,\ldots,\opt^m_H)+\rmax\end{array}$$
for all $m\in\N$. Together with $\int_{\R^n}\|\Phi^m(\x)\|^2\,\md\x=1$, the sequence $(\Phi^m)_{m\in\N}=((\varphi_1^m,\ldots,\varphi_H^m))_{m\in\N}$ is bounded in $(H^1(\R^n)\cap L^2_w(\R^n))^H$. Hence, up to extraction of a subsequence, it converges weakly in $H^1(\R^n)^H$, strongly in $L^2_{loc}(\R^n)^H$, and almost everywhere in~$\R^n$, to a certain limit $\Phi^\infty:=(\varphi_{1}^\infty,\ldots,\varphi_{H}^\infty)\in(H^1(\R^n)\cap L^2_w(\R^n))^H$, which is non\-negative componentwise. The boundedness of the sequences $(\int_{\R^n}\|\x\|^2(\varphi_i^m(\x))^2\,\md\x)_{m\in\N}$ also implies (as for the function $\phi_\infty$ after~\eqref{phim}) that the sequence $(\Phi^m)_{m\in\N}$ converges to~$\Phi^\infty$ in~$L^2(\R^n)^H$, and that
$$\int_{\R^n}\|\Phi^\infty(\x)\|^2\md\x=1.$$
By passing to the limit weakly in~\eqref{eq:eigenvalue_pb} (with parameters $(\delta^m,\alpha^m,\mu^m,\opt^m_1,\ldots,\opt^m_H)$), it follows that the pair $(\lambda,\Phi^\infty)$ is a weak solution of~\eqref{eq:eigenvalue_pb}, with $\Phi^\infty$ being norma\-lised in~$L^2(\R^n)^H$ and nonnegative componentwise. From elliptic regularity, $\Phi^\infty$ is actually a~$C^\infty(\R^n)^H$ solution of this problem and, from the strong elliptic maximum principle, it is positive componentwise in $\R^n$. By uniqueness of the principal eigenpair of~\eqref{eq:eigenvalue_pb}, one infers that $\Phi^\infty=\Phi$ and $\lambda=\lambda_H(\delta,\alpha,\mu,\opt_1,\ldots,\opt_H)$. By uniqueness of the limit, it also follows that $\Phi^m\to\Phi$ in $L^2(\R^n)^H$ as $m\to+\infty$. One has then shown the continuity of the map $(\delta,\alpha,\mu,\opt_1,\ldots,\opt_H)\mapsto\lambda_H(\delta,\alpha,\mu,\opt_1,\ldots,\opt_H)$ in $(0,+\infty)^3\times(\R^n)^H$. Together with the inequalities~\eqref{ineqs} and the definition~\eqref{deflambdaH0}, one concludes that the map~$(\delta,\alpha,\mu,\opt_1,\ldots,\opt_H)\mapsto\lambda_H(\delta,\alpha,\mu,\opt_1,\ldots,\opt_H)$ is continuous in~$[0,+\infty)\times(0,+\infty)^2\times(\R^n)^H$.

Next, since the quantity $Q_H(\Psi)$ in~\eqref{eq:rayleigh_Q} is affine with respect to $(\delta,\alpha,\mu)\in(0,+\infty)^3$ for each $\Psi\in(H^1(\R^n)\cap L^2_w(\R^n))^H$, one gets from~\eqref{Rayleigh} that the map $(\delta,\alpha,\mu)\mapsto\lambda_H(\delta,\alpha,\mu,\opt_1,\ldots,\opt_H)$ is concave in $(0,+\infty)^3$, for each $(\opt_i)_{1\le i\le H}\in(\R^n)^H$. Together with the continuity proved in the previous paragraph, one infers that the map $(\delta,\alpha,\mu)\mapsto\lambda_H(\delta,\alpha,\mu,\opt_1,\ldots,\opt_H)$ is concave in $[0,+\infty)\times(0,+\infty)^2$.

Consider now $\delta\ge0$, $\alpha>0$, $0<\mu_1<\mu_2$, $(\opt_i)_{1\le i\le H}\in(\R^n)^H$, and let us show that $\lambda_H(\delta,\alpha,\mu_1,\opt_1,\ldots,\opt_H)<\lambda_H(\delta,\alpha,\mu_2,\opt_1,\ldots,\opt_H)$. The case $\delta=0$ is obvious by~\eqref{deflambdaH0}, so let us assume that $\delta>0$. Let $\Phi_1$ and $\Phi_2$ be the principal eigenvectors associated to the principal eigenvalues $\lambda_H(\delta,\alpha,\mu_1,\opt_1,\ldots,\opt_H)$ and $\lambda_H(\delta,\alpha,\mu_2,\opt_1,\ldots,\opt_H)$, respectively, and let $Q_{H,1}$ and $Q_{H,2}$ be the functionals defined in~\eqref{eq:rayleigh_Q}, associated with $\mu_1$ and $\mu_2$, respectively. Since no component of the vector $(\varphi_1,\ldots,\varphi_H):=\Phi_2$ is constant, one has $\int_{\R^n}\|\nabla\varphi_i(\x)\|^2\,\md\x>0$ for each $1\le i\le H$, hence
$$\lambda_H(\delta,\alpha,\mu_2,\opt_1,\ldots,\opt_H)=Q_{H,2}(\Phi_2)>Q_{H,1}(\Phi_2)\ge\lambda_H(\delta,\alpha,\mu_1,\opt_1,\ldots,\opt_H).$$

The strict increasing monotonicity of $\lambda_H(\delta,\alpha,\mu,\opt_1,\ldots,\opt_H)$ with respect to $\alpha>0$ can be shown similarly. Lastly, the monotonicity of $\lambda_H(\delta,\alpha,\mu,\opt_1,\ldots,\opt_H)$ with respect to $\delta\ge0$ will actually be shown independently in Proposition~\ref{prop:fitness-gain-small-delta} in Section~\ref{sec53}.
\end{proof}

\begin{proof}[Proof of Proposition~$\ref{prop:lambda-infini}$]
The fist part leading to~\eqref{eq-large-delta} is actually a straightforward adaptation of the proof of \cite[Proposition~4~$(i)$]{HamLavRoq20} performed in the case $H=2$. Details are omitted. Next, since $R_H(\x)= \rmax -\alpha\|\x-\mathbf{M}\|^2/2+R_H(\mathbf{M})-\rmax$, we get~\eqref{lambda-delta-gd-centre-masse}.
\end{proof}

%%%%%%%%%%%%%%%%%%%%%%%%%%%%%%%%%%%%%%%%%%%%%%%%%%%%%%%%

\subsection{Proofs of Propositions~\ref{prop:fitness-gain-small-delta}-\ref{prop:projection}}\label{sec53}

\begin{proof}[Proof of Proposition~$\ref{prop:fitness-gain-small-delta}$]
The parameters $\alpha>0$ and $\mu>0$ are fixed throughout the proof, as are the optima $(\opt_i)_{1\le i\le H}\in(\R^n)^H$. Let us first show that the map $\delta\mapsto\lambda_H(\delta)$ is differentiable in $(0,+\infty)$. For $\delta>0$ and $\Psi=(\psi_1,\ldots,\psi_H)\in(H^1(\R^n)\cap L^2_{w}(\R^n))^H,$ we denote by $Q_H(\Psi,\delta)$ the Rayleigh quotient~\eqref{eq:rayleigh_Q}. We also denote $\Phi^{\delta}=(\varphi_{1}^\delta,\ldots,\varphi_{H}^\delta)$ the principal (normalised) eigenvector associated with $\lambda_H(\delta)$, so that $\lambda_H(\delta)=Q_H(\Phi^{\delta},\delta)=\min_{\Psi\in(H^1(\R^n)\cap L^2_{w}(\R^n))^H}Q_H(\Psi,\delta)$.

Fix any $\delta>0$. We notice that, for $\varepsilon>0$,
\begin{equation}\label{notice}
\frac{Q_H(\Phi^{\delta\!+\!\varepsilon},\delta\!+\!\varepsilon)\!-\!Q_H(\Phi^{\delta\!+\!\varepsilon},\delta)}{\varepsilon}\le\frac{\lambda_H(\delta\!+\!\varepsilon)\!-\!\lambda_H(\delta)}\varepsilon \le \frac{Q_H(\Phi^{\delta},\delta\!+\!\varepsilon)\!-\!Q_H(\Phi^{\delta},\delta)}{\varepsilon},
\end{equation}
together with
\begin{equation}\label{notice2}
\frac{Q_H(\Phi^{\delta+\varepsilon},\delta+\varepsilon)-Q_H(\Phi^{\delta+\varepsilon},\delta)}{\varepsilon}=1-\!\sum_{1\le i< j\le H}\!\frac{2}{H-1}\int_{\R^n}\varphi_{i}^{\delta+\varepsilon}(\x)\,\varphi_{j}^{\delta+\varepsilon}(\x)\,\md\x
\end{equation}
and
\begin{equation}\label{notice3}
\frac{Q_H(\Phi^{\delta},\delta+\varepsilon)-Q_H(\Phi^{\delta},\delta)}{\varepsilon}=1-\!\sum_{1\le i< j\le H}\!\frac{2}{H-1}\int_{\R^n}\varphi_{i}^{\delta}(\x)\,\varphi_{j}^{\delta}(\x)\,\md\x.
\end{equation}
From the proof of the continuity property in Proposition~\ref{prop monolambda}, one knows that $\Phi^{\delta+\varepsilon}\to\Phi^\delta$ in $L^2(\R^n)^H$ as $\varepsilon\to0$. Thus, inserting~\eqref{notice2} and~\eqref{notice3} into~\eqref{notice} and passing to the limit as~$\varepsilon\to0^+$, we get
$$\frac{\lambda_H(\delta+\varepsilon)-\lambda_H(\delta)}{\varepsilon}\mathop{\longrightarrow}_{\varepsilon\to0^+}1-\sum _{1\le i<j\le h}\frac{2}{H-1}\int_{\R^n}\varphi_{i}^{\delta}(\x)\,\varphi_{j}^{\delta}(\x)\,\md\x.$$
By reverting the inequalities in~\eqref{notice} when $\varepsilon\in(-\delta,0)$, and using the arbitrariness of $\delta>0$, one then concludes that the map $\delta\mapsto\lambda_H(\delta)$ is differentiable in $(0,+\infty)$, with
\begin{equation}
\label{der-lambda-H}
\lambda_H'(\delta)=1-\sum _{1\le i<j\le H}\frac{2}{H-1}\int_{\R^n}\varphi_{i}^{\delta}(\x)\,\varphi_{j}^{\delta}(\x)\,\md\x\,<\,1\ \hbox{ for all }\delta>0.
\end{equation}

Notice also that, if the optima $\opt_i$ are all identical, then by uniqueness the functions~$\varphi^\delta_{i}$ are also all identical, hence $\int_{\R^n}\varphi_{i}^{\delta}(\x)\,\varphi_{j}^{\delta}(\x)\,\md\x=1/H$ for all $1\le i\le H$ (remember that $\Phi^{\delta}=(\varphi_{1}^\delta,\ldots,\varphi_{H}^\delta)$ has unit $L^2(\R^n)^H$ norm) and $\lambda'_H(\delta)=0$ for all $\delta>0$. On the other hand, if the optima $\opt_i$ are not all identical, then, for any $\delta>0$, the functions $\varphi^\delta_{i}$ are not all identical: otherwise, each $\varphi^\delta_{i}\in H^1(\R^n)\cap L^2_w(\R^n)$ would be a classical positive solution of $-(\mu^2/2)\Delta\varphi^\delta_{i}(\x)-r_i(\x)\,\varphi^\delta_{i}(\x)=\lambda_H(\delta)\,\varphi^\delta_{i}(\x)$ in $\R^n$ and then $r_i(\x)\equiv r_j(\x)$ in~$\R^n$ and $\opt_i=\opt_j$ for all $1\le i,j\le H$, a contradiction. Thus, for every $\delta>0$, there are $1\le i_0<j_0\le H$ such that $\varphi^\delta_{i_0}\not\equiv\varphi^\delta_{j_0}$ in $\R^n$, hence
$$\int_{\R^n}\varphi_{i_0}^{\delta}(\x)\,\varphi_{j_0}^{\delta}(\x)\,\md\x<\frac12\int_{\R^n}(\varphi_{i_0}^{\delta}(\x))^2\,\md\x+\frac12\int_{\R^n}(\varphi_{j_0}^{\delta}(\x))^2\,\md\x,$$
while the large inequality always holds for every $1\le i,j\le H$. Therefore,~\eqref{der-lambda-H} and the normalisation of  $\Phi^{\delta}=(\varphi_{1}^\delta,\ldots,\varphi_{H}^\delta)$ imply that $\lambda'_H(\delta)>0$ for all $\delta>0$.

Let us now show the differentiability of $\delta\mapsto\lambda_H(\delta)$ at $\delta=0$. We recall that
$$\lambda_H(0)=\lambda_1=\frac{\mu n\sqrt{\alpha}}{2}-\rmax=\frac{\mu^2}{2}\int_{\R^n}\|\nabla G_i(\x)\|^2\,\md\x-\int_{\R^n}r_i(\x)\,(G_i(\x))^2\,\md\x$$
for every $1\le i\le H$, with $G_i(\x)=G(\x-\opt_i)$ and $G$ defined in~\eqref{def:gaussian}. We also observe that Proposition~\ref{prop bound lambda} implies that the function $\delta\mapsto\lambda_H(\delta)$ is continuous at $0$, and that $0\le\liminf_{\delta\to0^+}(\lambda_H(\delta)-\lambda_H(0))/\delta\le\limsup_{\delta\to0^+}(\lambda_H(\delta)-\lambda_H(0))/\delta<1$ (the last strict inequality comes from the third inequality in~\eqref{ineqs}).

Consider now any sequence $(\varepsilon_m)_{m\in\N}$ of positive real numbers converging to $0$. The same arguments as in the proof of the continuity property in Proposition~\ref{prop monolambda} imply that, up to extraction of a subsequence, the sequence $(\Phi^{\varepsilon_m})_{m\in\N}=((\varphi_{1}^{\varepsilon_m},\ldots,\varphi_{H}^{\varepsilon_m}))_{m\in\N}$ converges weakly in $H^1(\R^n)^H$, strongly in $L^2(\R^n)^H$, and almost everywhere in $\R^n$, to a certain limit $\Phi:=(\varphi_{1},\ldots,\varphi_{H})\in(H^1(\R^n)\cap L^2_w(\R^n))^H$, which is nonnegative componentwise and satisfies $\int_{\R^n}\|\Phi(\x)\|^2\,\md\x=1$. Furthermore, each function $\varphi_i$ is a $C^\infty(\R^n)\cap H^1(\R^n)\cap L^2_w(\R^n)$ nonnegative solution of
$$-\frac{\mu^2}{2}\Delta\varphi_i(\x)-r_i(\x)\,\varphi_i(\x)=\lambda_1\,\varphi_i(\x)\ \hbox{ in }\R^n.$$
Therefore, by uniqueness of the principal eigenfunction for this equation, it follows that there is $p_i\in[0,+\infty)$ such that
$$\varphi_i(\x)=p_i\,G_i(\x)=p_i\,G(\x-\opt_i)\ \hbox{ for all $\x\in\R^n$},$$
hence
\begin{equation}\label{gaussianvector}
(\varphi_{1}^{\varepsilon_m},\ldots,\varphi_{H}^{\varepsilon_m})\to(p_1\,G_1,\ldots,p_H\,G_H)\ \hbox{ in $L^2(\R^n)^H$ as $m\to+\infty$}.
\end{equation}Since $\Phi$ is normalised in $L^2(\R^n)^H$, one also has
$$p_1^2+\cdots+p_H^2=1.$$
Now, consider any $\qg=(q_1,\ldots,q_H)\in\R^H$ with $\|\qg\|=1$, and call $\Psi:=(q_1G_1,\ldots,q_HG_H)\in(H^1(\R^n)\cap L^2_w(\R^n))^H$. Notice that $\int_{\R^n}\|\Psi(\x)\|^2\,\md\x=1$. Observe also that each $\varphi^{\varepsilon_m}_{H,i}$ belongs to $H^1(\R^n)\cap L^2_w(\R^n)$, hence
$$\begin{array}{l}
\ds\lambda_1\!\!\int_{\R^n}\!\!(\varphi^{\varepsilon_m}_{i}(\x))^2\md\x\vspace{3pt}\\
\qquad\ds\le\underbrace{\frac{\mu ^2}{2}\!\!\int_{\R^{n}}\!\!\|\nabla\varphi_{i}^{\varepsilon_m}(\x)\|^2\md\x\!-\!\rmax\int_{\R^n}(\varphi^{\varepsilon_m}_i(\x))^2\md\x+\frac{\alpha}{2}\!\int_{\R^{n}}\!\!\|\x-\opt_i\|^2(\varphi_{i}^{\varepsilon_m}(\x))^{2}(\x)\md\x}_{=:Q_1(\varphi^{\varepsilon_m}_{i},0)}\end{array}$$
and $\lambda_H(0)=\lambda_1\le Q_H(\Phi^{\varepsilon_m},0)$ by summing the inequalities from $i=1$ to $H$. The inequality $\lambda_H(\varepsilon_m)=Q_H(\Phi^{\varepsilon_m},\varepsilon_m)\le Q_H(\Psi,\varepsilon_m)$ and the equality $Q_H(\Psi,0)=\lambda_1=\lambda_H(0)$ then imply that
$$\frac{Q_H(\Phi^{\varepsilon_m},\varepsilon_m)-Q_H(\Phi^{\varepsilon_m},0)}{\varepsilon_m}\le\frac{\lambda_H(\varepsilon_m)-\lambda_H(0)}{\varepsilon_m}\le\frac{Q_H(\Psi,\varepsilon_m)-Q_H(\Psi,0)}{\varepsilon_m}$$
for every $m\in\N$, hence, as in~\eqref{notice2}-\eqref{notice3},
\begin{equation}\label{varepsm}\begin{array}{l}
\ds1-\!\sum_{1\le i< j\le H}\!\frac{2}{H-1}\int_{\R^n}\varphi_{i}^{\varepsilon_m}(\x)\,\varphi_{j}^{\varepsilon_m}(\x)\,\md\x\vspace{3pt}\\
\qquad\qquad\ds\le\frac{\lambda_H(\varepsilon_m)-\lambda_H(0)}{\varepsilon_m}\le1-\!\sum_{1\le i< j\le H}\!\frac{2}{H-1}\int_{\R^n}q_i\,q_j\,G_i(\x)\,G_j(\x)\,\md\x.\end{array}
\end{equation}
Using~\eqref{gaussianvector} and passing to the limit as $m\to+\infty$ in the left-hand side of the previous inequality leads to
$$\sum_{1\le i< j\le H}\int_{\R^n}q_i\,q_j\,G_i(\x)\,G_j(\x)\,\md\x\le\sum_{1\le i< j\le H}\int_{\R^n}p_i\,p_j\,G_i(\x)\,G_j(\x)\,\md\x,$$
hence
$$\sum_{1\le i, j\le H}\int_{\R^n}q_i\,q_j\,G_i(\x)\,G_j(\x)\,\md\x\le\sum_{1\le i, j\le H}\int_{\R^n}p_i\,p_j\,G_i(\x)\,G_j(\x)\,\md\x,$$
since each $G_i=G(\cdot-\opt_i)$ has unit $L^2(\R^n)$ norm and since both vectors $\pg:=(p_1,\ldots,p_H)$ and $\qg=(q_1,\ldots,q_H)$ of $\R^H$ have unit Euclidean norm. Denote $A=(a_{ij})_{1\le i,j\le H}$ the symmetric matrix whose entries are
\begin{equation}\label{entries-A}
a_{ij}:=\int_{\R^n}G_i(\x)\,G_j(\x)\,\md\x=e^{-\sqrt{\alpha}\,\|\opt_i-\opt_j\|^2/(4\mu)}>0,
\end{equation}
by~\eqref{prod-sca-gauss}. One has then obtained that $\qg A\qg\le\pg A\pg$ for all $\qg\in\R^H$ with unit Euclidean norm (namely the quadratic form associated to $A$ is maximal on the unit Euclidean sphere of~$\R^H$ at the normalised vector $\pg$). From Perron-Frobenius theorem, since $a_{ij}>0$ and~$p_i\ge0$ for all $1\le i,j\le H$, it follows that $\pg A\pg$ is the largest eigenvalue $\mu_A$ of $A$, that~$\pg$ is the unique normalised principal eigenvector of the matrix $A$, and that $p_i>0$ for all $1\le i\le H$. By uniqueness of the limit, one concludes that $\Phi^\varepsilon\to(p_1G_1,\ldots,p_HG_H)$ in~$L^2(\R^n)^H$ as~$\varepsilon\to0^+$. Furthermore, by using $\qg=\pg$ and $\varepsilon$ instead of $\varepsilon_m$ in~\eqref{varepsm}, and by passing to the limit as~$\varepsilon\to0^+$, one gets that $\delta\mapsto\lambda_H(\delta)$ is differentiable at $0$, with
\begin{equation}\label{deflambda'}
\lambda_H'(0)=1-\!\sum_{1\le i< j\le H}\!\frac{2}{H-1}\int_{\R^n}p_i\,p_j\,G_i(\x)\,G_j(\x)\,\md\x=1-\!\frac{1}{H-1}\sum_{1\le i\neq j\le H}\!p_ip_ja_{ij}.
\end{equation}
But $\sum_{1\le i\neq j\le H}\!p_ip_ja_{ij}=\pg A\pg-\sum_{1\le i\le H}p_i^2=\mu_A-1$. Finally,
\begin{equation}\label{deflambda'2}
\lambda_H'(0)=1-\frac{\mu_A-1}{H-1}=\frac{H-\mu_A}{H-1}.
\end{equation}

Observe that~\eqref{deflambda'} confirms that $\lambda'_H(0)<1$ (since $a_{ij}>0$ and $p_i>0$ for all $1\le i,j\le H$). Furthermore, if the optima $\opt_i$ are all identical, then we already know that $\lambda'_H(\delta)=0$ for all $\delta>0$, hence the function $\delta\mapsto\lambda_H(\delta)$, which is continuous in~$[0,+\infty)$ is constant in~$[0,+\infty)$, thus $\lambda'_H(0)=0$ too and $\lambda_H(\delta)=\lambda_H(0)=\lambda_1=\mu\,n\,\sqrt{\alpha}/2-\rmax$ for all $\delta\ge0$. On the other hand, if the optima $\opt_i$ are not all identical, then the functions $G_i$ are not all identical and there are $1\le i_0<j_0\le H$ such that $G_{i_0}\not\equiv G_{j_0}$, hence $p_{i_0}\,G_{i_0}\not\equiv p_{j_0}\,G_{j_0}$ (remember that $p_i>0$ for all $1\le i\le H$ and that the functions $G_i$ have all unit $L^2(\R^n)$ norm) and
$$\int_{\R^n}p_{i_0}\,p_{j_0}\,G_{i_0}(\x)\,G_{j_0}(\x)\,\md\x<\frac12\int_{\R^n}(p_{i_0}\,G_{i_0}(\x))^2\,\md\x+\frac12\int_{\R^n}(p_{j_0}\,G_{j_0}(\x))^2\,\md\x=\frac{p_{i_0}^2+p_{j_0}^2}{2},$$
while the large inequality holds for all $1\le i,j\le H$. Therefore,~\eqref{deflambda'} and the equality $p_1^2+\cdots+p_H^2=1$ imply that $\lambda'_H(0)>0$.

Lastly, with $\opt_1,\opt_2$ as in~\eqref{O1-O2}, formula~\eqref{deflambda'2} implies that
$$\lambda_2'(0)=1-e^{-\sqrt{\alpha}\beta^2/\mu}$$
and
$$\lambda'_3(0)=1-\sqrt{\frac{a_{12}^2+a_{13}^2+a_{23}^2}{3}}\times\cos\left[\frac13\,\arccos\left(a_{12}a_{13}a_{23}\sqrt{\frac{27}{(a_{12}^2+a_{13}^2+a_{23}^2)^3}}\right)\right],$$
from the del Ferro-Cardan formula for the roots of the characteristic polynomial of the matrix $A$ (notice that the argument inside the $\arccos$ ranges in $(0,1]$ from the positivity of the $a_{ij}$'s and the arithmetic-geometric mean inequality). After some algebraic transformations based on \eqref{entries-A}, one infers that $\mathcal{P}:=\{\opt_3=(a_1,\ldots,a_n)\in\R^n:\lambda'_3(0)<\lambda'_2(0)\}$ is given by formula~\eqref{patatoide}, and that $\{\opt_3=(a_1,\ldots,a_n)\in\R^n:\lambda'_3(0)>\lambda'_2(0)\}=\R^n\setminus\overline{\mathcal{P}}$. The last conclusion of Proposition~\ref{prop:fitness-gain-small-delta} immediately follows, and the proof is thereby complete.
\end{proof}

\begin{proof}[Proof of Proposition~$\ref{prop equilateral}$] Denote $\lambda_3^e=\lambda_3(\opt _1,\opt _2, \opt _3)$, with $\opt_1=(-\beta,0,\ldots,0)$, $\opt_2=(\beta,0,\ldots,0)$, $\opt _3=(0,\sqrt 3 \beta,0, \ldots,0)$, and $\beta>0$. We first show some symmetry properties. Let $(\varphi_1,\varphi_2,\varphi_3)$ be the normalised principal eigenvector associated with the principal eigenvalue $\lambda_3^e$. Define the symmetry
\begin{equation}\label{defiota}
\iota(\x)=\iota(x_1, x_2,\ldots,x_n):=(-x_1, x_2,\ldots,x_n).
\end{equation}
As $r_2\circ \iota=r_1$ and  $r_3 \circ\iota=r_3$, we get that $(\varphi_2 \circ \iota,\varphi_1 \circ \iota,\varphi_3 \circ \iota)$ is also a normalised positive eigenvector. By uniqueness, it follows that
$$\varphi_2=\varphi_1 \circ \iota\  \hbox{ and }\ \varphi_3=\varphi_3 \circ \iota.$$
Now, define
$$\kappa(\x)=\kappa(x_1, x_2,\ldots,x_n):=\lp \frac{1}{2}(x_1-\beta)+\frac{\sqrt{3}}{2}x_2,\frac{\sqrt{3}}{2}(x_1+\beta)-\frac{1}{2}x_2 ,x_3,\ldots,x_n\rp,$$
which is the orthogonal affine reflection with respect to the hyperplane parallel to the directions $x_3,\ldots,x_n$ and containing the line joining $\opt_1$ and $(\opt_2+\opt_3)/2$. We have
$$r_1\circ \kappa=r_1, \ r_2\circ \kappa=r_3, \ r_3\circ \kappa=r_2.$$
Additionally, one has $\Delta (\varphi \circ \kappa)(\x)=(\Delta\varphi)(\kappa(\x))$ for all $\varphi\in C^2(\R^n)$ and $\x\in\R^n$. Thus, $(\varphi_1\circ \kappa, \varphi_3\circ \kappa, \varphi_2\circ \kappa)$ is also a normalised positive eigenvector. Again, by uniqueness, one infers that
$$\varphi_1=\varphi_1 \circ \kappa, \  \varphi_2=\varphi_3 \circ \kappa, \hbox{ and }\varphi_3=\varphi_2 \circ \kappa.$$

Consider now the equation satisfied by $\varphi_1$, namely
\begin{equation}\label{eq:phi1_H3}
-\frac{\mu^2}{2}\Delta \varphi_1(\x) - r_1(\x) \varphi_1(\x)- \delta \left(   \frac {\varphi_2(\x)+\varphi_3(\x)}{2} -   \varphi_1(\x) \right) = \lambda_3^e\,\varphi_1(\x).
\end{equation}
Multiplying by $\varphi_1$ and integrating by parts, we get
$$\begin{array}{rcl}
\ds\int_{\R^n}\left(\frac{\mu^2}{2}\|\nabla\varphi_1(\x)\|^2-r_1(\x)\,(\varphi_1(\x))^2\right)\md\x+ \delta \, \int_{\R^n} (\varphi_1(\x))^2\,\md\x & & \vspace{3pt}\\
\ds- \frac{\delta}{2}  \, \int_{\R^n} \varphi_1(\x)\, \varphi_2(\x)\,\md\x- \frac{\delta}{2}  \, \int_{\R^n} \varphi_1(\x)\,\varphi_3(\x)\,\md\x  & = & \ds\lambda_3^e \int_{\R^n}(\varphi_1(\x))^2\,\md\x.\end{array}$$
Moreover,
$$\int_{\R^n} \varphi_1(\x)\, \varphi_3(\x)\,\md\x=\int_{\R^n} (\varphi_1\circ\kappa )(\x)\, (\varphi_2\circ\kappa)(\x)\,\md\x=\int_{\R^n} \varphi_1(\x)\, \varphi_2(\x)\,\md\x.$$
Finally, we have
\begin{equation}\label{eq:phi1_iota}\begin{array}{rcl}
\ds\int_{\R^n}\!\!\!\left(\frac{\mu^2}{2}\|\nabla \varphi_1(\x)\|^2 - r_1(\x)(\varphi_1(\x))^2 \right)\md\x+ \delta\! \int_{\R^n}\!(\varphi_1(\x))^2 \,\md\x & \!\!\!\! & \vspace{3pt}\\
\ds- \delta\!\int_{\R^n} \varphi_1(\x)\, \varphi_2 (\x) \,\md\x & \!\!=\!\! & \ds\lambda_3^e\!\int_{\R^n}\!(\varphi_1(\x))^2\, \md\x,\end{array}
\end{equation}
and, by symmetry (using a composition with the function $\iota$),
\begin{equation}\label{eq:phi2_iota}\begin{array}{rcl}
\ds\int_{\R^n}\!\!\!\left(\frac{\mu^2}{2}\|\nabla \varphi_2(\x)\|^2\!-\!r_2(\x)(\varphi_2(\x))^2\right)\,\md\x\!+\!\delta\! \int_{\R^n}\!(\varphi_2(\x))^2\,\md\x & \!\!\!\! & \vspace{3pt}\\
\ds- \delta\,\int_{\R^n} \varphi_1(\x)\,\varphi_2(\x)\,\md\x  & \!\!=\!\! & \ds\lambda_3^e\int_{\R^n}(\varphi_2(\x))^2\, \md\x.\end{array}
\end{equation}
Adding \eqref{eq:phi1_iota} and \eqref{eq:phi2_iota}, we observe that
\begin{equation}\label{Q2}
Q_2\lp \frac{\varphi_1}{\sqrt{\int_{\R^n} (\varphi_1^2+\varphi_2^2)(\x) \,\md\x} }, \frac{\varphi_2}{\sqrt{\int_{\R^n} (\varphi_1^2+\varphi_2^2)(\x) \,\md\x}}\rp= \lambda_3^e,
\end{equation}
where $Q_2$ is the Rayleigh quotient defined by \eqref{eq:rayleigh_Q} with $H=2$. Since $(\varphi_1,\varphi_2)\in(H^1(\R^n)\cap L^2_w(\R^n))^2$, it follows from~\eqref{Rayleigh} that $\lambda_3^e \ge \lambda_2(\opt_1,\opt_2)$.

Finally, if $\lambda_3^e$ were equal to $\lambda_2(\opt_1,\opt_2)$, since the principal eigenvector associated with the principal eigenvalue $\lambda_2(\opt_1,\opt_2)$ with $H=2$ is the unique positive normalised minimum of $Q_2$ in $(H^1(\R^n)\cap L^2_w(\R^n))^2$, it would follow from~\eqref{Q2} that $(\varphi_1,\varphi_2)$ would be (up to normalisation) a principal eigenvector of~\eqref{eq:eigenvalue_pb} with $H=2$. Together with~\eqref{eq:phi1_H3} and using the positivity of $\delta$, one would infer that $\varphi_3\equiv\varphi_2$ in $\R^n$. Similarly, by using the equation similar to~\eqref{eq:phi1_H3} satisfied by $\varphi_2$, one would get that $\varphi_3\equiv\varphi_1$ in $\R^n$. Finally, $\varphi_1\equiv\varphi_2\equiv\varphi_3$ in $\R^n$ and~\eqref{eq:phi1_H3} implies that $\varphi_1$ would be a positive multiple of the normalised principal eigenfunction $G_1=G(\cdot-\opt_1)$ of the case of only one host. But, similarly, $\varphi_2$ would be a positive multiple of $G_2=G(\cdot-\opt_2)$. This leads to a contradiction, since $\opt_1\neq\opt_2$ (because $\beta>0$). Therefore, $\lambda_3(\opt_1,\opt_2,\opt_3)=\lambda_3^e>\lambda_2(\opt_1,\opt_2)$ and the proof of Proposition~\ref{prop equilateral} is thereby complete.
\end{proof}

\begin{proof}[Proof of Proposition~$\ref{prop:byebyeO3}$]
With the notation~\eqref{eq:eigenvalue_pb_loss}, evaluating the Rayleigh quotient asso\-ciated with $\lambda_3=\lambda_3(\opt_1,\opt_2,\opt_3)$ at $\tilde\Phi:=(\tphi_1,\tphi_2,0)$, we get
\begin{equation}\label{eq:lambdatilde2a}
\lambda_3\leq Q_3(\tilde\Phi)=\tilde{\lambda}_2=\tilde{\lambda}_2(\opt_1,\opt_2).
\end{equation}
On the other hand, as the minimum in the Rayleigh quotient $Q_3$ is reached at $\Phi=(\varphi_1,\varphi_2,\varphi_3)$, the normalised eigenvector associated with $\lambda_3$, we have
$$\begin{array}{rcl}
\lambda_3  & \!\!\ge\!\! & \ds\min_{\substack{(\psi_1,\psi_2)\in\mathcal{E}^2 \\ \int_{\R^n}\psi_1^2+\psi_2^2=1-\int_{\R^n}\varphi_3^2}} \left\{ \frac{\mu ^2}{2}\int_{\R^{n}}\|\nabla \psi_1\|^2+\|\nabla \psi_2\|^2- \int_{\R^{n}}(r_1 \, \psi_1^2+r_2 \, \psi_2^2) -\delta \int_{\R^{n}} \psi_1 \, \psi_2  \right\}\\
& \!\!\!\! & \ds+\min_{\substack{\psi\in\mathcal{E} \\ \int_{\R^n}\psi^2=\int_{\R^n}\varphi_3^2}} \left\{ \frac{\mu ^2}{2}\int_{\R^{n}}\|\nabla\psi\|^2- \int_{\R^{n}} r_3 \, \psi ^2 \right\}+\delta-\delta \int_{\R^{n}} \varphi_3 \, (\varphi_1+\varphi_2),\end{array}$$
where $\mathcal{E}=H^1(\R^n)\cap L^2_w(\R^n)$. As the first minimum in the above formula is precisely $(\tilde{\lambda}_2-\delta)\lp1-\int_{\R^n}\varphi_3^2 \rp,$ and the second minimum is $\lambda_1 \, \int_{\R^n}\varphi_3^2$, we get
\begin{equation}\label{eq:za}
\lambda_3\ge  (\tilde{\lambda}_2-\delta)\lp1-\int_{\R^n}\varphi_3^2 \rp + \lambda_1 \, \int_{\R^n}\varphi_3^2    +\delta-\delta \int_{\R^{n}} \varphi_3 \, (\varphi_1+\varphi_2).
\end{equation}

Our goal is now to show that $\int_{\R^{n}} \varphi_3 \, (\varphi_1+\varphi_2)$ becomes small as $\|\opt_3\|\to +\infty$. Since $\lambda_3=Q_3(\Phi),$ we have
\begin{equation*}
\lambda_3 > -\sum_{i=1}^3 \int_{\R^n}r_i\varphi_i^2+ \delta - \delta \, \sum_{1\le i<j\le3}\int_{\R^n} \varphi_i \, \varphi_j\geq -\sum_{i=1}^3 \int_{\R^n}r_i\varphi_i^2
\end{equation*}
since $\int_{\R^n}\|\Phi\|^2=1$. Using \eqref{eq:lambdatilde2a} and  the definition \eqref{def ri} of $r_i(\x)$, we get
\begin{equation*}
\alpha\, \sum_{i=1}^3\int_{\R^{n}}\frac{ \norm{\x-\opt_i}^2}{2} \, (\varphi_i(\x))^2\,\md\x< \tilde{\lambda}_2+\rmax.
\end{equation*}
This implies that, for any radius $R>0$, and $i=1,2,3$,
\begin{equation*}
\|\varphi_i\|^2_{L^2(\R^n\backslash B(\opt_i,R))}<\frac{2\,(\tilde{\lambda}_2+\rmax)}{\alpha \, R^2}.
\end{equation*}
Next, we have, for $i=1,2$,
\begin{align*}
\int_{\R^{n}} \varphi_3 \, \varphi_i & = \int_{ B(\opt_i,R)} \varphi_3 \, \varphi_i + \int_{\R^{n}\backslash B(\opt_i,R)} \varphi_3 \, \varphi_i \\
& \le \|\varphi_3\|_{L^2(B(\opt_i,R))}  \|\varphi_i\|_{L^2(B(\opt_i,R))}+ \|\varphi_3\|_{L^2(\R^n\backslash B(\opt_i,R))}\|\varphi_i\|_{L^2(\R^n\backslash B(\opt_i,R))}.
\end{align*}
Taking $R=\min(\norm{\opt_3-\opt_1}/2,\norm{\opt_3-\opt_2}/2)$, we have $B(\opt_i,R)\subset \R^n\backslash B(\opt_3,R)$ for $i=1,2$. Thus, we get
\begin{align}
\int_{\R^{n}} \varphi_3 \, (\varphi_1+\varphi_2)\le &  \ \|\varphi_3\|_{L^2(\R^n\backslash B(\opt_3,R))}(\|\varphi_1\|_{L^2(\R^n)}+\|\varphi_2\|_{L^2(\R^n)}) \nonumber \\
& \ + \|\varphi_3\|_{L^2(\R^n)}\,(\|\varphi_1\|_{L^2(\R^n\backslash B(\opt_1,R))}+\|\varphi_2\|_{L^2(\R^n\backslash B(\opt_2,R))})\nonumber\\
\le & \ \sqrt{\frac{2(\tilde{\lambda}_2+\rmax)}{\alpha \, R^2}}\,(\|\varphi_1\|_{L^2(\R^n)}+\|\varphi_2\|_{L^2(\R^n)}+2 \|\varphi_3\|_{L^2(\R^n)}) \label{eq:ineg_R}\\
\le & \ \sqrt{\frac{12(\tilde{\lambda}_2+\rmax)}{\alpha \, R^2}},\nonumber
\end{align}
where the last inequality uses $\int_{\R^n}\|\Phi\|^2=\int_{\R^n}\varphi_1^2+\int_{\R^n}\varphi_2^2+\int_{\R^n}\varphi_3^2=1$ together with the Cauchy-Schwarz inequality $a+b+2c\le\sqrt{6(a^2+b^2+c^2)}$ for any $(a,b,c)\in\R^3$.

We come back to \eqref{eq:za}. Since $\int_{\R^n}\varphi_3^2 \in (0,1)$, we have
\begin{equation*}
\lambda_3\ge \inf_{a\in (0,1)}\{ (\tilde{\lambda}_2-\delta)\lp1-a \rp+ \lambda_1 \, a   \} +\delta-\delta \int_{\R^{n}} \varphi_3 \, (\varphi_1+\varphi_2).
\end{equation*}
Using \eqref{eq:ineglambda2tilde_lambda1}, namely $\tilde{\lambda}_2-\delta\le \lambda_1$, we get
$$\lambda_3\ge \tilde{\lambda}_2-\delta \int_{\R^{n}} \varphi_3 \, (\varphi_1+\varphi_2).$$
Finally, we recall that $R=\min(\norm{\opt_3-\opt_1}/2,\norm{\opt_3-\opt_2}/2)$ in \eqref{eq:ineg_R} and, together with~\eqref{eq:lambdatilde2a}, we get~\eqref{loin}. The last conclusion of Proposition~\ref{prop:byebyeO3} then follows from~\eqref{eq:ineglambda2tilde_lambda1}.
\end{proof}

\begin{proof}[Proof of Proposition~$\ref{prop:O3=O1}$]
We denote $(\varphi_1,\varphi_2)$ the normalised principal eigenvector associated with $\lambda_2=\lambda_2(\opt _1,\opt_2)$. By symmetry we have $\varphi_1=\varphi_2\circ\iota$, with $\iota$ as in~\eqref{defiota}, hence $\int _{\R^n}\varphi _1^2=\int _{\R^n} \varphi _2 ^2=1/2$. Next, we take $\opt_3\in B(\opt_2,\rho)$ for some $\rho>0$ (the case $\opt_3\in B(\opt_1,\rho)$ can be handled similarly), we call $\lambda_3=\lambda_3(\opt _1,\opt_2,\opt_3)$, and we test the Rayleigh quotient $Q_3$ at $\Phi:=(a\varphi_1,a\varphi _2,b\varphi_2)$ with
$$a^2+\frac{b^2}{2}=1,$$
so that $\Phi$ is normalised. We obtain after some straightforward computations
\begin{eqnarray*}
\lambda_3\le Q_3(\Phi)&=&a^2\lambda_2+b^2\left(\frac {\mu ^2}2\int _{\R^n}\|\nabla \varphi _2\|^2-r_2\varphi_2 ^2\right)+b^2\int _{\R^n}(r_2-r_3)\varphi_2^2\\
&&\qquad+\delta\left(\int _{\R^n} b^2\varphi_2^2+a^2\varphi_1\varphi_2-ab\varphi_1\varphi_2-ab\varphi_2^2\right)\\
&=&a^2\lambda_2+b^2\left(\lambda_2 \int  _{\R^n} \varphi _2^2+\delta \int _{\R^n} (\varphi_1-\varphi_2)\varphi _2\right)+b^2\int _{\R^n}(r_2-r_3)\varphi_2^2\\
&&\qquad+\delta\left(\int _{\R^n} b^2\varphi_2^2+a^2\varphi_1\varphi_2-ab\varphi_1\varphi_2-ab\varphi_2^2\right)\\
&=& \lambda _2+\delta\left((b^2+a^2-ab)\int _{\R^n} \varphi_1\varphi _2 -ab \int _{\R^n} \varphi _2^2\right)+b^2\int _{\R^n}(r_2-r_3)\varphi_2^2.
\end{eqnarray*}
We now select $a=b=\sqrt{2/3}$ and thus obtain
$$\lambda_3-\lambda_2\leq\frac{2\delta}{3}\left(\int_{\R^n}\varphi_1\varphi_2-\int_{\R^n}\varphi_2^2\right)+\frac 23\int_{\R^n}(r_2-r_3)\varphi_2^2.$$
Since the Cauchy-Schwarz inequality yields $\int_{\R^n}\varphi_1\varphi_2-\int_{\R^n}\varphi_2^2<0$ ($\varphi_1$ and $\varphi_2$ are not colinear because $\beta \neq 0$) and since the dominated convergence theorem implies that $\int_{\R^n}(r_2-r_3)\varphi_2^2\to 0$ as $\rho\to 0$, we have $\lambda_3-\lambda_2<0$ if $\rho>0$ is small enough.
\end{proof}

\begin{proof}[Proof of Proposition~$\ref{prop:hmid}$] First, from Proposition~\ref{prop bound lambda} and Corollary~\ref{cor bound lambda bis}, we know that, whatever the position of $\opt_3$, we have
\begin{equation}\label{ineg-lambdab}
\lambda_1\le\lambda _1+\delta\left(1-\sum_{1\le i< j\le 3}\int_{\R^n}\varphi_i(\x)\varphi_j(\x)d\x\right)\leq \lambda_3 \leq \lambda _1+\delta\left(1-\frac{e^{-\sqrt{\alpha}\beta^2/\mu}}{3}\right).
\end{equation}
Set $\psi_i=\varphi_i(\cdot+\opt_i).$ We have
$$\begin{array}{rcl}
\lambda_3 & \!\!\!\!=\!\!\!\! & \!\ds-\rmax\!+\!\sum_{i=1}^3\!\int_{\R^n}\!\!\lp\!\frac{\mu^2}{2}\|\nabla \psi_i(\x)\|^2 +\frac{\alpha}{2}\,\|\x\|^2(\psi_i(\x))^2\!\rp\!\md\x\!+\!\delta\!-\!\delta\!\!\sum_{1\le i< j\le 3}\int_{\R^n}\!\!\varphi_i(\x)\varphi_j(\x)\md\x\vspace{3pt}\\
& \!\!\!\!\ge\!\!\!\! & \ds-\rmax +\sum_{i=1}^3\int_{\R^n}\!\!\lp\frac{\mu^2}{2}\|\nabla \psi_i(\x)\|^2 +\frac{\alpha}{2}\,\|\x\|^2 \, (\psi_i(\x))^2 \rp\!\md\x.\end{array}$$
Thus, from \eqref{ineg-lambdab}, we obtain that
$$\sum_{i=1}^3\int_{\R^n}\frac{\alpha}{2}\,\|\x\|^2 \,(\psi_i(\x))^2\,\md\x\le\rmax +\lambda_1+\delta\left(1-\frac{e^{-\sqrt{\alpha}\beta^2/\mu}}{3}\right).$$
Considering any radius $R>0$, the last inequality implies that
$$\|\psi_i\|_{L^2(\R^n\setminus B(\mathcal{O},R))}^2\leq\frac{2\,(\rmax +\lambda_1+\delta)}{\alpha\, R^2}$$
for every $\beta >0$ and $1\le i\le 3$.

Now, consider the particular case where $\opt_3=\Oc=(0,\ldots,0)$, so that the Euclidean distance between two optima is at least $\beta$. In this case, $\psi_3=\varphi_3$ and we have, for every $R>0$ and $\beta \ge 2 R$,
\begin{align}
\int_{\R^n}\varphi_1(\x)\,\varphi_3(\x)\,\md\x & = \int_{\R^n}\psi_1(\x-\opt_1)\,\psi_3(\x)\,\md\x  \nonumber\\
& =  \displaystyle\int_{B(\mathcal O,R)}\psi_1(\x-\opt_1)\,\psi_3(\x)\,\md\x+ \int_{\R^n\setminus B(\mathcal O,R)}\psi_1(\x-\opt_1)\,\psi_3(\x)\,\md\x \nonumber\\
& \le \|\psi_1\|_{L^2(B(\opt_1,R))}  \|\psi_3\|_{L^2(\R^n)}+ \|\psi_1\|_{L^2(\R^n)} \, \|\psi_3\|_{L^2(\R^n\setminus B(\mathcal O,R))} \nonumber\\
& \le \|\psi_1\|_{L^2(\R^n\setminus B(\mathcal O,R))}  \|\psi_3\|_{L^2(\R^n)}+ \|\psi_1\|_{L^2(\R^n)} \, \|\psi_3\|_{L^2(\R^n\setminus B(\mathcal O,R))} \nonumber\\
& \le\|\psi_1\|_{L^2(\R^n\setminus B(\mathcal O,R))}+\|\psi_3\|_{L^2(\R^n\setminus B(\mathcal O,R))} \nonumber
\end{align}
since $B(\opt_1,R)\subset \R^n\setminus B(\mathcal O,R)$ and $\|\psi_i\|_{L^2(\R^n)}=\|\varphi_i\|_{L^2(\R^n)}\le1$ for each $1\le i\le 3$. As a result
$$\int_{\R^n}\varphi_1(\x)\varphi_3(\x)d\x \leq\frac{\sqrt{8\,(\rmax +\lambda_1+\delta)}}{R \, \sqrt{\alpha}}.$$
Applying the same arguments, we finally get
$$\sum_{1\le i< j\le 3}\int_{\R^n}\varphi_i(\x)\,\varphi_j(\x)\,\md\x\le\frac{\sqrt{72\,(\rmax +\lambda_1+\delta)}}{R \, \sqrt{\alpha}}.$$
Together with \eqref{ineg-lambdab}, this shows that
$$\lambda_3(\opt_1,\opt_2,\Oc)\to \lambda_1+\delta \hbox{ as } \beta \to +\infty.$$

Consider now the case $\opt_3=\opt_1$. We note that $\lambda_3(\opt_1,\opt_2,\opt_1)=\lambda_3(\opt_1,\opt_1,\opt_2)$. Then, using the result of Proposition~\ref{prop:byebyeO3}, we obtain that $\lambda_3(\opt_1,\opt_1,\opt_2)\to \tilde{\lambda}_2(\opt_1,\opt_1)$ as $\beta\to + \infty$, where $\tilde{\lambda}_2(\opt_1,\opt_1)$ is defined by \eqref{eq:eigenvalue_pb_loss}, in the particular case where the two optima are at the same position. In such case, $\tphi_1=\tphi_2$ by uniqueness and therefore, from \eqref{eq:eigenvalue_pb_loss},
$$-\frac{\mu^2}{2}\Delta \tphi_1- r_1(\x) \tphi_1+\frac{\delta}{2}\,\tphi_1=\tilde{\lambda}_2(\opt_1,\opt_1)\, \tphi_1,$$
so that $\tilde{\lambda}_2(\opt_1,\opt_1)=\lambda_1+\delta/2$, and we are done.
\end{proof}

\begin{remark}{\rm The conclusion $\lim_{\beta\to+\infty}\lambda_3(\opt_1,\opt_2,\opt_1)=\lambda_1+\delta/2$ of Proposition~$\ref{prop:hmid}$ looks at first glance similar to the conclusion~\eqref{conclusion} derived in the proof of Proposition~$\ref{prop:statio_states}$~$(iii)$. However, the arguments used in both proofs are not the same, since the proof of Proposition~$\ref{prop:hmid}$ mainly relies on $L^2$ estimates, while that of Proposition~$\ref{prop:statio_states}$~$(iii)$ is based on $L^1$ estimates, with a reasoning done by contradiction and actually leading to the contradicting conclusions~\eqref{lambdainfty} and~\eqref{conclusion}. Notice also that the system~\eqref{eq:eigenvalue_pb_loss} used in the proof of Proposition~$\ref{prop:hmid}$ is symmetric in $(\tilde{\varphi}_1,\tilde{\varphi}_2)$, whereas the system~\eqref{eq:phi1phi2} derived in the argument by contradiction in the proof of Proposition~$\ref{prop:statio_states}$~$(iii)$ is not symmetric in $(\varphi_1,\varphi_2)$, complicating its analysis.}
\end{remark}

\begin{proof}[Proof of Proposition~$\ref{prop:projection}$]  We need to recall the notion of Steiner symmetrisation of a function with respect to the variable $x_k$. Consider first a measurable function $h:\R\to\R$, $x\mapsto h(x)$, which is either nonnegative and belongs to $L^p(\R)$ for some $1\le p<\infty$, or which is such that $h(x)\to\inf_{\R}h\in[-\infty,+\infty)$ as $|x|\to+\infty$. Then there exists a unique (in the class of functions which are identical almost everywhere)  function $h^\sharp:\R\to\R$, $x\mapsto h^\sharp(x)$, such that: $(i)$~$h^\sharp$ is symmetric with respect to $x=0$ and nonincreasing in $[0,+\infty)$, i.e., for all $x,y\in \R$, $h^\sharp(x) \ge h^\sharp(y)$ if $|x|\le|y|$; $(ii)$~$h^\sharp$ has the same distribution function as $h$, that~is,
$$\hbox{meas }\{x\in \R:  h^\sharp(x)>\alpha\}=\hbox{meas }\{x\in\R : h(x)>\alpha\}$$
for all $\alpha\in \R$, where meas denotes the one-dimensional Lebesgue measure.

Consider now a measurable function $h:\R^n\to\R$, $\x\mapsto h(\x)$, which is either nonnegative and belongs to $L^p(\R^n)$ for some $1\le p<\infty$, or which is such that $h(\x)\to\inf_{\R^n}h\in[-\infty,+\infty)$ as $\|\x\|\to+\infty$. For $1\leq k \leq n$ and $(x_1,\ldots,x_{k-1},x_{k+1},\ldots,x_n)\in\R^{n-1}$, the measurable function $x_k\mapsto h(x_1,\ldots,x_k,\ldots,x_n)$ is either nonnegative and belongs to $L^p(\R)$ (for almost every $(x_1,\ldots,x_{k-1},x_{k+1},\ldots,x_n)\in\R^{n-1}$), or is such that $h(x_1,\ldots,x_k,\ldots,x_n)\to\inf_{\R^n}h=\inf_{\R}h(x_1,\ldots,x_{k-1},\cdot,x_{k+1},\ldots,x_n)\in[-\infty,+\infty)$ as~$|x_k|\to+\infty$. We can then rearrange the function $x_k\mapsto h(x_1,\ldots,x_k,\ldots,x_n)$ as above. This corresponds to the Steiner rearrangement of $h$ with respect to the variable~$x_k$, and we denote it by $h^{\sharp_k}:\R^n\to\R$, $\x\mapsto h^{\sharp_k}(\x)$.

In the sequel, we use some known properties of this rearrangement, which we briefly recall. Firstly, for any nonnegative function $h\in L^p(\R^n)$ with $1\le p<\infty$, the nonnegative function $h^{\sharp_k}$ belongs to~$L^p(\R^n)$ too, and
\begin{equation}\label{eq:Steiner_norm}
\int_{\R^n} h^p= \int_{\R^n}(h^{\sharp_k})^p.
\end{equation}
Secondly, for any nonnegative $L^2(\R^n)$ functions $h$ and $j$, the Hardy-Littlewood inequality asserts that
\begin{equation}\label{eq:Steiner_HLP}
0\le\int_{\R^n} h \, j \le  \int_{\R^n} h^{\sharp_k}\, j^{\sharp_k}.
\end{equation}
Thirdly, the P\'olya-Szeg\"o inequality says that, for any nonnegative function $h\in W^{1,p}(\R^n)$ with $1\le p<\infty$, the function $h^{\sharp_k}$ belongs to~$W^{1,p}(\R^n)$ too, and
\begin{equation}\label{eq:Steiner_PS}
\int_{\R^n}\|\nabla h\|^p\geq  \int_{\R^n}\|\nabla h^{\sharp_k}\|^p.
\end{equation}

Equipped with this, we are in position to prove the result in Proposition~\ref{prop:projection}. For the proof, without loss of generality; using the notation~\eqref{O1-O2} and the fact that the map $\opt_3\mapsto\lambda_3(\opt_1,\opt_2,\opt_3)$ is invariant by rotation around the axis $\R\times\{0\}^{n-1}$, we can assume that $\opt_3\in\R^2\times\{0\}^{n-2}$. Let $(\varphi_1,\varphi_2,\varphi_3)\in(H^1(\R^n)\cap L^2_w(\R^n))^3$ be the normalised positive eigenvector associated with $\lambda_3=\lambda_3(\opt_1,\opt_2,\opt_3)$. Remember from~\eqref{Rayleigh} that
\begin{equation} \label{eq:lambda3_beforeprojection}
\lambda_3= \sum_{i=1}^3\int_{\R^n}\lp\frac{\mu^2}{2}\|\nabla \varphi_i\|^2 -   r_i \, \varphi_i^2 \rp +\delta - \delta \int_{\R^n} (\varphi_1 \, \varphi_2 + \varphi_2\, \varphi_3 +\varphi_1 \, \varphi_3).
\end{equation}
Consider $(\varphi_1^{\sharp_2},\varphi_2^{\sharp_2},\varphi_3^{\sharp_2})$ the Steiner symmetrisations of the eigenfunctions with respect to the variable~$x_2$. From~\eqref{eq:Steiner_norm}, we know that
\begin{equation}\label{eq:steiner_sumnorm}
\int_{\R^n}(\varphi_1^{\sharp_2})^2+(\varphi_2^{\sharp_2})^2+(\varphi_3^{\sharp_2})^2= \int_{\R^n}\varphi_1^2+\varphi_2^2+\varphi_3^2=1.
\end{equation}
Additionally, from \eqref{eq:Steiner_HLP}, we have
\begin{equation} \label{eq:HLP_2}
\int_{\R^n} (\varphi_1 \, \varphi_2 + \varphi_2\, \varphi_3 +\varphi_1 \, \varphi_3) \le  \int_{\R^n} (\varphi_1^{\sharp_2} \, \varphi_2^{\sharp_2} + \varphi_2^{\sharp_2}\, \varphi_3^{\sharp_2} +\varphi_1^{\sharp_2} \, \varphi_3^{\sharp_2}).
\end{equation}

Now, we have to prove that, for each $i=1,2,3$, the function $\varphi_i^{\sharp_2}$ belongs to $H^1(\R^n)\cap L^2_w(\R^n)$. First of all, it belongs to $H^1(\R^n)$, from~\eqref{eq:Steiner_norm} and~\eqref{eq:Steiner_PS}. Let us now show that the integral $\int_{\R^n}\|\x\|^2\,(\varphi_i^{\sharp_2}(\x))^2\,\md\x$ converges. To do so, for any $R>0$, call $b_R$ the nonnegative~$L^2(\R^n)$ function defined by $b_R(\x):=\max(R^2-\|\x\|^2,0)$ and observe that
$$\begin{array}{rcl}
\ds\int_{\R^n}\min(\|\x\|^2,R^2)\,(\varphi_i(\x))^2\,\md\x & = & \ds\int_{\R^n}R^2\,(\varphi_i(\x))^2\,\md\x-\int_{\R^n}b_R(\x)\,(\varphi_i(\x))^2\,\md\x\vspace{3pt}\\
& \ge & \ds\int_{\R^n}R^2\,(\varphi_i^{\sharp_2}(\x))^2\,\md\x-\int_{\R^n}b_R^{\sharp_2}(\x)\,(\varphi_i^{\sharp_2}(\x))^2\,\md\x\end{array}$$
from~\eqref{eq:Steiner_norm}-\eqref{eq:Steiner_HLP} (we here use the fact that $\varphi_i^2$ belongs to $L^2(\R^n)$, since it is continuous and decays to $0$ faster than exponentially as $\|\x\|\to+\infty$). Since $b_R^{\sharp_2}=b_R$, one then gets that
$$\begin{array}{rcl}
\ds\int_{\R^n}\|\x\|^2\,(\varphi_i(\x))^2\,\md\x & \!\!\!\ge\!\!\! & \ds\int_{\R^n}\min(\|\x\|^2,R^2)\,(\varphi_i(\x))^2\,\md\x\vspace{3pt}\\
& \!\!\!\ge\!\!\! & \ds\int_{\R^n}(R^2-b_R(\x))\,(\varphi_i^{\sharp_2}(\x))^2\,\md\x=\int_{\R^n}\min(\|\x\|^2,R^2)\,(\varphi_i^{\sharp_2}(\x))^2\,\md\x.\end{array}$$
The monotone convergence theorem then yields the convergence of $\int_{\R^n}\|\x\|^2\,(\varphi_i^{\sharp_2}(\x))^2\,\md\x$. Thus, $\varphi_i^{\sharp_2}\in H^1(\R^n)\cap L^2_w(\R^n)$, for each $i=1,2,3$.

Next, we claim that, for each $i=1,2,3$, the Hardy-Littlewood inequality can be applied to the couple $(r_i,\varphi_i^2)$, although $r_i$ does not have a constant sign in $\R^n$, namely we claim that
\begin{equation} \label{eq:ineg_phi_mi}
\int_{\R^n} r_i \, \varphi_i^2 \le \int_{\R^n} r_i^{\sharp_2} \, (\varphi_i^2)^{\sharp_2}= \int_{\R^n} r_i^{\sharp_2} \, (\varphi_i^{\sharp_2})^2.
\end{equation}
Notice first that all integrals in~\eqref{eq:ineg_phi_mi} converge since the functions $\varphi_i$ and $\varphi_i^{\sharp_2}$ belong to $L^2(\R^n)\cap L^2_w(\R^n)$ and since $(r_1^{\sharp_2},r_2^{\sharp_2},r_3^{\sharp_2})=(r_1,r_2,r_{\opt_3^{\sharp}})$, where
$$r_{\opt_3^{\sharp}}(\x)=\rmax-\alpha\,\frac{\|\x-\opt_3^{\sharp}\|^2}{2}$$
(the equality $r_3^{\sharp_2}=r_{\opt_3^{\sharp}}$ holds because $\opt_3\in\R^2\times\{0\}^{n-2}$). To show~\eqref{eq:ineg_phi_mi}, set, for any arbitrary $K>0$, $r_i^K(\x):=r_i(\x)+K$ if $r_i(\x)+K>0$ and $r_i^K(\x):=0$ otherwise, that is, $r_i^K(\x)=\max(r_i(\x)+K,0)$. The function $r_i^K$ is continuous and compactly supported, hence it is in $L^2(\R^n)$. From~\eqref{eq:Steiner_HLP} applied to the nonnegative $L^2(\R^n)$ functions~$r_i^K$ and~$\varphi_i^2$, together with the property $(\varphi_i^2)^{\sharp_2}=(\varphi_i^{\sharp_2})^2$, we infer that
$$0\le\int_{\R^n} r_i^K\,\varphi_i^2\le\int_{\R^n} (r_i^K)^{\sharp_2}\,(\varphi_i^{\sharp_2})^2.$$
As $(r_i^K)^{\sharp_2}=\max(r_i^{\sharp_2}+K,0)$ by definition of $r_i$ in~\eqref{def ri}, we get that
\begin{equation} \label{eq:riK}
\int_{r_i+K>0} r_i \, \varphi_i^2 + K\int_{r_i+K>0}\varphi_i^2-K\int_{r_i^{\sharp_2}+K>0}(\varphi_i^{\sharp_2})^2\le  \int_{r_i^{\sharp_2}+K>0} r_i^{\sharp_2} \, (\varphi_i^{\sharp_2})^2 .
\end{equation}
It also follows from~\eqref{eq:Steiner_norm} that the functions $\varphi_i$ and $\varphi_i^{\sharp_2}$ have the same $L^2(\R^n)$ norm, hence
\begin{equation}\label{ineqK}
K\int_{r_i+K>0}\varphi_i^2-K\int_{r_i^{\sharp_2}+K>0}(\varphi_i^{\sharp_2})^2=-K\int_{r_i+K\le 0}\varphi_i^2+K\int_{r_i^{\sharp_2}+K\le0}(\varphi_i^{\sharp_2})^2.
\end{equation}
Next, we note that there are some positive constants $C$ and $K_0$ such that
$$\big\{\x\in\R^n:r_i(\x)+K\le 0\big\}\cup\big\{\x\in\R^n:r_i^{\sharp_2}(\x)+K\le 0\big\}\ \subset\ \R^n \setminus B(\mathcal{O},C\sqrt{K})$$
for all $K\ge K_0$ and $1\le i\le 3$. Thus,~\eqref{ineqK} implies that, for all $K\ge K_0$,
\begin{equation}\label{ineqK2}
\left|K\int_{r_i+K>0}\varphi_i^2-K\int_{r_i^{\sharp_2}+K>0}(\varphi_i^{\sharp_2})^2\right|\le K\, \int_{\R^n \setminus B(\mathcal{O},C\sqrt{K})} \varphi_i^2+(\varphi_i^{\sharp_2})^2.
\end{equation}
We then use the fact that the functions $x\mapsto\|\x\|\,\varphi_i(\x)$ and $x\mapsto\|\x\|\,\varphi_i^{\sharp_2}(\x)$ are in $L^2(\R^n)$ (the property for $\varphi_i^{\sharp_2}$ follows from the previous paragraph). Hence,
$$C^2\,K\int_{\R^n \setminus B(\mathcal{O},C\,\sqrt{K})}(\varphi_i(\x))^2\,\md\x\le\int_{\R^n\setminus B(\mathcal{O},C\sqrt{K})}\|\x\|^2\,(\varphi_i(\x))^2\,\md\x\to 0\ \hbox{ as }K\to+\infty,$$
and a similar property holds for $\varphi_i^{\sharp_2}$. Finally,~\eqref{ineqK2} shows that
$$K\int_{r_i+K>0}\varphi_i^2-K\int_{r_i^{\sharp_2}+K>0}(\varphi_i^{\sharp_2})^2\to 0\ \hbox{ as }K\to+\infty.$$
Passing to the limit $K\to +\infty$ in \eqref{eq:riK} and using the dominated convergence theorem, we obtain~\eqref{eq:ineg_phi_mi}.

Finally, using~\eqref{eq:Steiner_PS}, we get that
\begin{equation} \label{eq:ineg_PS_2}
\sum_{i=1}^3\int_{\R^n}\|\nabla \varphi_i\|^2\ge\sum_{i=1}^3\int_{\R^n}\|\nabla \varphi_i^{\sharp_2}\|^2.
\end{equation}
Coming back to~\eqref{eq:lambda3_beforeprojection}, and using~\eqref{eq:HLP_2}, \eqref{eq:ineg_phi_mi} and \eqref{eq:ineg_PS_2}, it follows that
\begin{equation} \label{eq:lambda3_beforeprojection-conclusion}
\lambda_3\ge\sum_{i=1}^3\int_{\R^n}\lp \frac{\mu^2}{2}\,\|\nabla \varphi_i^{\sharp_2}\|^2-r_i^{\sharp_2} \, (\varphi_i^{\sharp_2})^2 \rp +\delta - \delta \int_{\R^n} (\varphi_1^{\sharp_2} \, \varphi_2^{\sharp_2} + \varphi_2^{\sharp_2}\, \varphi_3^{\sharp_2} +\varphi_1^{\sharp_2} \, \varphi_3^{\sharp_2}).
\end{equation}
Using \eqref{eq:steiner_sumnorm} and $(\varphi_1^{\sharp_2},\varphi_2^{\sharp_2},\varphi_3^{\sharp_2})\in(H^1(\R^n)\cap L^2_w(\R^n))^3$, together with $(r_1^{\sharp_2},r_2^{\sharp_2},r_3^{\sharp_2})=(r_1,r_2,r_{\opt_3^{\sharp}})$, it follows that $(\varphi_1^{\sharp_2},\varphi_2^{\sharp_2},\varphi_3^{\sharp_2})$ is an admissible triplet in the Rayleigh quotient associated with the optima $(\opt_1,\opt_2,\opt_3^{\sharp})$. As a conclusion,~\eqref{eq:lambda3_beforeprojection-conclusion} yields $\lambda_3 \ge\lambda_3(\opt_1,\opt_2,\opt_3^{\sharp})$, which proves Proposition~\ref{prop:projection}.
\end{proof}

%%%%%%%%%%%%%%%%%%%%%%%%%%%%%%%%%%%%%%%%%%%%%%%%%%%%%
%%%%%%%%%%%%%%%%%%%%%%%%%%%%%%%%%%%%%%%%%%%%%%%%%%%%%

\bibliographystyle{plain}

%\bibliography{biblio_lionel,test}

\end{document}